\documentclass[10pt,a4paper]{amsart}

\usepackage{amsmath,amsthm,verbatim,amscd,amssymb,setspace,enumitem,hyperref}
\usepackage{exscale,color}
\usepackage[colorinlistoftodos,prependcaption]{todonotes}
\usepackage{graphicx}
\usepackage{tikz}
\usetikzlibrary{patterns}
\usepackage{enumitem}
\usepackage{mathrsfs}
\usepackage{tikz-cd}
\usepackage{caption}

\renewcommand{\epsilon}{\varepsilon}
\newcommand{\N}{\mathbb{N}}

\newcommand{\R}{\mathbb{R}}
\newcommand{\C}{\mathbb{C}}
\renewcommand{\P}{\mathbb{P}}

\newcounter{mtheorem}
\newtheorem{mtheorem}[mtheorem]{Theorem}

\newtheorem{mcorollary}[mtheorem]{Corollary}

\renewcommand{\P}{\mathbb{P}}
\newcommand{{\vol}}{\rm vol}

\newcommand{\Ric}{\operatorname{Ric}}

\newcommand{\Rm}{\operatorname{Rm}}

\providecommand{\sol}{\operatorname{sol}}

\def\Rm{\operatorname{Rm}}

\newtheoremstyle{fancy}{}{}{\itshape}{}{\textbf\bgroup}{.\egroup}{ }{}
\newtheoremstyle{fancy2}{}{}{\rm}{}{\textbf\bgroup}{.\egroup}{ }{}

\theoremstyle{fancy}
\newtheorem{theorem}{Theorem}[section]
\newtheorem{lemma}[theorem]{Lemma}
\newtheorem{corollary}[theorem]{Corollary}
\newtheorem{defn-prop}[theorem]{Definition \& Proposition}
\newtheorem{prop}[theorem]{Proposition}

\newtheorem{conjecture}[theorem]{Conjecture}

\theoremstyle{fancy2}
\newtheorem{definition}[theorem]{Definition}
\newtheorem{example}[theorem]{Example}
\newtheorem{remark}[theorem]{Remark}

\newtheorem{claim}[theorem]{Claim}

\theoremstyle{definition}
\newtheorem{assumption}{Assumption}[section]

\theoremstyle{definition}

\setlist{leftmargin=*}

\numberwithin{equation}{section}

\begin{document}
\title{K\"ahler-Ricci Tangent Flows in the Analytic Minimal Model Program}

\author{Longteng Chen}
\address[Longteng Chen]{Université Paris-Saclay, CNRS, Laboratoire de Mathématiques d'Orsay, 91405 Orsay, France }
\email{longteng.chen@universite-paris-saclay.fr}
\author{Max Hallgren}
\address[Max Hallgren]{University of Arizona, Department of Mathematics, 617 N. Santa Rita Ave., Tucson, AZ 85721}
\email{mhallgren@arizona.edu}
\author{Lucas Lavoyer}
\address[Lucas Lavoyer]{Mathematisches Institut, Universit\"at M\"unster, 48149 M\"unster, Germany}
\email{lucas.lavoyer@uni-muenster.de}

\begin{abstract} We describe certain finite-time singularities of the K\"ahler–Ricci flow arising in the analytic minimal model program. Assuming that convergence to an asymptotically conical K\"ahler–Ricci shrinker is realized by holomorphic maps, we prove that, in a fixed holomorphic gauge, the nearby flow is modeled on the shrinker at the level of K\"ahler potentials. Consequently, every noncollapsed K\"ahler–Ricci flow through singularities in complex dimension two is modeled on a shrinker–cone–expander transition, confirming a strong form of Song's conjectural picture. We also show analogous results in higher dimensions under the Calabi ansatz, and improve known results in the compact shrinker case. These give the first compact Ricci flows through conical singularities whose small-scale behavior is fully described.
\end{abstract}

\maketitle

\section{Introduction}

\subsection{Background and local models}

A central problem in complex geometry is to understand complex manifolds up to birational equivalence. The minimal model program addresses this problem by replacing a given variety with a simpler birational model through a sequence of birational surgeries. A typical surgery is encoded by a diagram
\begin{equation} 
\label{diagram:macro} \begin{tikzcd}
	M && {M'} \\
	& Y
	\arrow[dashed, from=1-1, to=1-3]
	\arrow["\pi", from=1-1, to=2-2]
	\arrow["{\pi'}"', from=1-3, to=2-2]
\end{tikzcd}
\end{equation}
where $M,M',Y$ are normal analytic varieties, $M,M'$ have mild singularities, and $\pi,\pi'$ are proper morphisms contracting appropriate families of holomorphic curves in $M,M'$. 

The analytic minimal model
program, initiated by Song and Tian \cite{songtian}, seeks a metric counterpart of this process. For simplicity, we restrict here to the case in which $M$ and $M'$ are smooth and projective. Given a K\"ahler metric $\omega$ on $M$ in the appropriate K\"ahler class, it is shown in \cite{songtian} that there are canonical solutions $(\omega_t)_{t\in [0,T)}$ and $(\omega_t')_{t\in (T,T')}$ of the K\"ahler-Ricci flow 
\begin{equation*}
    \partial_t \omega_t = -\Ric(\omega_t)
\end{equation*}
such that $\omega_0 = \omega$ and 
\begin{equation*}
    \lim_{t \nearrow T} \pi_{\ast}\omega_t = \omega_Y =\lim_{t\searrow T}\pi_{\ast}'\omega_t'
\end{equation*} 
in the sense of currents for some singular metric $\omega_Y$ on $Y$, with smooth convergence away from the exceptional loci of $\pi,\pi'$. The morphisms $\pi$ and $\pi'$ contract exactly those curves whose areas with respect to $\omega_t$ and $\omega_t'$, respectively, converge to $0$ as $t\to T$. It is conjectured in \cite{songtian} that the flow through singularities is continuous in the Gromov-Hausdorff sense, so that the Ricci flow provides a continuous geometric transformation corresponding to a sequence of discrete algebraic transformations. Song-Weinkove confirmed this continuity in the case $\dim_{\mathbb{C}}(M)=2$ \cite{SW2}; in that setting, $M'=Y$ and $\pi$ blows down finitely many disjoint $(-1)$-curves in $M$. In higher dimensions, this continuity is still conjectural, though it is known in certain situations where $\omega$ is invariant under a large symmetry group \cite{songyuan,JianSong}.  
For further background on the K\"ahler--Ricci flow and its relation to
the analytic minimal model program, we refer to
\cite{SongTian-SurfPositKodDim,SongTian-CanonicalMeasures,SW1,SW2,
SongWeinkove,TosattiNotes,songyuan,jiansongtian,
HallgrenJianSongTian} and the references therein.

The aforementioned results do not, however, provide a detailed description of the small-scale metric geometry of the surgery. This is an important difference
from Perelman's treatment of three-dimensional Ricci flow \cite{Perelman2003}, where a good understanding of the
geometry of high-curvature regions is essential for the surgery
construction. Consequently, several fundamental questions
remain: which blow-up models can occur near the contracted locus, which tangent cones arise at the singular time slice, and how the small-scale behaviors of the incoming and outgoing flows are related near the surgery. 

To state the picture conjectured in \cite{JianSong}, we consider a K\"ahler-Ricci flow through singularities corresponding to the diagram \eqref{diagram:macro}. For simplicity, we shift the singular time to $t=0$, and continue to assume that $M$ and $M'$
 are smooth, though a more general version is formulated in \cite[Section 6]{JianSong}.
 
\begin{conjecture}[\cite{JianSong}] 
\label{conjecture:song}
The K\"ahler-Ricci flow through a smooth flip or divisorial contraction satisfies the Type I curvature bounds
\begin{equation}
\label{eq:TypeIcurvatureconjecturein} \limsup_{t\nearrow 0}|t| \sup_M |\Rm_{g_t
}|_{g_t} <\infty, 
\end{equation}
\begin{equation} \label{eq:TypeIcurvatureconjectureout}
 \limsup_{t\searrow 0} t\sup_{M'} |\Rm_{g_t'}|_{g_t'}<\infty,
\end{equation}
and for any choice of $y_0 \in Y$, $x_0 \in \pi^{-1}(y_0)$, $x_0'\in (\pi')^{-1}(y_0)$, the pointed rescaled flows as $\lambda\to\infty$ satisfy
\begin{equation} \label{eq:tangentflowin}
    (M,(\lambda g_{\lambda^{-1}t})_{t\in [-\lambda T,0)},x_0) \to (M_{\sol},(g_{\sol,t})_{t\in (-\infty,0)},x_{\infty}), 
\end{equation}
\begin{equation} \label{eq:tangentflowout}
    \qquad (M',(\lambda g_{\lambda^{-1}t}')_{t\in (0,\lambda (T'-T)]},x_0') \to (M_{\exp},(g_{\exp,t})_{t\in (0,\infty)},x_{\infty}')
\end{equation}
in the pointed Cheeger-Gromov sense, where $(M_{\sol},g_{\sol,t})$ and $(M_{\exp},g_{\exp,t})$ are the self-similar flows generated by asymptotically conical shrinking and expanding K\"ahler-Ricci solitons, respectively. Moreover, $(M_{\sol},g_{\sol})$ and $(M_{\exp},g_{\exp})$ have a common asymptotic cone $(\mathcal{C},g_{\mathcal{C}})$, equal to the unique tangent cone of $(Y,g_0)$ at $y_0$. In particular, the ancient flow $(g_{\sol,t})_{t<0}$ and the immortal flow $(g_{\exp,t})_{t>0}$  form a self-similar K\"ahler-Ricci flow through singularities corresponding to the diagram
\begin{equation}
\label{diagram:micro}
\begin{tikzcd}
	{M_{\operatorname{sol}}} && {M_{\exp}} \\
	& {\mathcal{C}}
	\arrow[dashed, from=1-1, to=1-3]
	\arrow["{\pi_{\operatorname{sol}}}", from=1-1, to=2-2]
	\arrow["{\pi_{\exp}}"', from=1-3, to=2-2]
\end{tikzcd},
\end{equation}
and this serves as the infinitesimal model of the birational surgery \eqref{diagram:macro} at $y_0$. 
\end{conjecture}
In this paper, we establish a strong form of Conjecture \ref{conjecture:song} in complex dimension two (Corollary \ref{cor:2dim}), for metric flips satisfying a Calabi ansatz (Corollary \ref{cor:calabi}), and for canonical surgical contractions to smooth orbifolds satisfying a Calabi ansatz (Corollary \ref{cor:orb}). In each of these settings, it is already known that \eqref{eq:TypeIcurvatureconjecturein} and \eqref{eq:tangentflowin} hold \cite{jiansongtian,CHM}. Our Theorem \ref{thm:main1} upgrades the infinitesimal description \eqref{eq:tangentflowin} of the singularity to a local description: we show that on a fixed neighborhood of the
exceptional set of $\pi$, the original K\"ahler--Ricci flow can be written as a
$\partial\overline{\partial}$-exact perturbation of its tangent flow, with
quantitative decay estimates for the potential function and its derivatives. In particular, this shows $(Y,g_0)$ is a compact analytic space with isolated conical singularities in the sense of \cite[Definition 1.1]{CHL}, thereby ensuring a strong local description of the flow coming out of the singularity as well, and verifying Conjecture \ref{conjecture:song} in these settings.

\subsection{Statements of main theorems}

Let $(M^n,J)$ be an $n$-dimensional compact K\"ahler manifold, and let $(\widetilde{g}_t)_{t\in [-T,0)}$ be a K\"ahler-Ricci flow on $M$ with first singular time $t=0$. In \cite[Section 2.6]{Bam3}, it was shown that there is a naturally defined metric space $(M_0,d_0)$ which serves as the $0$-time slice of the flow. We refer the reader to Section \ref{sec-basic setup for polarized KRF} for the precise definition. By \cite{GPSS,LiZhang}, $M_0$ is isometric to the unique Gromov-Hausdorff limit of $(M,d_{\widetilde{g}_t})$ as $t\nearrow 0$. By \cite[Theorem 1.4]{hacon}, there exists a surjective holomorphic map $\pi:M\to Y$ with connected fibers to a normal K\"ahler space $(Y,\omega_Y)$ such that $\alpha:=\pi^{\ast}[\omega_Y]=[\widetilde{\omega}_{-T}]-Tc_1(M)$. Then by \cite[Theorem 1.5]{collinstosatti}, the K\"ahler-Ricci flow $(\widetilde{g}_t)_{t\in [-T,0))}$ develops singularities along the non-K\"ahler locus $E_{nK}(\alpha)$ of the class $\alpha$, so $\widetilde{g}_t|_{M\setminus E_{nK}(\alpha)}$ converges locally smoothly to a K\"ahler metric $\widetilde{g}_0$, and $(M\setminus E_{nK}(\alpha),d_{\widetilde{g}_0})$ is locally isometric to a (possibly empty) open subset of $M_0$. 

\begin{remark} \label{rem:pihat} By extending arguments of \cite{CHM}, we will show in Lemma \ref{limit on the base} that there exists a canonical continuous and surjective map $\widehat{\pi}:M_0 \to Y$ making the following diagram commute:
\[\begin{tikzcd}
	{M\setminus E_{nK}(\alpha)} && M \\
	\\
	{M_0} && Y
	\arrow[hook, from=1-1, to=1-3]
	\arrow[hook', from=1-1, to=3-1]
	\arrow["\pi", from=1-3, to=3-3]
	\arrow["{\widehat{\pi}}"', from=3-1, to=3-3]
\end{tikzcd}\]
where we identify $M\setminus E_{nK}(\alpha)$ with a subset of $M_0$ by $x\mapsto \lim_{t_0\nearrow 0}(\nu_{x,t_0;t})_{t<t_0}$ as in \cite[Section 2.6]{Bam3}. 
\end{remark}

For any $\widetilde{\nu}=(\widetilde{\nu}_t)_{t\in [-T,0)} \in M_0$, and any $t_i \nearrow 0$, the Ricci flow compactness theory of \cite{Bam3} gives a subsequence whose parabolic dilations 
\begin{equation*}
    (M,(|t_i|^{-1}\widetilde{g}_{|t_i|t})_{t\in [-|t_i|^{-1}T,0)},(\widetilde{\nu}_{|t_i|t})_{t\in [-|t_i|^{-1}T,0)})
\end{equation*}
converge in an appropriate sense to the flow induced by a (possibly singular) gradient K\"ahler-Ricci shrinker. Such a limit is called a tangent flow of $(M,J,(\widetilde{g}_t)_{t\in [-T,0)},\widetilde{\nu})$. We restrict to the case in which the tangent shrinker $(M_{\sol},J_{\sol},g_{\sol},f_{\sol})$ is smooth. Writing $d\widetilde{\nu}_t = (2\pi|t|)^{-n}e^{-\widetilde{f}_t}d\widetilde{g}_t$, there is then an open exhaustion $(U_i)_{i\in \mathbb{N}}$ of $M_{\sol}$, and open smooth embeddings $\psi_i:U_i \to M$ such that 
\begin{equation} \label{eq:diffeoconvergence}
    |t_i|^{-1}\psi_i^{\ast}\widetilde{g}_{|t_i|t}\to g_{\sol,t}, \qquad \psi_i^{\ast}J \to J_{\sol}, \qquad \psi_i^{\ast}\widetilde{f}_{|t_i|t} \to f_{\sol,t}+W
\end{equation}
in $C_{\operatorname{loc}}^{\infty}(M_{\sol}\times (-\infty,0))$ as $i\to \infty$, where $g_{\sol,t}$, $f_{\sol,t}$, and $W$ are defined in Section \ref{subsection:ACshrinkers}.

\begin{definition}\label{definition: holomorphic approximation}
We say $(M,J,(g_t)_{t\in [-T,0)},(\nu_t)_{t\in [-T,0)})$ can be \emph{holomorphically approximated} by the shrinker $(M_{\sol},J_{\sol},g_{\sol},f_{\sol})$ if there exist embeddings $\psi_i$ as in \eqref{eq:diffeoconvergence} that additionally satisfy $\psi_i^{\ast}J = J_{\sol}$. 
\end{definition}

\begin{definition} We say a gradient K\"ahler-Ricci shrinker $(M_{\sol},J_{\sol},g_{\sol},f_{\sol})$ is \emph{asymptotically conical} (AC) if it is \emph{non-flat}, non-compact, complete and satisfies
\begin{equation} \label{eq:quadraticcurvaturedecay}
    \sup_{M_{\sol}} |\Rm_{g_{\sol}}|_{g_{\sol}}f_{\sol}<\infty.
\end{equation}
\end{definition}

Our first theorem shows that holomorphic approximation by an AC shrinker forces uniqueness of the tangent flow and yields convergence, in a fixed holomorphic gauge, at the level of K\"ahler potentials. In the following we adhere to the notation of Remark \ref{rem:pihat} and Theorem \ref{thm:AC}; in particular, $\pi:M\to Y$ is a holomorphic map with connected fibers mapping onto a K\"ahler space $(Y,\omega_Y)$ such that $[\widetilde{\omega}_{-T}]-Tc_1(M) = \pi^{\ast}[\omega_Y]$, and $\pi_{\sol}:M_{\sol}\to \mathcal{C}$ is the Remmert reduction map of the AC shrinker $M_{\sol}$ to its asymptotic cone $(\mathcal{C},g_{\mathcal{C}})$. The vertex and radius function of the cone are denoted by $o$ and $r$, respectively. For any point in $M\setminus E_{nK}(\alpha),$ the convergence below will trivially be satisfied if one considers the Gaussian shrinker as the model, so we only focus on the behavior of the flow near its singular set $E_{nK}(\alpha).$

\begin{mtheorem} \label{thm:main1} Assume $(M,J,(\widetilde{g}_t)_{t\in [-T,0)},(\widetilde{\nu}_t)_{t\in [-T,0)})$ can be holomorphically approximated by an AC K\"ahler-Ricci shrinker $(M_{\sol},J_{\sol},g_{\sol},f_{\sol})$.

Then there are neighborhoods $\mathcal{U} \subseteq \{r<\frac{1}{4}\}$ of $o$ and $\mathcal{V}\subseteq Y$ of $y_0:=\widehat{\pi}(\widetilde{\nu})$, along with biholomorphisms $F:\pi_{\sol}^{-1}(\mathcal{U})\to \pi^{-1}(\mathcal{V})$ and $\widehat{F}:\mathcal{U} \to \mathcal{V}$ and functions 
\begin{equation*} \widetilde{\varphi} \in C^{\infty}\left(\pi_{\sol}^{-1}(\mathcal{U})\times [-\min\{T,\frac{1}{4}\},0)\right)\cap C^0\left(\pi_{\sol}^{-1}(\mathcal{U}) \times [-\min\{T,\frac{1}{4}\},0]\right), \qquad \widehat{\varphi} \in C^{\infty}(\mathcal{U}\setminus \{o\})\cap C^0(\mathcal{U})\end{equation*} such that the following hold:
\begin{enumerate}
    \item\label{mainthm:1} $\widehat{F}(o)=y_0$, $F(\pi_{\sol}^{-1}(o))=\pi^{-1}(y_0)$, and the following diagram commutes:
    \[\begin{tikzcd}
	{M_{\operatorname{sol}}} & {\pi_{\operatorname{sol}}^{-1}(\mathcal{U})} && {\pi^{-1}(\mathcal{V})} & M \\
	\\
	{\mathcal{C}} & {\mathcal{U}} && {\mathcal{V}} & Y
	\arrow[hook', from=1-2, to=1-1]
	\arrow["F", from=1-2, to=1-4]
	\arrow["{\pi_{\operatorname{sol}}}"', from=1-2, to=3-2]
	\arrow[hook, from=1-4, to=1-5]
	\arrow["\pi", from=1-4, to=3-4]
	\arrow[hook', from=3-2, to=3-1]
	\arrow["{\widehat{F}}"', from=3-2, to=3-4]
	\arrow[hook, from=3-4, to=3-5]
\end{tikzcd};\]

    \item \label{mainthm:2} $F^{\ast}\widetilde{\omega}_t = \omega_{\sol,t}+\sqrt{-1}\partial \overline{\partial} \widetilde{\varphi}_t$ on $\pi^{-1}_{\sol}(\mathcal{U})$ for all $t\in [-\min\{T,\frac{1}{4}\},0)$;

    \item \label{mainthm:3} $\widehat{F}^{\ast}(\pi_{\ast}\widetilde{\omega}_0) = \omega_{\mathcal{C}}+\sqrt{-1}\partial \overline{\partial}\widehat{\varphi}$ and $\pi_{\sol}^{\ast}\widehat{\varphi}=\widetilde{\varphi}_0$ on $\mathcal{U}\setminus \{o\}$, where $\widetilde{\omega}_0:=\lim_{t\nearrow 0}\widetilde{\omega}_t$ in the sense of currents;

    \item \label{mainthm:4} For any $k\in \mathbb{N}$, there exists a constant $C_k \in [1,\infty)$ such that
    \begin{equation*} |(\nabla^{g_{\sol,t}})^k\widetilde{\varphi}_t|_{g_{\sol,t}}(x) \leq C_k \left( r(x)+\sqrt{|t|}  \right)^{2-k} \left| \log \left( r(x)+\sqrt{|t|}\right) \right|^{-1}
    \end{equation*}
    for all $(x,t)\in \pi_{\sol}^{-1}(\mathcal{U})\times [-\min\{T,\frac{1}{4}\},0)$,

    \item \label{mainthm:5} For any $k\in \mathbb{N}$, there exists $C_k \in [1,\infty)$ such that 
    \begin{equation*} |(\nabla^{g_{\mathcal{C}}})^k\widehat\varphi|_{g_{\mathcal{C}}}(x) \leq C_k r^{2-k}(x) \left| \log r(x)\right|^{-1}
    \end{equation*}
    for all $x\in \mathcal{U} \setminus \{o\}$.
\end{enumerate}
\end{mtheorem}

\begin{remark} We note that \ref{mainthm:1}, \ref{mainthm:2} and \ref{mainthm:4} immediately imply the smooth Cheeger-Gromov convergence \eqref{eq:tangentflowin}. Similarly, \ref{mainthm:1}, \ref{mainthm:3} and \ref{mainthm:5} imply the pointed Gromov-Hausdorff convergence $(Y,\lambda d_{Y},y_0)\to (\mathcal{C},d_{g_{\mathcal{C}}},o)$ as $\lambda \to \infty$, where $d_Y$ is the length metric of $(\mathcal{V}\setminus \{y_0\},\pi_{\ast}\widetilde{\omega}_0)$ completed near $y_0$. 
\end{remark}

\begin{remark} 
    Our proof of Theorem \ref{thm:main1} also gives fast convergence in a smaller region of the flow (cf. \eqref{eq:metricclosenessinshrinkerregion}):
    \begin{equation*}
        \sup_{t\in [-\min\{T,\frac{1}{4}\},0)} \sup_{\{r \leq \beta \sqrt{|t|\log\frac{1}{|t|}}\} } |(\nabla^{g_{\sol,t}})^k \widetilde{\varphi}_t|_{g_{\sol,t}} \leq C_k |t|^{1+\beta-\frac{k}{2}}
    \end{equation*}
for some $\beta>0$. With respect to the modified flow \eqref{eq:intro:modifiedKRF}, defined for all large $s\geq 0$, this corresponds to exponential decay of the relative K\"ahler potential on a region of the form
\begin{equation*}
    \bigcup_{s \geq s_0} \left( \{f_{\sol} \leq \beta s\}\times \{s\} \right).
\end{equation*}
\end{remark}

\subsubsection{K\"ahler-Ricci flow with surgery on complex surfaces}

We now apply Theorem \ref{thm:main1} to birational surgeries appearing in the analytic minimal model program \cite{songtian}. Combined with the results of \cite{CHM,CHL}, this yields a detailed description of any K\"ahler-Ricci flow through surgery on K\"ahler surfaces, and confirms Conjecture \ref{conjecture:song} in complex dimension two. 

Let $(M^2,(g_t)_{t\in [-T,0)})$ be a compact K\"ahler-Ricci flow of complex dimension two that develops a singularity at time $t=0$ and satisfies
    \begin{equation} \label{eq:noncollapsing}
        \lim_{t \nearrow 0} \operatorname{Vol}_{g_t}(M)>0.
    \end{equation}
By \cite{SW1}, there is a blowdown map $\pi:M\to M'$ along a finite disjoint union of $(-1)$-curves $E_1,...,E_N$. The flow performs a canonical surgical contraction with respect to $\pi$ and yields a canonical K\"ahler-Ricci flow $(g_t')_{t\in (0,T']}$ on $M'$ and a metric $d$ on $M'$ such that $T'>0$ and 
    \begin{equation*}
        \lim_{t \nearrow 0} (M,d_{g_t}) = (M',d) = \lim_{t\searrow 0} (M',d_{g_t'})
    \end{equation*}
    in the Gromov-Hausdorff sense, and such that $(M',d)$ is the metric completion of $(M'\setminus \cup_{i=1}^N \pi(E_i),\pi_{\ast}\omega_0)$. We let $\pi_{\operatorname{FIK}}:M_{\operatorname{FIK}}=\mathcal{O}_{\P^1}(-1)
\to \mathbb{C}^2$ be the blowup of $\mathbb{C}^2$ at the origin, and recall that there exists a unique (up to holomorphic isometry) K\"ahler-Ricci shrinker $g_{\operatorname{FIK}}$ on $M_{\operatorname{FIK}}$ \cite{FIK,ConlonDeruelleSun}. The asymptotic cone of this soliton is the K\"ahler cone metric $\omega_{\mathcal{C}} := \frac{\sqrt{-1}}{\sqrt{2}}\partial \overline{\partial}(|z|^{\sqrt{2}})$ on $\mathbb{C}^2$. Moreover, there exists a unique $U(2)$-invariant gradient K\"ahler-Ricci expander $g_{\operatorname{Cao}}$ on $\mathbb{C}^2$ asymptotic to $(\mathbb{C}^2,\omega_{\mathcal{C}})$ at infinity, constructed in \cite{CaoExpander}.  

    By \cite[Theorem A]{CHM}, the flow $(g_t)_{t\in [-T,0)}$ satisfies the Type I curvature bound \eqref{eq:TypeIcurvatureconjecturein} as $t\nearrow 0$.
    Consequently, \cite{Naber4solitons,EndersMullerTopping,MM2015} implies that every tangent flow at time $t=0$ is a smooth K\"ahler-Ricci soliton with bounded curvature. By \eqref{eq:TypeIcurvatureconjecturein} and \cite{CifarelliConlonDeruelleKsurface}, every nontrivial tangent flow is in fact holomorphically isometric to the FIK shrinker $M_{\operatorname{FIK}}$. The following corollary gives a stronger, fixed-gauge sense in which the flow in and out of the singularity at $t=0$ is modeled on a specific shrinker-expander transition. 

\begin{mcorollary} \label{cor:2dim}

    \begin{enumerate}
        \item \label{cor:2dim1} The flow $(M',(g_t')_{t\in (0,T']})$ satisfies the Type I curvature bound \eqref{eq:TypeIcurvatureconjectureout}.

        \item \label{cor:2dim2} For each contracted $(-1)$-curve $E_i \subseteq M$, there exists an open holomorphic embedding $F_i:\mathcal{U} \to M$, where $\mathcal{U} \subseteq M_{\operatorname{FIK}}$ is an open set containing the unique $(-1)$-curve $E_{\operatorname{FIK}}$ of $M_{\operatorname{FIK}}$, such that $F_i(E_{\operatorname{FIK}})=E_i$ and $F_i^{\ast}\omega_t = \omega_{\operatorname{FIK},t}+\sqrt{-1}\partial \overline{\partial} \varphi_{i,t}$ for some $\varphi_i \in C^{\infty}(\mathcal{U} \times [-T,0)) \cap C^0(\mathcal{U}\times [-T,0])$, where for each $k\in \mathbb{N}$,
        \begin{equation*}
            |(\nabla^{g_{\operatorname{FIK},t}})^k \varphi_{i,t}|_{g_{\operatorname{FIK},t}}\leq o \left( (r(x)+\sqrt{|t|})^{2-k} \right)
        \end{equation*}
        for $(x,t)\in \mathcal{U} \times [-T,0)$ as $r(x)+\sqrt{|t|} \to 0$, where $r$ is the radius function of the cone metric $\omega_{\operatorname{FIK},0}=\omega_\mathcal{C}$.

        \item \label{cor:2dim3} $(M',d)$ is a compact analytic space with isolated conical singularities in the sense of \cite[Definition 1.1]{CHL}, whose cone singularities $y_i = \pi(E_i)$ are each modeled on $(\mathbb{C}^2,\omega_{\mathcal{C}})$. 
        
        \item \label{cor:2dim4} For each $i\in \{1,...,N\}$, there exists an open holomorphic embedding $F_i':\pi_{\operatorname{FIK}}(\mathcal{U}) \subseteq \C^2\to M'$ such that $\pi \circ F_i = F_i' \circ \pi_{\operatorname{FIK}}$ and $(F_i')^{\ast}\omega_t' = \omega_{\operatorname{Cao},t}+\sqrt{-1}\partial \overline{\partial} \psi_{i,t}$ for some 
        \newline $\psi_i \in C^{\infty}(\pi_{\operatorname{FIK}}(\mathcal{U})\times (0,T'])\cap C^0(\pi_{\operatorname{FIK}}(\mathcal{U})\times [0,T'])$ with $\pi_{\operatorname{FIK}}^{\ast} \psi_{i,0} = \varphi_{i,0}$ such that for each $k\in \mathbb{N}$,
        \begin{equation*}
            |(\nabla^{g_{\operatorname{Cao},t}})^k \psi_{i,t}|_{g_{\operatorname{Cao},t}}\leq o\left( (r(x)+\sqrt{t})^{2-k} \right)
        \end{equation*}
        for $(x,t)\in \pi_{\operatorname{FIK}}(\mathcal{U})\times (0,T']$ as $r(x)+\sqrt{t} \to 0$.
    \end{enumerate}
\end{mcorollary}

\begin{remark}
 By \cite{CHL,ZhangXu}, a similar description holds for K\"ahler-Ricci flows $(M^n,(\omega_t)_{t\in [-T,0)})$ in any dimension, under the following assumptions:
\begin{enumerate}
    \item the Type I curvature bound \eqref{eq:TypeIcurvatureconjecturein} holds,
    \item $\pi:M\to M'$ is the blowup of $M'$ at finitely many points, and $\pi^{\ast}[\omega_{M'}]=[\omega_0]$ for some K\"ahler metric $\omega_{M'}$ on $M'$.
\end{enumerate}
\end{remark}

\subsubsection{K\"ahler-Ricci flow under Calabi ansatz}

We also obtain a precise description of higher-dimensional K\"ahler-Ricci flows with Calabi ansatz that undergo divisorial contractions or flips by combining Theorem \ref{thm:main1} with \cite{CHL} and \cite[Corollary 2.8]{jiansongtian}. 

Fix integers $n > m \geq 0$. Following the setup of \cite{songyuan}, we consider the total space $M$ of the holomorphic $\mathbb{P}^{m+1}$-bundle $\sigma:\mathbb{P}(\mathcal{O}_{\mathbb{P}^n}\oplus \mathcal{O}_{\mathbb{P}^n}(-1)^{\oplus(m+1)})\to \mathbb{P}^n$ and the total space $M'$ of $\sigma':\mathbb{P}(\mathcal{O}_{\mathbb{P}^m}\oplus \mathcal{O}_{\mathbb{P}^m}(-1)^{\oplus(n+1)})\to \mathbb{P}^m$. Then there is a normal projective variety $Y$ along with birational morphisms $\pi:M\to Y$ and $\pi':M'\to Y$ which contract exactly the zero sections $E,E'$ of $\sigma,\sigma'$, respectively. Equip $M$ with the natural holomorphic $U(n+1)\times U(m+1)$-action for which $U(n+1)\times \{1\}$ lifts the standard action on $\mathbb{P}^n$ and $\{1\}\times U(m+1)$ preserves each fiber. Suppose $(M^{m+n+1},(g_t)_{t\in [-T,0)})$ is a K\"ahler-Ricci flow whose initial metric $g_{-T}$ is $U(n+1)\times U(m+1)$-invariant and which satisfies \eqref{eq:noncollapsing}. By \cite{JianSong}, the K\"ahler-Ricci flow performs a surgical metric flip at time $t=0$, hence there is a canonical K\"ahler-Ricci flow $(g_t')_{t\in (0,T']}$ on $M'$ and metric $d_Y$ on $Y$ such that $T'>0$ and
\begin{equation*}
    \lim_{t\nearrow 0} (M,d_{g_t}) = (Y,d_Y) = \lim_{t\searrow 0} (M',d_{g_t'})
\end{equation*}
in the Gromov-Hausdorff sense. By \cite{lisoliton}, there exists a unique (up to holomorphic isometries) $U(n+1)\times U(m+1)$-invariant AC K\"ahler-Ricci shrinker $(M_{\sol},J_{\sol},g_{\sol})$ whose underlying complex manifold $(M_{\sol},J_{\sol})$ is the total space of $\mathcal{O}_{\mathbb{P}^n}(-1)^{\oplus(m+1)}$ and which is asymptotic to a K\"ahler cone metric $(\mathcal{C},g_{\mathcal{C}})$ on the affine cone $\mathcal{C}$ over $\mathbb{P}^n \times \mathbb{P}^m$. Moreover, there is a unique (up to holomorphic isometries) $U(n+1)\times U(m+1)$-invariant AC K\"ahler-Ricci expander $(M_{\exp},J_{\exp},g_{\exp})$ whose underlying complex manifold $(M_{\exp},J_{\exp})$ is the total space of $\mathcal{O}_{\mathbb{P}^m}(-1)^{\oplus(n+1)}$ and which is asymptotic to $(\mathcal{C},g_{\mathcal{C}})$. Let $E_{\sol},E_{\exp}$ denote the zero sections of $\mathcal{O}_{\mathbb{P}^n}(-1)^{\oplus(m+1)}\to \mathbb{P}^n$ and $\mathcal{O}_{\mathbb{P}^m}(-1)^{\oplus(n+1)}\to \mathbb{P}^m$, respectively, and let $y_0 \in Y$ denote the unique singular point of $Y$.

By \cite[Theorem 2.7]{jiansongtian}, the pre-singularity flow satisfies the Type I bound \eqref{eq:TypeIcurvatureconjecturein}, and the unique nontrivial tangent flow is the shrinker $M_{\sol}$. As in Corollary \ref{cor:2dim}, the next result upgrades this tangent-flow statement to a fixed-gauge description modeled on $M_{\sol} \to \mathcal{C} \leftarrow M_{\exp}$. 

\begin{mcorollary} \label{cor:calabi} The flow satisfies a Type I curvature bound \eqref{eq:TypeIcurvatureconjectureout} emerging from the time $t=0$ singularity. Moreover, for some neighborhood $\mathcal{U} \subseteq \mathcal{C}$ of $o$, there are open holomorphic embeddings $F:\pi_{\sol}^{-1}(\mathcal{U}) \to M$, $F_0:\mathcal{U} \to Y$, and $F':\pi_{\exp}^{-1}(\mathcal{U})\to M'$ such that $F_0(o)=y_0$ and the following diagram commutes:
\begin{equation*} \begin{tikzcd}
	M && Y && {M'} \\
	\\
	{\pi_{\sol}^{-1}(\mathcal{U})} && {\mathcal{U}} && {\pi_{\exp}^{-1}(\mathcal{U})}
	\arrow["\pi", from=1-1, to=1-3]
	\arrow["{\pi'}"', from=1-5, to=1-3]
	\arrow["F", hook, from=3-1, to=1-1]
	\arrow["{\pi_{\operatorname{sol}}}"', from=3-1, to=3-3]
	\arrow["{F_0}", hook, from=3-3, to=1-3]
	\arrow["{F'}", hook, from=3-5, to=1-5]
	\arrow["{\pi_{\exp}}", from=3-5, to=3-3]
\end{tikzcd}.\end{equation*}
 Furthermore, there are 
 \begin{equation*}\varphi \in C^{\infty}(\pi_{\sol}^{-1}(\mathcal{U})\times [-T,0))\cap C^0(\pi_{\sol}^{-1}(\mathcal{U}) \times [-T,0]), \quad \psi \in C^{\infty}(\pi_{\exp}^{-1}(\mathcal{U})\times (0,T'])\cap C^0(\pi_{\exp}^{-1}(\mathcal{U})\times [0,T']),\end{equation*}
 and $\widehat{\varphi}_0 \in C^{\infty}(\mathcal{U}\setminus \{o\})\cap C^0(\mathcal{U})$ such that the following hold:
\begin{enumerate}

        \item \label{cor:calabi1} $F^{\ast}\omega_t = \omega_{\sol,t}+\sqrt{-1}\partial \overline{\partial} \varphi_{t}$, where for each $k\in \mathbb{N}$,
        \begin{equation*}
            |(\nabla^{g_{\sol,t}})^k \varphi_{t}|_{g_{\sol,t}}\leq o \left( (r(x)+\sqrt{|t|})^{2-k} \right)
        \end{equation*}
        for $(x,t)\in \pi_{\sol}^{-1}(\mathcal{U}) \times [-T,0)$ as $r(x)+\sqrt{|t|} \to 0$, where $r$ is the radius function of the cone metric $\omega_{\sol,0}$;

        \item \label{cor:calabi2} $(Y,d_Y)$ is a compact analytic space with isolated conical singularities in the sense of \cite[Definition 1.1]{CHL}, with a unique isolated cone singularity $y_0$ modeled on $(\mathcal{C},\omega_{\mathcal{C}})$,
        
        \item \label{cor:calabi3} $(F')^{\ast}\omega_t = \omega_{\exp,t}+\sqrt{-1}\partial \overline{\partial} \psi_t$, where for each $k\in \mathbb{N}$,
        \begin{equation*}
            |(\nabla^{g_{\exp,t}})^k \psi_t|_{g_{\exp,t}}\leq o \left( (r(x)+\sqrt{t})^{2-k} \right)
        \end{equation*}
        for $(x,t)\in \pi_{\exp}^{-1}(\mathcal{U})\times (0,T']$ as $r(x)+\sqrt{t} \to 0$,

        \item \label{cor:calabi4} $\pi_{\sol}^{\ast}\widehat{\varphi}_0=\varphi_0$ and $\pi_{\exp}^{\ast}\widehat{\varphi}_0=\psi_0$.
    \end{enumerate}
\end{mcorollary}

We now use a minor extension of \cite{CHL} to treat the Calabi-ansatz flow of \cite{KRFonHirzebruchSurf,SW2}, whose post-singular continuation is an orbifold K\"ahler-Ricci flow. The outgoing model is an orbifold K\"ahler-Ricci expander on $\mathbb{C}^n/\mathbb{Z}_k$. 

Fix $n\geq 2$ and $1\leq k\leq n-1$, and follow the setup of \cite{KRFonHirzebruchSurf}. Let $M:=\mathbb{P}(\mathcal{O}_{\mathbb{P}^{n-1}}\oplus \mathcal{O}_{\mathbb{P}^{n-1}}(-k))$. We equip $M$ with the natural holomorphic action of $U(n)/\mathbb{Z}_k$ described in \cite[Section 2.2]{KRFonHirzebruchSurf}. Suppose $(M,(g_t)_{t\in [-T,0)})$ is a K\"ahler-Ricci flow whose initial metric is $U(n)/\mathbb{Z}_k$-invariant and which satisfies \eqref{eq:noncollapsing}. By \cite[Theorem 1.2]{SW2}, there is a K\"ahler orbifold $M'$ and a holomorphic map $\pi:M\to M'$ contracting a divisor $\mathbb{P}^{n-1} \cong E \subseteq M$ to the unique orbifold point of $M'$, and the flow $(g_t)_{t\in [-T,0)}$ can be canonically continued as an orbifold K\"ahler-Ricci flow $(g_t')_{t\in (0,T')}$ on $M'$. By \cite[Corollary 2.8]{jiansongtian}, the flow satisfies the Type I bound \eqref{eq:TypeIcurvatureconjecturein}, and the unique tangent flow based along $E$ is the AC K\"ahler-Ricci shrinker $g_{\sol}$ on the total space of $M_{\sol}:=\mathcal{O}_{\mathbb{P}^{n-1}}(-k)$ constructed in \cite{FIK}. The asymptotic cone $(\mathcal{C},\omega_{\mathcal{C}})$ of $M_{\sol}$ is a cone metric on $\mathbb{C}^n/\mathbb{Z}_k$, where $\mathbb{Z}_k$ acts with weight $(1,...,1)$, induced by a metric on $\mathbb{C}^n$ of the form $\sqrt{-1}\partial \overline{\partial}(\frac{1}{p_{n,k}}|z|^{p_{n,k}})$, where $p_{n,k}>0$. Let $\pi_{\sol}:M_{\sol}\to \mathcal{C}$ be the contraction map, with exceptional set the zero section $\mathbb{P}^{n-1}\cong E_{\sol}\subseteq M_{\sol}$. 

On the other hand, \cite[Proposition 5.1]{CaoExpander} gives a unique $U(n)/\mathbb{Z}_k$-invariant orbifold K\"ahler-Ricci expander $(M_{\exp},g_{\exp})$ on $\mathbb{C}^n/\mathbb{Z}_k$ asymptotic to $(\mathcal{C},\omega_{\mathcal{C}})$. By combining Theorem \ref{thm:main1} with \cite{CHL}, we conclude that this expander in fact models the flow emerging from the singularity at time $t=0$. 

\begin{mcorollary} \label{cor:orb} 
    The flow satisfies a Type I curvature bound \eqref{eq:TypeIcurvatureconjectureout} emerging from the time $t=0$ singularity. Moreover, for some neighborhood $\mathcal{U} \subseteq \mathcal{C}$ of $o$, there are open holomorphic embeddings $F:\pi_{\sol}^{-1}(\mathcal{U})\to M$ and $F_0:\mathcal{U}\to M'$, and there are $\varphi\in C^{\infty}(\pi^{-1}(\mathcal{U})\times [-T,0))\cap C^0(\pi^{-1}(\mathcal{U})\times [-T,0])$ and $\psi \in C^{\infty}(\mathcal{U}\times (0,T'])\cap C^0(\mathcal{U} \times [0,T'])$ such that the assertions \ref{cor:calabi1}-\ref{cor:calabi4} of Corollary \ref{cor:calabi} hold with $Y=M'$, $M_{\exp}=\mathcal{C}$, $F_0=F'$,  $\pi'=\operatorname{id}_{M'}$, $\pi_{\exp} = \operatorname{id}_{\mathcal{U}}$, and $\widehat{\varphi}_0=\psi_0$.

\end{mcorollary}

\subsubsection{General conjecture and an improvement of the compact case}
Based on ideas of \cite{SunZhang}, we now conjecture a necessary and sufficient condition for a compact K\"ahler-Ricci flow $(M,J,(\widetilde{g}_t)_{t\in [-T,0)})$ to be holomorphically approximated by a K\"ahler-Ricci shrinker. We suppose the flow develops a finite-time singularity at time $t=0$. By \cite{hacon}, there is a bimeromorphic Fano fibration to a normal K\"ahler space $(Y,\omega_Y)$ such that $[\widetilde{\omega}_0]=\pi^{\ast}[\omega_Y]$.  

\begin{conjecture} \label{conj:main1} Given $\widetilde{\nu} \in M_0$, $(M,J,(\widetilde{g}_t)_{t\in [-T,0)},\widetilde{\nu})$ can be holomorphically approximated by a K\"ahler-Ricci shrinker if and only if there exists a K\"ahler-Ricci shrinker $(M_{\sol},J_{\sol},g_{\sol})$ such that 
\begin{equation*}
    (\pi:M\to Y,\widehat{\pi}(\widetilde{\nu}))\cong (\pi_{\sol}:M_{\sol} \to Y_{\sol},o)
\end{equation*}
as Fano fibration germs, where $\pi_{\sol}:M_{\sol}\to Y_{\sol}$ is the polarized Fano fibration guaranteed by \cite{SunZhang}, with vertex $o$.
\end{conjecture}

In the special case where $Y$ is a point, Conjecture \ref{conj:main1} follows from \cite{ChenWangI,ChenwangII,bamscalar,CSW,DerSz,HanLiUniqueness,JunshengBPFree}. While Theorem \ref{thm:main1} does not directly apply in this setting, a similar (but simpler) argument yields the following strengthening of previously known convergence results. 

\begin{mtheorem} \label{thm:main3} Suppose $(M,J,g_{\sol})$ is a compact K\"ahler-Ricci shrinker. Then there exists $a>0$ such that for any K\"ahler metric $\omega \in c_1(M)$, the unique K\"ahler-Ricci flow solution $(\widetilde{\omega}_t)_{t\in [-1,0)}$ with $\widetilde{\omega}_{-1} = \omega$ can be written as $\zeta^{\ast}\widetilde{\omega}_t = \omega_{\sol,t}+\sqrt{-1}\partial \overline{\partial} \widetilde{\varphi}_t$ for some biholomorphism $\zeta:M\to M$ and some $\widetilde{\varphi} \in C^{\infty}(M\times [-1,0))$ satisfying 
\begin{equation*}
    \sup_M |(\nabla^{g_{\sol,t}})^k \widetilde{\varphi}_t|_{g_{\sol,t}} \leq C_k |t|^{1+a-\frac{k}{2}}
\end{equation*}
for all $k\in \mathbb{N}$. 
\end{mtheorem}
\begin{remark} In terms of the modified flow $(\omega_s)_{s\in [0,\infty)}$ defined in \eqref{eq:intro:modifiedKRF}, this implies exponentially fast convergence $\omega_s \to \omega_{\sol}$ as $s\to \infty$ at the level of K\"ahler potentials, improving the convergence rates of \cite[Corollary 4.3]{DerSz} and \cite[Theorem 0.5]{TZZdeformed}. Compared with \cite[Theorem 1]{PSSW11}, \cite[Theorem 1]{GPS}, and \cite[Theorem 0.1]{TZZZ}, no symmetry assumptions on the initial metric are required. This answers in the affirmative a question of \cite[p. 2722]{zhuICM}.
\end{remark}

\subsection{Outline of proof}
We briefly describe the proof of Theorem \ref{thm:main1}. By our assumption that the flow on $M$ can be holomorphically approximated by the AC K\"ahler-Ricci shrinker $(M_{\sol},J_{\sol},g_{\sol},f_{\sol})$, there is a sequence of holomorphic open embeddings $\psi_i:M_{\sol}\supseteq U_i \to V_i \subseteq M$ realizing the convergence 
\begin{equation*}
    |t_i|^{-1}\psi_i^{\ast}\widetilde{g}_{t_i} \to g_{\sol}, \qquad \psi_i^{\ast}\widetilde{f}_{t_i} \to f_{\sol}+W \quad \operatorname{in} \quad C_{\operatorname{loc}}^{\infty}(M_{\sol}).
\end{equation*}
Using this and the existence of $\pi:M\to Y$, we show in Lemma \ref{lem:deldelbar} that
\begin{equation*}
    |t_i|^{-1}\psi_i^{\ast} \widetilde{\omega}_{t_i}=\omega_{\sol}+\sqrt{-1}\partial \overline{\partial}\varphi_{i,-1}, \qquad \varphi_{i,-1}\to 0 \quad \operatorname{in} \quad C_{\operatorname{loc}}^{\infty}(M_{\sol}).
\end{equation*}
for some $\varphi_i$ defined on large open sets of $M_{\sol}$. We fix $i\in \mathbb{N}$ sufficiently large. On a spacetime domain growing exponentially fast in $s\in [0,\infty)$, there is then an induced solution $(\omega_s)_{s\in [0,\infty)}$ of the modified K\"ahler-Ricci flow
\begin{equation}
\label{eq:intro:modifiedKRF} \partial_s \omega_s =-\left( \Ric(\omega_s)+\mathcal{L}_{\frac{1}{2}X}\omega_s-\omega_s\right),
\end{equation}
where $X:= \nabla^{g_{\sol}}f_{\sol}$. Moreover, we can write $\omega_s = \omega_{\sol}+\sqrt{-1}\partial \overline{\partial}\varphi_s$, where $\varphi_0=\varphi_{i,-1}$ and where the potential satisfies
\begin{equation} \label{eq:intro:modifiedCMA}
\partial_s\varphi_s=
\log\left(\frac{(\omega_{\sol}+\sqrt{-1}\partial\overline{\partial}
\varphi_s)^n}{\omega_{\sol}^n}\right)
-\frac{1}{2}X\cdot\varphi_s+\varphi_s.
\end{equation}
Our main goal is to prove exponential decay for $\varphi_s$ on increasingly large open sets as $s\to \infty$, as such an estimate can then be combined with a pseudolocality argument to establish our main results.

To determine if such a bound is plausible, we first consider the linearization of \eqref{eq:intro:modifiedCMA} at $\varphi=0$, which is the following drift heat equation:

\begin{equation} \label{eq:intro:linear}
\partial_s\psi_s=
\Delta_{\omega_{\sol}}\psi_s-\frac{1}{2}X\cdot\psi_s+\psi_s.
\end{equation}
By analyzing the neutral and unstable modes of the operator in the linearization \eqref{eq:intro:linear}, we show in Section \ref{section:lineartheory} that they are precisely the directions coming from certain
pluriharmonic functions of sub-quadratic growth and from holomorphic vector fields that commute with the soliton vector field. In particular, they correspond to changes of gauge rather than to genuine instabilities of the metric, suggesting our desired estimate should hold if one modifies the potential appropriately and changes the reference metric $\omega_{\sol}$. We also prove sharp pointwise growth estimates for these eigenfunctions and their derivatives in Lemma \ref{lem:eigenfunctiongrowth}.

This leads to the construction, in Section \ref{section:modelmetrics},
of a finite-dimensional family of nearby model shrinkers
\begin{equation*} 
\omega_h=\omega_{\sol}+\sqrt{-1}\partial\overline{\partial}u_h,
\qquad h\in H,
\end{equation*}
obtained from $\omega_{\sol}$ by biholomorphisms preserving $X$. Changing the reference model from one metric $\omega_h$ to another absorbs the non-pluriharmonic neutral directions of the
linearized equation. Our construction follows the strategy of \cite{ChiuSzek}. We also prove sharp growth estimates for the relative K\"ahler potentials of model metrics and their derivatives in Proposition \ref{prop:modelmetrics}.

The next goal is to make precise in what sense the analysis of \eqref{eq:intro:linear} passes to the nonlinear equation \eqref{eq:intro:modifiedCMA}. As a first step, in Lemma \ref{lem:unnormalizedpseudolocality},  we prove local persistence estimates showing
that, once the flow is close to a model shrinker on a large region at
one scale, it remains reasonably close on the corresponding expanding parabolic region for a definite amount of modified time. In Lemmas \ref{lem:comparabilityofpotentials} and \ref{lem:improvementofcomparability}, we also record and slightly extend some key estimates for conjugate heat kernels and their parabolic regularizations established in \cite{FangLi}. 

Using the technical results of Section \ref{section:KRFnearAC}, we proceed to prove the non-concentration estimate, Lemma \ref{lem:nonconcentration}. Roughly
speaking, it says that if $|\varphi_{s_0}|$ is small in weighted
$L^2$ on a large sublevel set at one time, then $|\varphi_{s_0+1}|$ is comparably small in $L^{2+\epsilon}$ on a larger sublevel set. This prevents the integral from concentrating at
spatial infinity. The bound for the negative part $(\varphi_{s_0+1})_-$ uses the rough strategy of \cite{LSS,LiSzek,ghosh}, but the proof for the positive part must modify this strategy considerably, making crucial use of the precise heat kernel estimates from \cite{FangLi}.

In Proposition \ref{prop:decay}, we show that any solution of \eqref{eq:intro:modifiedCMA} which is small in a large region must have a definite decay in its $L^2$-norm on appropriate regions, after gauge modifications to eliminate the neutral and unstable modes. This is derived from our linear theory by a contradiction-compactness argument, using a blow-up strategy analogous to that used in the K\"ahler-Einstein setting \cite{ChiuSzek} (see also  Simon \cite[p. 270]{SimonNotes}, \cite{Cheeger-Tian} for earlier incarnations of this argument).

In Section \ref{section:mainproofs}, we iterate the decay proposition. The iteration produces a rapidly convergent sequence of model parameters $h_j\to h_\infty$ and hence a
limiting model shrinker $\omega_{h_\infty}$. With respect to this
limiting model, the sequence of modified potentials decays as $s\to \infty$. In order to obtain the same estimates for a fixed potential with these properties, we combine the exponential decay of the correction terms from Proposition \ref{prop:decay}, the pointwise bounds for eigenfunctions of the drift Laplacian from Lemma \ref{lem:eigenfunctiongrowth}, and the pointwise bounds of the relative K\"ahler potentials from Proposition \ref{prop:modelmetrics}. Returning to the original un-normalized time variable gives fast convergence
of the K\"ahler--Ricci flow to the self-similar shrinker in a fixed
holomorphic gauge on a region of size approximately $\sqrt{|t|\log\frac{1}{|t|}}$ as $t\nearrow 0$. A final pseudolocality and interpolation argument
converts the integral decay into the pointwise estimates for the
potential and its derivatives stated in Theorem \ref{thm:main1}. In
particular, the limiting shrinker and the holomorphic gauge are
independent of the rescaling sequence, giving uniqueness of the tangent
flow.

The corollaries follow by combining Theorem \ref{thm:main1} with the known tangent flow information in these settings. The proof of Theorem \ref{thm:main3} is a
compact analogue of the same stability argument, where no
non-concentration at infinity is needed.

\subsection{Acknowledgments} The authors are grateful to Richard Bamler, Ronan Conlon, Alix Deruelle, Hajo Hein, James Hotchkiss, Kimoi Kemboi, Jiewon Park, Felix Schulze, Jian Song, Gang Tian, Hung Tran, Lu Wang, and Xiaohua Zhu for helpful discussions and comments. LL is funded by the German Research Foundation (DFG) – Project-ID 427320536 – SFB 1442, and by Germany’s Excellence Strategy EXC 2044/2 390685587, Mathematics Münster: Dynamics–Geometry–Structure.

\subsection{AI Disclosure} The writing and mathematical content of this paper is due to the authors, with the small exception of a strategy suggested by ChatGPT to simplify the proof of Proposition \ref{prop:modelmetrics} \ref{modelmetrics3}. ChatGPT was otherwise used only to assist with literature searches and to identify inconsistencies in notation and minor errors.

\section{Preliminaries}

\subsection{Notation and conventions}
\label{subsection:notation}

We collect here the basic notation and conventions used throughout the paper. For more specialized notation, we instead indicate the location in the paper where these definitions are given.

We let $\mathbb{N}$ denote the set of nonnegative integers.\\

\paragraph{\emph{K\"ahler geometry.}}
Unless otherwise stated, all manifolds are smooth and connected, and
$n$ denotes their complex dimension. If $(M,J,g)$ is K\"ahler, we denote its K\"ahler form
by
\begin{equation*}
\omega(\cdot,\cdot)=g(J\cdot,\cdot)
\end{equation*}
and its Riemannian volume form by
\begin{equation*}
dg=\frac{\omega^n}{n!}.
\end{equation*}
We use $\nabla^g$ for the Levi--Civita connection, and write
$\langle\cdot,\cdot\rangle_g$ and $|\cdot|_g$ for the corresponding
inner product and norm. When the metric is clear from the context, the
superscript or subscript $g$ may be omitted. The musical
isomorphisms corresponding to $g$ are denoted by $\sharp$ and $\flat$.

Our Laplacian is the K\"ahler Laplacian
\begin{equation*}
\Delta_\omega u
 :=
 \operatorname{tr}_{\omega}
 \bigl(\sqrt{-1}\partial\overline{\partial}u\bigr)
 =
 g^{i\overline{j}}
 \partial_i\partial_{\overline{j}}u,
\end{equation*}
which is equal to half of the Riemannian Laplace-Beltrami operator. We also write $\Delta_g$ for this operator when $g$ is the metric
associated with $\omega$. 
In particular, $-\Delta_g$ has
nonnegative spectrum on a compact manifold. 
We set $d^c:=\frac{\sqrt{-1}}{2}
(\overline{\partial}-\partial)$, so that
$dd^c=\sqrt{-1}\partial\overline{\partial}$.
We use $\mathcal{A}^k(M)$ and $\mathcal{A}^{p,q}(M)$ for the spaces of
smooth real $k$-forms and smooth complex $(p,q)$-forms, respectively. We let $J$ act on $\alpha \in \mathcal{A}^1(M)$ by $J\alpha=\alpha \circ J$.

The Ricci form and K\"ahler scalar curvature are
\begin{equation*}
\Ric(\omega)
 =
 -\sqrt{-1}\partial\overline{\partial}
 \log\omega^n,
\qquad
\operatorname{R}_\omega = \operatorname{R}_g
 =
 \operatorname{tr}_{\omega}\Ric(\omega).
\end{equation*}
We write $\Ric(g)$ and $\Rm(g)$ for the Ricci and Riemann curvature
tensors. Scalar curvature is always taken in the K\"ahler
normalization above; it is one half of the Riemannian scalar curvature.\\

\paragraph{\emph{Metric notation.}}
For a Riemannian metric $g$, we let $d_g$ denote the corresponding length metric and $B_g(x,r)$ the open geodesic ball. We similarly let
\begin{equation*}
\operatorname{Vol}_{g},
\qquad
\operatorname{diam}_{g},
\qquad
\operatorname{inj}_{g}
\end{equation*}
denote volume, diameter, and injectivity radius with respect to $g$.

For a tensor $T$ on an open set $U\subset M$, we define 
\begin{equation*}
\|T\|_{C^k(U,g)}
 :=
 \sum_{j=0}^k
 \sup_U
 \bigl|(\nabla^g)^jT\bigr|_g.
\end{equation*}
\\

\paragraph{\emph{Singular-time and heat-kernel notation.}} The metric space $(M_0,d_0)$ is defined in Section \ref{sec-basic setup for polarized KRF}, and $d_{W_1}$ denotes the 1-Wasserstein distance.\\

\paragraph{\emph{Shrinking solitons.}}

The notation $g_{\sol},f_{\sol},b_{\sol},\Omega(D),\phi_s,g_{\sol,t},f_{\sol,t},b_{\sol,t},\Delta_{f_{\sol}},\nu_{\sol},W,X$ is established in subsection \ref{subsection:ACshrinkers}.\\

\paragraph{\emph{Model metrics.}}
The finite-dimensional vector space $H$, the biholomorphisms
$\zeta_s^h$, and the associated objects
\[
g_h,\quad \omega_h,\quad f_h,\quad u_h, \quad b_h, \quad \Omega_{h}(D),
\qquad h\in H,
\]
and the corresponding time-dependent quantities are introduced in Section \ref{section:modelmetrics}. \\

\paragraph{\emph{Constants and asymptotic notation.}}
The letters $C,C_k,\ldots$ denote positive constants whose values may
change from line to line. Dependence on the complex dimension
$n$ is suppressed. When we say that a statement holds if
\begin{equation*}
\epsilon\leq\overline{\epsilon}(b_1,\ldots,b_\ell),
\end{equation*}
means that the statement holds whenever $\epsilon$ is bounded above by
a sufficiently small positive constant depending only on the displayed parameters.
Analogously,
\begin{equation*}
A\geq\underline{A}(b_1,\ldots,b_\ell)
\end{equation*}
means that $A$ is required to be sufficiently large in terms of those parameters.
The notation
\begin{equation*}
\Psi(a_1,\ldots,a_k\mid b_1,\ldots,b_\ell)
\end{equation*}
denotes a nonnegative quantity depending on the displayed parameters
such that
\begin{equation*}
\lim_{(a_1,\ldots,a_k)\to(0,\ldots,0)}
\Psi(a_1,\ldots,a_k\mid b_1,\ldots,b_\ell)
=0
\end{equation*}
for every fixed choice of $b_1,\ldots,b_\ell$.

\subsection{K\"ahler-Ricci shrinkers}

\label{subsection:ACshrinkers}

By a K\"ahler-Ricci shrinker, we mean a connected and complete non-flat gradient shrinking K\"ahler-Ricci soliton. We usually denote a K\"ahler-Ricci shrinker by $(M_{\sol},J_{\sol},g_{\sol},f_{\sol})$, subject to the normalizations
\begin{equation}\label{eq:solitonnormalization}
    \Ric(\omega_{\sol}) + \mathcal{L}_{\frac{1}{2}X}\omega_{\sol} = \omega_{\sol}, \qquad \operatorname{R}_{g_{\sol}} + |\partial f_{\sol}|_{g_{\sol}}^2 = f_{\sol},
\end{equation}
where $X:=\nabla^{g_{\sol}}f_{\sol}$ denotes the (real-holomorphic) complete soliton vector field \cite{ZhuhongZhang}. The entropy of $(M_{\sol},g_{\sol})$ is
\begin{equation*}
    W:= \log \left( \frac{1}{(2\pi)^n} \int_{M_{\sol}} e^{-f_{\sol}} dg_{\sol} \right),
\end{equation*}
which is the constant such that
\begin{equation*}
    d\nu_{\sol} := \frac{1}{(2\pi)^n} e^{-(f_{\sol}+W)}dg_{\sol}
\end{equation*}
is a probability measure on $M_{\sol}$. Moreover, since $f_{\sol}\geq \operatorname{R}_{g_{\sol}} >0$ \cite{BinglongChen} on $M_{\sol}$, we may define 
\begin{equation*}
    b_{\sol}:= \sqrt{2f_{\sol}}, \qquad \Omega(D):= \{x\in M_{\sol} \: | \: b_{\sol}(x) \leq D\}
\end{equation*}
for all $D >0$. We also define the drift Laplacian $\Delta_{f_{\sol}}$ of the shrinker by
\begin{equation*}
    \Delta_{f_{\sol}}u:= \Delta_{g_{\sol}}u-\frac{1}{2}X\cdot u = \Delta_{g_{\sol}}u -\frac{1}{2} \langle \nabla^{g_{\sol}}f_{\sol},\nabla^{g_{\sol}}u\rangle_{g_{\sol}}. 
\end{equation*}
This is a densely defined formally self-adjoint operator on $L^2(M_{\sol},\nu_{\sol})$. We let $(\phi_s)_{s\in \mathbb{R}}$ be the flow of $\frac{1}{2}X$, so that the self-similar ancient solution of the K\"ahler-Ricci flow associated to the shrinker is
\begin{equation*}
g_{\sol,t} = |t|\,\phi_{\log \frac{1}{|t|}}^*g_{\sol},\qquad t<0.
\end{equation*}
We similarly extend the above definitions to time-dependent objects:
\begin{equation*}
    f_{\sol,t} := \phi_{\log \frac{1}{|t|}}^{\ast}f_{\sol}, \qquad b_{\sol,t}:=\sqrt{2|t|f_{\sol,t}}, \qquad \Omega_t(D):= \{x \in M_{\sol} \: | \: b_{\sol,t}(x)\leq D\}.
\end{equation*}

We now recall some well-known facts concerning Ricci shrinkers.

\begin{theorem} \label{thm:basicKRSfacts}
    \begin{enumerate}
        \item \label{thm:basicKRSfacts1} For any $x_0 \in \operatorname{argmin}_{M_{\sol}}f_{\sol}$, there is a constant $C(g_{\sol})>0$ such that
        \begin{equation*}
   \frac{1}{2}(d_{g_{\sol}}(x_0,x)-C(g_{\sol}))_+^2 \le f_{\sol}(x)\le \frac{1}{2}(d_{g_{\sol}}(x_0,x)+C(g_{\sol}))^2,
\end{equation*}
holds for all $x\in M_{\sol}$. 

    \item \label{thm:basicKRSfacts2} For any $\rho \in (0,1)$, there exists $D_{g_{\sol}}(\rho) \in (1,\infty)$ such that on $M_{\sol}\setminus \Omega (D_{g_{\sol}}(\rho))$, we have 
    \begin{equation*}
        (1-\rho)d_{g_{\sol}}(x_0,\cdot) \leq b_{\sol} \leq (1+\rho)d_{g_{\sol}}(x_0,\cdot).
    \end{equation*}
    \end{enumerate}
\end{theorem}
\begin{proof}
\ref{thm:basicKRSfacts1} This is \cite[Theorem 1.1]{CaoZhou}.

\ref{thm:basicKRSfacts2} This is immediate from \ref{thm:basicKRSfacts1}.
\end{proof}

\begin{lemma} \label{lem:primitivecomparingbandd} If $(M_{\sol},g_{\sol})$ has bounded scalar curvature, then there exists $C_1 =C_1(g_{\sol})\in \mathbb{R}$ such that for all $t\in [-1,0)$ the following holds on $M_{\sol}$:  \begin{equation*}
    b_{\sol}\le d_{g_{\sol,t}}(x_0,\cdot)+C_1.
\end{equation*}
\end{lemma}
\begin{proof}
Recalling that $|\nabla^{g_{\sol}}b_{\sol}|_{g_{\sol}}\le 1$ so that $b_{\sol} \leq d_{g_{\sol}}(x_0,\cdot)+b_{\sol}(x_0)$, 
\begin{equation} \label{eq:primitivefboundinclaim}
    f_{\sol,t} \leq \frac{1}{2} \left(\frac{d_{g_{\sol,t}}(x_0,\cdot)}{\sqrt{|t|}}+b_{\sol}(x_0) \right)^2
\end{equation}
for all $t\in [-1,0]$. By Theorem \ref{thm:basicKRSfacts}, we have \begin{equation} \label{eq:primitivetimederivative}
    \partial_t(|t|f_{\sol,t})=-f_{\sol,t}+|t|\cdot |\partial f_{\sol,t}|_{g_{\sol,t}}^2 = -|t|\operatorname{R}_{g_{\sol,t}} \in [-C(g_{\sol}),0].
\end{equation}
Integrating this from time $-1$ to time $t\in [-1,0)$ and combining with \eqref{eq:primitivefboundinclaim} yields
\begin{equation*}
    f_{\sol} \leq C(g_{\sol})+\frac{1}{2}(d_{g_{\sol,t}}(x_0,\cdot)+C(g_{\sol}))^2,
\end{equation*}
from which the claim follows.
\end{proof}

We now specialize to the case where $M_{\sol}$ is AC. 

\begin{theorem} \label{thm:AC} \cite[Theorem A]{ConlonDeruelleSun} If $(M_{\sol},J_{\sol},g_{\sol},f_{\sol})$ is an AC K\"ahler-Ricci shrinker, then there is a K\"ahler cone metric $g_{\mathcal{C}}$ on the Remmert reduction $\pi_{\sol}:M_{\sol}\to \mathcal{C}$ of $M_{\sol}$, such that the following properties hold:
\begin{enumerate}
    \item $|(\pi_{\sol})_*g_{\sol}-g_\mathcal{C}|_{g_{\mathcal{C}}}=O(r^{-2})$ on $\mathcal{C}\setminus \{o\}$ as $r\to \infty$, where $r$ denotes the radial function of the K\"ahler cone $(\mathcal{C},g_\mathcal{C})$ and $o$ is the vertex;
    \item $d\pi_{\sol}(X)=r\partial_r$;
    \item The exceptional set $\pi_{\sol}^{-1}(o)$ is a compact analytic subset.
\end{enumerate}
\end{theorem}

For an AC K\"ahler-Ricci shrinker, $(\pi_{\sol})_{\ast}g_{\sol,t}$ converges locally smoothly to $g_{\mathcal{C}}$ on $\mathcal{C}\setminus \{o\}$ as $t \nearrow 0$. Under this identification, $\pi_{\sol}\circ\phi_s\circ\pi_{\sol}^{-1}$ multiplies the radial coordinate by $e^{\frac{s}{2}}$ for every $s\in\R$.

\begin{lemma} \label{lem:basicACfacts}
    \begin{enumerate}
        \item \label{lem:basicACfacts1} $b_{\sol,t}$ converge uniformly on $M_{\sol}$ as $t\nearrow 0$ to $\pi_{\sol}^{\ast}r$.

        \item \label{lem:basicACfacts2} For all $(x,t)\in M_{\sol} \times (-\infty,0)$, we have
        \begin{equation*}
            C(g_{\sol})^{-1}(b_{\sol,t}(x)+\sqrt{|t|})\leq b_{\sol,0}(x)+\sqrt{|t|} \leq (b_{\sol,t}(x)+\sqrt{|t|}).
        \end{equation*}

        \item \label{lem:basicACfacts3} For any $k\in \mathbb{N}$, there exists $C_k =C_k(g_{\sol})$ such that for all $(x,t)\in M_{\sol}\times (-\infty,0)$, we have
        \begin{equation*}
            |(\nabla^{g_{\sol,t}})^k\Rm(g_{\sol,t})|_{g_{\sol,t}}(x) \leq \frac{C_k}{(b_{\sol,t}(x)+\sqrt{|t|})^{k+2}}.
        \end{equation*}

        \item \label{lem:basicACfacts4}
        There exists $C(g_{\sol}) \in (1,\infty)$ such that for any $x\in M_{\sol}$ and $r\in (0,\infty)$,
        \begin{equation*}
            \operatorname{Vol}_{g_{\sol}}(B_{g_{\sol}}(x,r))\geq C(g_{\sol})^{-1}r^{2n}.
        \end{equation*}
    \end{enumerate}
\end{lemma}
\begin{remark} \label{rem:alsocompact}
Lemma \ref{lem:basicACfacts} \ref{lem:basicACfacts3} also holds if $M_{\sol}$ is compact instead of AC, which is a consequence of Shi's derivative estimates. 
\end{remark}
\begin{proof}
    By \eqref{eq:primitivetimederivative}, we have
    \begin{equation*}
        0\leq b_{\sol,t}^2(x)-b_{\sol,0}^2(x) \leq C(g_{\sol})|t|,
    \end{equation*}
by which \ref{lem:basicACfacts1} and \ref{lem:basicACfacts2} are immediate. \ref{lem:basicACfacts3} holds when $t=-1$ by \cite[Section 2.2.3]{ConlonDeruelleSun}, so the claim for general $t<0$ follows by parabolic rescaling. \ref{lem:basicACfacts4} holds for any Riemannian cone $\mathcal{C}$ with smooth link by a basic computation. For an AC K\"ahler-Ricci shrinker, we use the fact that $M_{\sol}$ and $\mathcal{C}$ are quasi-isometric outside a compact set. 
\end{proof}

\subsection{Singular time-slice and tangent flows}

Suppose $(M,(g_t)_{t\in [-T,0)})$ is a compact K\"ahler-Ricci flow developing a finite-time singularity at $t=0$. For $x,y\in M, s<t$, one can define the heat kernel $K(x,t;y,s)$ of the Ricci flow by
\begin{equation}
    \begin{cases}
        &(\partial_t -\Delta_{\omega_t}) K(\cdot,\cdot;y,s)=0,\\
        &(-\partial_s - \Delta_{\omega_s}+\operatorname{R}_{\omega_s}) K(x,t;\cdot,\cdot)=0,\\
        &\lim _{s\to t} K(x,t;\cdot,s)=\delta_x,\\
        &\lim_{t\to s}K(\cdot,t;y,s)=\delta_y.\\     
    \end{cases}
\end{equation}
We define the conjugate heat kernel measure $\nu_s$ based at point $(x,t) \in M\times [-T,0)$ by
\begin{equation*}
    d{\nu}_{x,t;s}=K(x,t;\cdot,s)dg_s
\end{equation*}
for all $s\in [-T,t)$. Letting $n=\dim_\C M$, we define the corresponding \textit{potential function} $f_s = f_{x,t;s}$ by 
\begin{equation*}
     d\nu_{x,t;s}=(2\pi (t-s))^{-n}e^{- f_s}dg_s.
\end{equation*}
The \textit{pointed Nash entropy} at $(x,t)$ is
given by 
\begin{equation*} \mathcal{N}_{x,t}(\tau):= \int_M f_{t-\tau}\,d\nu_{x,t;t-\tau}-n. \end{equation*}

Following \cite{Bam2,Bam3}, we define $M_0$, the metric space at the singular time 0, as the set consisting of conjugate heat flows with variance going to 0 as time approaches 0. 

\begin{definition}[\cite{Bam3}]\label{conjugate kernel based at sing time}
    A conjugate heat flow based at the singular time is a conjugate heat flow $\left(\mu_t\right)_{t \in [-T, 0)}$ on $(M,(g_t)_{t\in [-T,0)})$ with the property that
\begin{equation}
\lim _{t \nearrow 0} \operatorname{Var}\left(\mu_t\right)=0.    
\end{equation}Let $M_0$ denote the set of all conjugate heat flows based at the singular time $0$. A distance function $d_0$ on $M_0$ can be defined as follows  
\begin{equation}
    d_0(\mu_1,\mu_2) := \lim_{t \nearrow 0} d_{W_1}^{g_t}(\mu_{1,t},\mu_{2,t}).
\end{equation}
\end{definition}
We extend the pointed Nash entropy to $\mu=(\mu_t)_{t\in [-T,0)} \in M_0$ by writing $d\mu_t = (2\pi|t|)^{-n}e^{-f_t}dg_t$, and setting
\begin{equation*}
    \mathcal{N}_{\mu}(\tau):= \int_{M} f_{-\tau} d\mu_{-\tau} -n.
\end{equation*}
\begin{definition} For $\delta>0$ and $r\in (0,\infty)$, we say $(M,(g_t)_{t\in [-T,0)})$ is $(\delta,r)$-self-similar at $\mu \in M_0$ if $10r^2 \leq T$ and
\begin{equation*}
   \mathcal{N}_{\mu}(10^{-1}r^2)- \mathcal{N}_{\mu}(10r^2) < \delta.
\end{equation*}
\end{definition}

\subsection{Convergence of \texorpdfstring{$H_{2n}$}{H\_2n}-centers for polarized K\"ahler-Ricci flow}\label{sec-basic setup for polarized KRF} Let $(M,(\widetilde{g}_t)_{t\in [-T,0)})$ be a K\"ahler-Ricci flow developing a finite-time singularity at time $0$, and assume that there exists a surjective holomorphic map $\pi:M\to Y$ to a normal K\"ahler variety $(Y,\omega_Y)$ such that $\pi^{\ast}[\omega_Y]=[\widetilde{\omega}_{-T}]-Tc_1(M)$. 
Noting that the parabolic Schwarz lemma \cite[Lemma 3.7.3]{SongWeinkove} still applies for holomorphic maps to K\"ahler spaces (cf. \cite[p. 115]{TosattiZhang3folds}), there exists a constant $C>0$ such that for any $t \in [-T,0)$, 
\begin{equation}
    \widetilde{\omega}_t \;\geq\; \tfrac{1}{C}\,\pi^*\omega_Y.
\end{equation}
Given this, the proof of \cite[Lemma 2.10]{CHM} goes through verbatim in our situation, and we use its consequences in the proof of the following.

\begin{lemma}\label{limit on the base}
 There exists a surjective Lipschitz map $\widehat{\pi}:(M_0,d_0)\rightarrow (Y,d_{g_Y})$ satisfying the following property for some constant $C \in (0,\infty)$. For any $\mu\in M_0$, $\widehat{\pi}(\mu)$ is the unique point $y\in Y$ such that 
\begin{equation}\label{schwarz lemma} \pi(x_t)\in B_{g_Y}(y, C\sqrt{|t|}), \end{equation}
holds for every $t\in [-T,0)$ and any $H_{2n}$-center $(x_t,t)$ of $\mu$.
\end{lemma}
\begin{proof}
Let $(x_{t_1},t_1)$ and $(x_{t_2},t_2)$ be two $H_{2n}$-centers of $\mu\in M_0$ with $-T\leq t_1<t_2<0$. Then by the property of $H_{2n}$-centers and monotonicity of $d_{W_1}$-distance, we obtain
\begin{equation}\label{2.6}
\begin{aligned}
        d_{W_1}^{\widetilde{g}_{t_1}}(\delta_{x_{t_1}},\nu_{x_{t_2},t_2;t_1})&\leq d_{W_1}^{\widetilde{g}_{t_1}}(\delta_{x_{t_1}},\mu_{t_1})+d_{W_1}^{\widetilde{g}_{t_1}}(\mu_{t_1},\nu_{x_{t_2},t_2;t_1})\leq \sqrt{H_{2n}|t_1|}+\sqrt{H_{2n}|t_2|}
\end{aligned}
\end{equation} Let $(z_{t_1},t_1)$ be an $H_{2n}$-center of $(x_{t_2},t_2)$. Then, using \eqref{2.6} and the triangle inequality, we can estimate the distance between $z_{t_1}$ and $x_{t_1}$: 
\begin{equation}\label{different Hn centers}
    d_{\widetilde{g}_{t_1}}(z_{t_1}, x_{t_1})\leq d_{W_1}^{\widetilde{g}_{t_1}}(\delta_{z_{t_1}}, \nu_{x_{t_2},t_2;t_1})+d_{W_1}^{\widetilde{g}_{t_1}}(\nu_{x_{t_2},t_2;t_1}, \delta_{x_{t_1}})\leq \sqrt{H_{2n}}(\sqrt{t_2-t_1}+\sqrt{|t_2|}+\sqrt{|t_1|}).
\end{equation} Building on work of \cite{jiansongtian},  it has been proved in \cite[Lemma 2.10]{CHM} that for some constant $C \in (0,\infty)$ depending only on $\widetilde{g}_{-T}$, 
\begin{equation}\label{estimate in CHM}
    d_{g_Y}(\pi(z_{t_1}),\pi(x_{t_2}))\leq C\sqrt{H_{2n}(t_2-t_1)}.
\end{equation}
Combining \eqref{2.6}, \eqref{different Hn centers} and  \eqref{estimate in CHM}, we obtain 
\begin{equation}\label{estimate of distance on Y}
    d_{g_Y}(\pi(x_{t_1}),\pi(x_{t_2}))\leq C\sqrt{H_{2n}} (\sqrt{t_2-t_1}+\sqrt{|t_2|}+\sqrt{|t_1|}).
\end{equation}
From \eqref{estimate of distance on Y}, it follows that the sequence $\pi(x_t)$ converges to a unique limit with respect to the distance $d_{g_Y}$; we denote this limit by $\widehat{\pi}(\mu)$. By letting $t_2 \nearrow 0$ in \eqref{estimate of distance on Y}, we then obtain \eqref{schwarz lemma}. To show that $\widehat{\pi}$ is Lipschitz, suppose $\mu^1,\mu^2 \in M_0$, and let $(z_{t}^i,t)$ be an $H_{2n}$-center of $\mu^i$ for $i=1,2$. Then we can estimate
\begin{align*}
    d_{g_Y}(\widehat{\pi}(\mu^1),\widehat{\pi}(\mu^2)) \leq & d_{g_Y}(\widehat{\pi}(\mu^1),\pi(z_t^1))+d_{g_Y}(\pi(z_t^1),\pi(z_t^2))+d_{g_Y}(\pi(z_t^2),\widehat{\pi}(\mu^2))\\
    \leq & C\sqrt{|t|}+Cd_{\widetilde{g}_t}(z_t^1,z_t^2)\\
    \leq & C\sqrt{|t|}+Cd_{W_1}^{\widetilde{g}_t}(\delta_{z_t^1},\mu_t^1)+Cd_{W_1}^{\widetilde{g}_t}(\mu_t^1,\mu_t^2)+Cd_{W_1}^{\widetilde{g}_t}(\mu_t^2,\delta_{z_t^2})\\
    \leq & C\sqrt{|t|}+Cd_0(\mu^1,\mu^2).
\end{align*}
Taking $t\nearrow 0$, we conclude that $\widehat{\pi}$ is Lipschitz. Finally, for any $y\in Y$, choose $x\in \pi^{-1}(y)$, and let $\mu \in M_0$ be a conjugate heat kernel based at $(x,0)$, in the sense of \cite[Definition 2.2]{CHM}. Then $\widehat{\pi}(\mu)=\pi(x)=y$ by \cite[Lemma 2.10]{CHM}, hence $\widehat{\pi}$ is surjective.
\end{proof}

\subsection{Elliptic estimates and interpolation inequalities}
We now recall a weak version of the standard interior Schauder estimate with scalings for elliptic equations, which will be used throughout this paper. Throughout this section, we let $\Delta_g$ denote the Riemannian Laplace-Beltrami operator of $(M,g)$.

\begin{lemma}[Interior elliptic derivative estimate]\label{lem: schauder estimates}
    Let $(M^m,g)$ be a complete Riemannian manifold, and suppose that for some $x\in M$ and $r \in (0,\infty)$, there exist $(\Lambda_k)_{k\in \mathbb{N}}$ such that
\begin{equation*}\inf_{B_g(x,r)}\operatorname{inj}_g(\cdot)\geq \Lambda_0^{-1} r, \qquad 
\sup_{B_g(x,r)}\big|(\nabla^g)^k\Rm(g)\big|_g\le \Lambda_k\,r^{-2-k}.
\end{equation*}
   If $u\in C^\infty\big(B_g(x,r)\big)$ satisfies $\Delta_g u=f$, then for each $k\in \mathbb{N}$,
there exists 
$C_k=C_k(\Lambda_0,\dots,\Lambda_k) \in (1,\infty)$ such that 
\begin{equation*}
\|u\|_{C^{k+2}_r\left(B_g(x,\frac{r}{2})\right)}
\le C_k\Big(r^2\,\|f\|_{C^{k+1}_r\left(B_g(x,r)\right)}
      +\|u\|_{L^\infty\left(B_g(x, r)\right)}\Big),
\end{equation*}
where the scale-invariant norm $\|\cdot\|_{C_r^k}$ is defined by
\begin{equation*}
\|v\|_{C^{k}_r\left(B_g(x,\rho)\right)}
:=\sum_{i=0}^{k} r^{i}\sup_{B_g(x,\rho)}\big|(\nabla^g)^i v\big|_g .
\end{equation*}
\end{lemma}
\begin{proof}
By rescaling, we may assume $r=1$. 
The hypotheses ensure a lower bound on the
$C^{k+1,\frac{1}{2}}$-harmonic radius \cite[Theorem 6]{Hebeyharmonic}.
Thus there exists $\delta=\delta(\Lambda_0)\in(0,\frac{1}{2}]$ such that any $y\in B_g(x,\frac{1}{2})$ admits harmonic coordinates on $B_{g}(y,\delta)$ in which $g_{ij}$ is
uniformly $C^{k+1,\frac{1}{2}}$-controlled in the sense of \cite[Definition 5]{Hebeyharmonic}. Applying the Euclidean interior Schauder
estimate \cite[Corollary 11.2.3]{PetersenRmG} 
in these coordinates gives
\begin{align*}
\|u\|_{C^{k+2}(B_{g}(y,\frac{\delta}{2}),\, g)}
&\le C_k\Big(\big\|f\big\|_{C^{k,\frac{1}{2}}(B_{g}(y,\delta),\,g) }
      +\|u\|_{L^\infty(B_{g}(y,\delta))}\Big)\\
&\leq C_k \left( \|f\|_{C^{k+1}(B_g(x,1)),\,g)} +\|u\|_{C^0(B_g(x,1))} \right),
\end{align*}
from which the claim follows.
\end{proof}

Lemma \ref{lem: schauder estimates} will often be combined with the following standard  $L^2-L^{\infty}$ estimate.

\begin{lemma} \label{lem:L2toLinfty}
Let $(M^m,g)$ be a complete Riemannian manifold, and suppose that for some $x\in M$ and $r \in (0,\infty)$, there exists $\Lambda \in (1,\infty)$ such that
\begin{equation*}\inf_{B_g(x,r)}\operatorname{inj}_g(\cdot)\geq \Lambda^{-1} r, \qquad 
\sup_{B_g(x,r)}|\Rm(g)|_g\le \Lambda\,r^{-2}.
\end{equation*}
Then there exists $C=C(\Lambda) \in (1,\infty)$ such that if $u\in C^\infty\big(B_g(x,r)\big)$ satisfies $\Delta_g u=f$, then
\begin{equation*}
    \sup_{B_g(x,\frac{r}{2})} |u| \leq C(\Lambda)\left(\left( \frac{1}{r^m} \int_{B_g(x, r)} u^2 dg \right)^{\frac{1}{2}} + r^2 \sup_{B_g(x,r)}|f| \right).
\end{equation*}
\end{lemma}
\begin{proof} By choosing local coordinates as in Lemma \ref{lem: schauder estimates}, this follows from \cite[Theorem 8.17]{GilbargTrudinger}.
\end{proof}

We will also require the following standard interpolation inequality.

\begin{lemma} \label{lem:interpolation}  Let $(M^m,g)$ be a complete Riemannian manifold, and suppose that for some $x\in M$ and $r \in (0,\infty)$, there exist $(\Lambda _k)_{k\in \mathbb{N}}$ such that
\begin{equation*}\inf_{B_g(x,r)}\operatorname{inj}_g(\cdot)\geq \Lambda_0^{-1} r, \qquad 
\sup_{B_g(x,r)}\big|(\nabla^g)^k\Rm(g)\big|_g\le \Lambda_k\,r^{-2-k}.
\end{equation*}
Then there exist $C_k=C_k(\Lambda_0,...,\Lambda_{2k+2}) \in (1,\infty)$ such that for any $u\in C^{\infty}(B_g(x,r))$, we have
\begin{equation*}
    \| \nabla^k u\|_{C^0(B_{g}(x,\frac{r}{2}),g)} \leq C_k \left(\|\nabla^{2k}u\|_{C^{0}(B_{g}(x,r),g)}^{\frac{1}{2}}\|u\|_{C^0(B_{g}(x, r),g)}^{\frac{1}{2}}+r^{-k}\|u\|_{C^0(B_{g}(x,r),g)} \right).
\end{equation*}
\end{lemma}
\begin{proof}
    By arguing as in Lemma \ref{lem: schauder estimates}, it suffices to prove this when $r=1$ and in a harmonic coordinate chart. In this case, the claim follows from \cite[Theorem 2.2 and Remark 5]{Nirenberg} with $m=2k$, $j=k$, $p=q=r=\infty$, and $a=\frac{1}{2}$.
\end{proof}

\subsection{Holomorphic approximation}

In this subsection, we prove a sufficient criterion for a K\"ahler-Ricci flow to be holomorphically approximated by a K\"ahler-Ricci shrinker. We then verify several examples where this criterion is satisfied, and lastly show that if the shrinker is AC, the convergence also holds at the level of K\"ahler potentials.

\begin{lemma} \label{lem:perturb} Suppose $(M,J,(\widetilde{\omega}_t)_{t\in [-T,0)})$ is a compact K\"ahler-Ricci flow and $\nu=(\nu_t)_{t\in [-T,0)}$ is a conjugate heat kernel based at the singular time, such that some tangent flow is a smooth K\"ahler-Ricci shrinker $(M_{\sol},J_{\sol},g_{\sol},f_{\sol})$ admitting an exhaustion by strictly pseudoconvex precompact open sets $W_j \subseteq M_{\sol}$ with smooth boundary satisfying
\begin{equation*}
    H^1(W_j,T^{1,0}W_j)=0.
\end{equation*}
Then $(M,J,(\widetilde{\omega}_t)_{t\in [-T,0)})$ can be holomorphically approximated by $(M_{\sol},J_{\sol},g_{\sol},f_{\sol})$ at $\nu$.
\end{lemma}
\begin{proof} Without loss of generality, we may assume $\overline{W}_j \subseteq W_{j+1}$ for each $j\in \mathbb{N}$. Fix $j\in \mathbb{N}$, and write $W:=W_j$ and $W':=W_{j+1}$. For $i\geq \underline{i}(j)$, $\overline{W'}$ is contained in the domain of $\psi_i$. Because $H^1(W',T^{1,0}W')=0$ and $\psi_{i}^{\ast}J \to J_{\sol}$ (see \cite[Theorem 2.5]{HallgrenJian}) in $C^{\infty}(\overline{W'})$, \cite[Main Theorem]{hamilton} gives embeddings $\eta_i:\overline{W'} \to M_{\sol}$ satisfying the following:
\begin{align*}
    \eta_i^{\ast}J_{\sol} =  \psi_{i}^{\ast}J, \qquad \eta_i \to \operatorname{id}_{W'} \text{ in } C^{\infty}(\overline{W'}).
\end{align*}
For $i\geq \underline{i}(j)$, it follows that $\eta_i(W')\supseteq W$, hence $\widehat{F}_i := \psi_{i} \circ \eta_i^{-1} : W \to \psi_{i}(W')$ are open holomorphic embeddings satisfying $|t_i|^{-1}\widehat{F}_i^{\ast}\widetilde{g}_{|t_i| t} \to g_{\sol,t}$ in $C_{\operatorname{loc}}^{\infty}(W)$. The claim then follows by a diagonal argument. 
\end{proof}

\begin{corollary} \label{cor:complexrigid} Suppose $(M,J,(\widetilde{\omega}_t)_{t\in [-T,0)})$ is a compact K\"ahler-Ricci flow and $\widetilde{\nu}=(\widetilde{\nu}_t)_{t\in [-T,0)}$ is a conjugate heat kernel based at the singular time, such that some tangent flow is a smooth AC K\"ahler-Ricci shrinker $(M_{\sol},J_{\sol},g_{\sol},f_{\sol})$ satisfying $H^1(M_{\sol},T^{1,0}M_{\sol})=0$. Then $(M,J,(\widetilde{\omega}_t)_{t\in [-T,0)},\widetilde{\nu})$ can be holomorphically approximated by $(M_{\sol},J_{\sol},g_{\sol},f_{\sol})$.    
\end{corollary}
\begin{proof} Because $W_j := \Omega(j)$ has strictly pseudoconvex boundary for $j\in \mathbb{N}$ sufficiently large, \cite[Theorem 4.6]{Ohsawa} gives $H^1(W_j ,T^{1,0}W_j)\cong H^1(M_{\sol},T^{1,0}M_{\sol})=0$, and the claim then follows from Lemma \ref{lem:perturb}.
\end{proof}

\begin{example} \label{example:FIK} Let $M_{\operatorname{FIK}}$ denote the total space of $\mathcal{O}_{\mathbb{P}^1}(-1)$, which is the underlying complex manifold of an AC K\"ahler-Ricci shrinker constructed in \cite{FIK}. We will show that $H^1(M_{\operatorname{FIK}},T^{1,0}M_{\operatorname{FIK}})=0$, thereby ensuring that the hypotheses of Corollary \ref{cor:complexrigid} are satisfied. Next, we let $\pi:M_{\operatorname{FIK}}\to \mathbb{P}^1$ be the bundle map, and recall the short exact sequence (see \cite[Section 3]{Aikou})
\begin{equation*}  
0\to \pi^{\ast}\mathcal{O}_{\mathbb{P}^1}(-1) \to T^{1,0}M_{\operatorname{FIK}}\to \pi^{\ast}T^{1,0}\mathbb{P}^1 \to 0
\end{equation*}
of holomorphic vector bundles. This induces a long exact sequence of cohomology, which includes
\begin{equation}
    \cdots \to H^1(M_{\operatorname{FIK}},\pi^{\ast}\mathcal{O}_{\mathbb{P}^1}(-1))\to H^1(M_{\operatorname{FIK}},T^{1,0}M_{\operatorname{FIK}}) \to H^1(M_{\operatorname{FIK}},\pi^{\ast}T^{1,0}\mathbb{P}^1) \to \cdots.
\end{equation}
Because $\pi$ is affine, for any holomorphic line bundle $L$ on $\mathbb{P}^1$, the algebraic sheaf cohomology groups $H_{\operatorname{alg}}^1$ satisfy 
\begin{equation*}
    H_{\operatorname{alg}}^1(M_{\operatorname{FIK}},\pi^{\ast}L)\cong H_{\operatorname{alg}}^1(\mathbb{P}^1,L\otimes \pi_{\ast}\mathcal{O}_{M_{\operatorname{FIK}}})= H_{\operatorname{alg}}^1 \left( \mathbb{P}^1,L\otimes \operatorname{Sym}(\mathcal{O}_{\mathbb{P}^1}(1)) \right) \cong \bigoplus_{k=0}^{\infty} H_{\operatorname{alg}}^1(\mathbb{P}^1,L\otimes \mathcal{O}_{\mathbb{P}^1}(k)).
\end{equation*}
Taking $L=T^{1,0}\mathbb{P}^1 \cong \mathcal{O}_{\mathbb{P}^1}(2)$ and $L=\mathcal{O}_{\mathbb{P}^1}(-1)$, and using the fact that $H_{\operatorname{alg}}^1(\mathbb{P}^1,\mathcal{O}_{\mathbb{P}^1}(j))=0$ whenever $j\geq -1$, we conclude that 
\begin{equation*}
H_{\operatorname{alg}}^1(M_{\operatorname{FIK}},\pi^{\ast}\mathcal{O}_{\mathbb{P}^1}(-1))=0=H_{\operatorname{alg}}^1(M_{\operatorname{FIK}},\pi^{\ast}T^{1,0}\mathbb{P}^1).
\end{equation*}
Because the blowdown map $M_{\operatorname{FIK}}\to \mathbb{C}^2$ exhibits $M_{\operatorname{FIK}}$ as a complex analytic space admitting a proper morphism to an affine variety, the relative GAGA principle \cite[Exp XII, Th\'eor\`eme 4.2]{SGA} implies 
\begin{equation*}
H^1(M_{\operatorname{FIK}},\pi^{\ast}\mathcal{O}_{\mathbb{P}^1}(-1))=0=H^1(M_{\operatorname{FIK}},\pi^{\ast}T^{1,0}\mathbb{P}^1),
\end{equation*}
from which the claim follows.
\end{example}

\begin{lemma} \label{lem:deldelbar} Suppose
    $(M,J,(\widetilde{g}_t)_{t\in [-T,0)},(\widetilde{\nu}_t)_{t\in [-T,0)})$ can be holomorphically approximated by an AC K\"ahler-Ricci shrinker $(M_{\sol},J_{\sol},g_{\sol},f_{\sol})$.
    Then there are $t_i \nearrow 0$, $A_i \to \infty$, solutions 
    $\varphi_{i} \in C^{\infty}(\Omega(A_i) \times [-A_i,0))$ to
    \begin{equation*}
        \partial_t \varphi_{i,t} = \log \left( \frac{(\omega_{\sol,t}+\sqrt{-1}\partial \overline{\partial} \varphi_{i,t})^n}{\omega_{\sol,t}^n} \right),
    \end{equation*}
    and open holomorphic embeddings $F_i:\Omega(A_i)\to M$ such that $|t_i|^{-1}F_i^{\ast}\omega_{|t_i|t} = \widetilde{\omega}_{\sol,t} + \sqrt{-1} \partial \overline{\partial} \varphi_{i,t}$ and
    \begin{align*} F_i^{\ast}\widetilde{f}_{|t_i|t}\to f_{\sol,t}+W, \qquad \varphi_{i} \to 0  
    \end{align*}
    in $C_{\operatorname{loc}}^{\infty}(M_{\sol} \times (-\infty,0))$, where we write $d\widetilde{\nu}_t=(2\pi|t|)^{-n}e^{-\widetilde{f}_t}d\widetilde{g}_t$. 
\end{lemma}
\begin{proof} By \cite[Theorem 1.4]{hacon}, there exists a holomorphic map $\pi:M \to Y$ with connected fibers to a normal K\"ahler space $(Y,\omega_Y)$ with $\pi^{\ast}[\omega_Y]=[\widetilde{\omega}_{-T}]-Tc_1(M)$. By Definition \ref{definition: holomorphic approximation}, there exist $A_i \to \infty$, $t_i \nearrow 0$, and holomorphic open embeddings $F_i:\Omega(A_i)\to M$ such that
\begin{equation*}
    |t_i|^{-1}F_i^{\ast} \widetilde{g}_{|t_i|t} \to g_{\sol,t}, \qquad F_i^{\ast}\widetilde{f}_{|t_i|t} \to f_{\sol,t}+W
\end{equation*}
in $C_{\operatorname{loc}}^{\infty}(M_{\sol}\times (-\infty,0))$ as $i\to \infty$. By the local $\partial \overline{\partial}$-lemma for complex spaces, there exists a neighborhood $\mathcal{W}$ of $\widehat{\pi}(\widetilde{\nu})$ such that $\omega_Y = \sqrt{-1}\partial \overline{\partial} \psi$ for some $\psi \in C^{\infty}(\mathcal{W})$. By \cite[proof of Lemma 2.11]{CHM}, there exists a neighborhood $\mathcal{V} \subseteq M$ of $\pi^{-1}(\widehat{\pi}(\widetilde{\nu}))$ such that $\pi(\mathcal{V}) \subseteq \mathcal{W}$, and such that for any $A \in (1,\infty)$, we have $F_i(\Omega(A)) \subseteq \mathcal{V}$ for sufficiently large $i=i(A)\in \mathbb{N}$. In particular, we can write $\omega_Y=\sqrt{-1}\partial \overline{\partial}\psi$ for some smooth function $\psi$ defined in a neighborhood of $\pi(\mathcal{V})$. Because $[\widetilde{\omega}_{t_i}] = \pi^{\ast}[\omega_Y]+|t_i|c_1(M)$, we can find smooth volume forms $\Theta_i \in \mathcal{A}^{n,n}(M)$ satisfying
$\Ric(\Theta_i) = \frac{1}{|t_i|}(\widetilde{\omega}_{t_i}-\pi^{\ast}\omega_Y)$. 
It follows that
\begin{equation*}
    \widehat{\varphi}_{i,-1}:=\log \left( \frac{e^{-f_{\sol}}\omega_{\sol}^n}{F_i^{\ast}(e^{-\frac{1}{|t_i|}\pi^{\ast}\psi} \Theta_i)} \right)
\end{equation*}
satisfies
\begin{equation*}
    \sqrt{-1}\partial \overline{\partial} \widehat{\varphi}_{i,-1}= \frac{1}{|t_i|}F_i^{\ast}(\widetilde{\omega}_{t_i}-\pi^{\ast}\omega_Y)-\omega_{\sol}+\frac{1}{|t_i|}F_i^{\ast}(\pi^{\ast}\omega_Y) = \frac{1}{|t_i|}F_i^{\ast} \widetilde{\omega}_{t_i} - \omega_{\sol}.
\end{equation*}
By applying \cite[Lemma 3.2.1]{Hodge} to the compact manifold with boundary $\Omega(A)$, we therefore conclude that there exists $\beta_i \in \mathcal{A}^1(\Omega(A))$ such that $d\beta_i = \frac{1}{|t_i|}F_i^{\ast}\widetilde{\omega}_{t_i}-\omega_{\sol}$ and $\beta_i \to 0$ in $C_{\operatorname{loc}}^{\infty}(\Omega(A))$. We write $\beta_i = \beta_i^{0,1}+\overline{\beta_i^{0,1}}$, where $\beta_i^{0,1} \in \mathcal{A}^{0,1}(\Omega(A-1))$ satisfy the following:
\begin{equation*}
    \overline{\partial}\beta_i^{0,1} =0, \qquad 2\operatorname{Re}( \partial \beta_i^{0,1})=\frac{1}{|t_i|}F_i^{\ast}\widetilde{\omega}_{t_i}-\omega_{\sol}, \qquad \beta_i^{0,1} \to 0 \text{ in } C^{\infty}(\Omega(A-1)).
\end{equation*}
   Since $\Omega(A-1)$ has smooth, strictly pseudoconvex boundary for sufficiently large $A$, and because
   \begin{equation*}
       \Ric(\omega_{\sol})+ \sqrt{-1}\partial \overline{\partial}f_{\sol} = \omega_{\sol},
   \end{equation*}
    we can apply H\"ormander's $L^2$ estimates to find $u_i \in L^2(\Omega(A-1))$ satisfying $\overline{\partial}u_i = \beta_i^{0,1}$ and
   \begin{align*}
        \int_{\Omega(A-1)} |u_i|^2 e^{-f_{\sol}}\omega_{\sol}^n \leq \int_{\Omega(A-1)} |\beta_i|_{g_{\sol}}^2 e^{-f_{\sol}} \omega_{\sol}^n.
   \end{align*}
   Then $\varphi_{i,-1}:= 2\operatorname{Im}(u_i)$ satisfies similar estimates, as well as $\sqrt{-1} \partial \overline{\partial} \varphi_{i,-1} = \frac{1}{|t_i|}F_i^{\ast}\widetilde{\omega}_{t_i}-\omega_{\sol}$. Local elliptic regularity then implies $\varphi_{i,-1} \to 0$ in $C_{\operatorname{loc}}^{\infty}(\Omega(A-1))$, and the claim for $t=-1$ follows from a diagonal argument. Setting
\begin{equation*}
    \varphi_{i,t}:= \varphi_{i,-1} + \int_{-1}^t \log \left( \frac{(|t_i|^{-1}F_i^{\ast}\widetilde{\omega}_{|t_i|s})^n}{\omega_{\sol,s}^n} \right)ds,
\end{equation*}
one may compute
\begin{align*}
    \partial_t(|t_i|^{-1}F_i^{\ast}\widetilde{\omega}_{|t_i|t}-\omega_{\sol,t}-\sqrt{-1}\partial \overline{\partial} \varphi_{i,t})=0,
\end{align*}
and the claim for general $t$ follows.
\end{proof}

\section{The drift Laplacian of K\"ahler-Ricci shrinkers}

\label{section:lineartheory}

The purpose of this section is to analyze the spectrum of the drift Laplacian $\Delta_{f_{\sol}}$ corresponding to a fixed K\"ahler-Ricci shrinker $(M_{\sol}^{n},J_{\sol},g_{\sol},f_{\sol})$. As a first step, we note the following sufficient criterion for a Killing field on a K\"ahler-Ricci soliton to be real-holomorphic. 

\begin{prop}\label{prop: killing is real-holomorphic}
    If $(M_{\sol},g_{\sol})$ is not a metric product of the form $M'\times\mathbb{C}^{k}$, then
any Killing field of $M_{\sol}$ is real-holomorphic. 
\begin{proof} We recall that $M_{\sol}$ is simply connected \cite{esparzafundamental,SunZhang}. If some Killing field of $(M_{\sol},g_{\sol})$ is not real-holomorphic, then by \cite[Th\'eor\`eme 1]{Lich}, $M_{\sol}$ splits as a metric product $(M_{\sol},g_{\sol})=(M'\times M'',g'+g'')$, where $(M'',g'')$ is a Ricci flat manifold of positive dimension. We have rewritten the proof of this fact in the Appendix as Theorem \ref{thm:lich} for the reader's convenience. It follows by \cite[Lemma 2.5]{ChowLowDim} that $(M'',g'')$ is a Ricci-flat Ricci shrinker, which is therefore isometric to $\mathbb{R}^j$ for some $j \in \{1,...,2n\}$. By \cite[p.3]{HallgrenJian}, it follows that $(M_{\sol},g_{\sol})$ splits a factor of the flat Gaussian shrinker on $\mathbb{C}^k$ for some $k\in \{1,...,n\}$.
\end{proof}
\end{prop}

Let $d_{f_{\sol}}^*$ be the formal adjoint of $d: \mathcal{A}^k(M_{\sol})\mapsto \mathcal{A}^{k+1}(M_{\sol})$ with respect to the inner product given by
\begin{equation*}
    \langle\alpha,\beta\rangle=\int_{M_{\sol}}  \langle\alpha,\beta\rangle_{g_{\sol}}d\nu_{\sol}
\end{equation*}
for all $\alpha,\beta\in\mathcal{A}^k(M_{\sol})$. The following is needed to establish a certain $L^2$-orthogonal decomposition of $L^2(M_{\sol},\nu_{\sol})$.

\begin{lemma} \label{lem:closedrange} The bounded linear operator
\begin{equation*}
    d:W^{1,2}(M_{\sol},\nu_{\sol})\to L^2(M_{\sol},\nu_{\sol}; T^{\ast}M_{\sol}), \qquad u\mapsto du 
\end{equation*}
has closed image.
\end{lemma}
\begin{proof}
For any $u \in C_c^{\infty}(M_{\sol})$, the weighted Poincare inequality yields
\begin{equation} \label{eq:consequenceofpoincare}
    2\int_{M_{\sol}} |du|_{g_{\sol}}^2 d\nu_{\sol} \geq \int_{M_{\sol}} |u|^2 d\nu_{\sol} - \left| \int_{M_{\sol}} u d\nu_{\sol} \right|^2.
\end{equation}
By \cite[Theorem 1.7]{Grigoryan}, $C_c^{\infty}(M_{\sol})$ is dense in $W^{1,2}(M_{\sol},\nu_{\sol})$, so \eqref{eq:consequenceofpoincare} holds more generally for all $u\in W^{1,2}(M_{\sol},\nu_{\sol})$. Suppose $(u_j)_{j\in \mathbb{N}}$ is a sequence in $W^{1,2}(M_{\sol},\nu_{\sol})$ such that $(du_j)_{j\in \mathbb{N}}$ is a Cauchy sequence in $L^2(M_{\sol},\nu_{\sol};T^{\ast}M_{\sol})$. Setting $v_j := u_j -\int_{M_{\sol}} u_j d\nu_{\sol}$, we take $u=v_j-v_k$ in \eqref{eq:consequenceofpoincare} to obtain
\begin{equation*}
    \int_{M_{\sol}} |v_j - v_k|^2 d\nu_{\sol} \leq 2\int_{M_{\sol}} |du_j -du_k|^2 d\nu_{\sol}, 
\end{equation*}
hence $(v_j)_{j\in \mathbb{N}}$ is a Cauchy sequence in $W^{1,2}(M_{\sol},\nu_{\sol})$. It follows that $v_j \to v$ for some $v\in W^{1,2}(M_{\sol},\nu_{\sol})$ satisfying $\lim_{j\to \infty}\|du_j-dv\|_{L^2(M_{\sol},\nu_{\sol})}=0$.
\end{proof}

We now define a family of holomorphy potentials for holomorphic vector fields commuting with $\nabla^{g_{\sol}}f_{\sol}$.

\begin{definition}\label{definition: model metrics}
Let $H$ be the vector space of $h\in C^{\infty}(M_{\sol})\cap L^2(M_{\sol},\nu_{\sol})$ satisfying the following:
\begin{enumerate}
    \item \label{def:H1} $-\Delta_{f_{\sol}}h=h$,
    \item \label{def:H2} $(J_{\sol}\nabla^{g_{\sol}}f_{\sol})\cdot h=0$,
    \item \label{def:H3} $J_{\sol}\nabla^{g_{\sol}}h$ is a real-holomorphic Killing field.
\end{enumerate}
We equip $H$ with the norm $\|\cdot\|_{H}$ defined by
\begin{equation*}
\|h\|_{H}^{2}:=\int_{M_{\sol}}h^{2}d\nu_{\sol}
\end{equation*}
for $h\in H$. 
\end{definition}

\begin{remark}
    Note that $f_{\sol}-n \in H$, hence $H\neq \{0\}$ if $M_{\sol}$ is not K\"ahler-Einstein.
\end{remark}

\begin{remark} \label{rem:vfieldscommute}
For any $h \in H$, because $\mathcal{L}_{J_{\sol}\nabla^{g_{\sol}}f_{\sol}}g_{\sol}=0$ and $\mathcal{L}_{J_{\sol}\nabla^{g_{\sol}}f_{\sol}}h=0$, we also have 
\begin{equation*} 0=\mathcal{L}_{J_{\sol}\nabla^{g_{\sol}}f_{\sol}}\nabla^{g_{\sol}}h=[J_{\sol}\nabla^{g_{\sol}}f_{\sol},\nabla^{g_{\sol}}h], \end{equation*}
hence also $\nabla^{g_{\sol}}h$ is a holomorphic vector field satisfying $[\nabla^{g_{\sol}}h,\nabla^{g_{\sol}}f_{\sol}]=0$.
\end{remark}

In the following lemma, we show that an eigenvalue $\lambda$ of $-\Delta_{f_{\sol}}$ corresponds to an $f_{\sol}$-homogeneous pluriharmonic function $u_{PH}$ if $\lambda<1$. Moreover, the $\lambda=1$ eigenspace is spanned by homogeneous pluriharmonic functions and $H$.

\begin{lemma} \label{lem:lineartheory}
Assume that $M_{\sol}$ does not split a factor of $\mathbb{C}^{k}$ isometrically. 
\begin{enumerate}
    \item  \label{lem:lineartheory-assertion (i)} If $\eta\in\mathcal{A}^{1}(M_{\sol})$ is nonzero and satisfies $\int_{M_{\sol}}|\eta|_{g_{\sol}}^{2}d\nu_{\sol}<\infty$
and $(dd_{f_{\sol}}^{\ast}+d_{f_{\sol}}^{\ast}d)\eta=\lambda\eta$ for some $\lambda\in\mathbb{R}$,
then $\lambda\geq1$. 

    \item\label{lem:lineartheory-assertion (ii)} If in addition $d_{f_{\sol}}^{\ast}\eta=0$, then $\lambda\geq2$,
with equality only if $\eta^{\sharp}$ is a real-holomorphic Killing
field, in which case we also have $\eta(\nabla^{g_{\sol}} f_{\sol})=0$.

    \item \label{lem:lineartheory-assertion (iii)}If $u\in C^{\infty}(M_{\sol}) \cap L^2(M_{\sol},\nu_{\sol})$ is not pluriharmonic and satisfies 
\begin{equation*} -\Delta_{f_{\sol}}u=\lambda u\end{equation*} 
for some $\lambda\in\mathbb{R}$, then
$\lambda\geq 1,$ with equality if and only if $u=u_{\text{PH}}+h$ with
\begin{equation*}
    -\Delta_{f_{\sol}} u_{PH}=u_{PH},\quad  -\Delta_{f_{\sol}} h=h,
\end{equation*}
where $u_{\text{PH}}$ is pluriharmonic and $h\in H$.

\item \label{lem:lineartheory-assertion (iv)} If $u_{PH} \in C^{\infty}(M_{\sol})\cap L^2(M_{\sol},\nu_{\sol})$ is pluriharmonic and $\Delta_{f_{\sol}}u_{PH}=-u_{PH}$, and if $h\in H$, then
\begin{equation*}
    \int_{M_{\sol}} u_{PH}h d\nu_{\sol}=0.
\end{equation*}

\end{enumerate}
\end{lemma}
\begin{proof}
\ref{lem:lineartheory-assertion (i)} Observe that $d_{f_{\sol}}^*=d^*+\iota_{X}$, where $X:=\nabla^{g_{\sol}}f_{\sol}$ and $d^*$ is the formal adjoint of $d$ with respect to the (unweighted) Riemannian volume measure. For any $\eta \in \mathcal{A}^1(M_{\sol})$, we compute
\begin{equation*}
    \begin{split}
       ( dd_{f_{\sol}}^*+d^*_{f_{\sol}}d)\eta&=d(d^*+\iota_{X})\eta+(d^*+\iota_{X})d\eta\\
       &=(dd^*+d^*d)\eta+d\iota_{X}\eta+\iota_{X}d\eta\\
       &=(dd^*+d^*d)\eta+\mathcal{L}_X\eta.
    \end{split}
\end{equation*}
For any vector field $Y$ on $M_{\sol}$, we can compute
\begin{align*}
    (\mathcal{L}_X\eta)(Y) & = (\nabla^{g_{\sol}}_X \eta)(Y)+\eta(\nabla^{g_{\sol}}_Y X)\\ &= (\nabla_X^{g_{\sol}}\eta)(Y)+(\nabla^{g_{\sol}} \nabla^{g_{\sol}}f_{\sol})(\eta^{\sharp},Y),
\end{align*}
so combining expressions gives
\begin{equation*}
    ( dd_{f_{\sol}}^*+d^*_{f_{\sol}}d)\eta=(dd^*+d^*d)\eta+(\nabla^{g_{\sol}}\nabla^{g_{\sol}}f_{\sol})(\eta^\sharp)+\nabla^{g_{\sol}}_{X}\eta.
\end{equation*}
Further combining this with the Weitzenb\"ock formula 
\begin{equation*}
    (dd^*+d^*d)\eta=(\nabla^{g_{\sol}})^*\nabla^{g_{\sol}}\eta+\Ric_{g_{\sol}}(\eta^\sharp)
\end{equation*}
yields the following weighted Weitzenb\"ock formula:
\begin{equation} \label{eq:weightedweitzenbock}
     ( dd_{f_{\sol}}^*+d^*_{f_{\sol}}d)\eta=(\nabla_{f_{\sol}}^{g_{\sol}})^*\nabla^{g_{\sol}}\eta+(\Ric_{g_{\sol}}+\nabla^{g_{\sol}}\nabla^{g_{\sol}}f_{\sol})(\eta^\sharp) = (\nabla_{f_{\sol}}^{g_{\sol}})^*\nabla^{g_{\sol}}\eta+\eta^{\sharp},
\end{equation}
where $(\nabla^{g_{\sol}})_{f_{\sol}}^{\ast} = (\nabla^{g_{\sol}})^{\ast} + \iota_X$ is formal adjoint of $\nabla^{g_{\sol}}$
with respect to the measure $\nu_{\sol}$. 

Fix $x_{0}\in M_{\sol}$, and let $\phi_{r}\in C^{\infty}(M_{\sol})$ satisfy $\phi_{r}|_{B_{g_{\sol}}(x_{0},r)}\equiv 1$,
$\operatorname{supp}(\phi_{r})\subseteq B_{g_{\sol}}(x_{0},2r)$, and $|\nabla^{g_{\sol}}\phi_{r}|_{g_{\sol}}\leq10r^{-1}$.
If $\eta\in\mathcal{A}^{1}(M_{\sol})$ satisfies 
\begin{equation*}
   (dd_{f_{\sol}}^{\ast}+d_{f_{\sol}}^{\ast}d)\eta =\lambda \eta, \qquad \int_{M_{\sol}} |\eta|_{g_{\sol}}^2 d\nu_{\sol}<\infty
\end{equation*}
for some $\lambda\in\mathbb{R}$,
then we can integrate by parts \eqref{eq:weightedweitzenbock} against $\phi_r^2 \eta$ and use Cauchy's inequality to obtain
\begin{align*}
    (\lambda-1)\int_{M_{\sol}} |\eta|_{g_{\sol}}^2 \phi_r^2d\nu_{\sol} &= \int_{M_{\sol}} \langle (\nabla_{f_{\sol}}^{g_{\sol}})^{\ast}\nabla^{g_{\sol}}\eta,\phi_r^2\eta\rangle_{g_{\sol}}d\nu_{\sol}\\
    &= \int_{M_{\sol}}|\nabla^{g_{\sol}}\eta|_{g_{\sol}}^2 \phi_r^2 d\nu_{\sol} - \frac{C}{r}\int_{M_{\sol}}|\nabla^{g_{\sol}}\eta|_{g_{\sol}}\phi_r |\eta|_{g_{\sol}} d\nu_{\sol}\\
    &\geq (1-\epsilon)\int_{M_{\sol}}|\nabla^{g_{\sol}}\eta|_{g_{\sol}}^2 \phi_r^2 d\nu_{\sol} -\frac{C(\epsilon)}{r^2}\int_{M_{\sol}} |\eta|_{g_{\sol}}^2 d\nu_{\sol}
\end{align*}
for any $\epsilon>0$. Taking $r\to \infty$, using the monotone convergence theorem, and then taking $\epsilon \searrow 0$ gives
\begin{equation*}
    (\lambda-1)\int_{M_{\sol}} |\eta|_{g_{\sol}}^2 d\nu_{\sol} \geq \int_{M_{\sol}} |\nabla^{g_{\sol}}\eta|_{g_{\sol}}^2 d\nu_{\sol},
\end{equation*}
from which it follows that $\lambda \geq 1$. 

\ref{lem:lineartheory-assertion (ii)} By arguing as in \ref{lem:lineartheory-assertion (i)}, we may integrate by parts to obtain
\begin{equation*}
\lambda\int_{M_{\sol}}|\eta|_{g_{\sol}}^{2}d\nu_{\sol}=\int_{M_{\sol}}\left(|d\eta|_{g_{\sol}}^{2}+(d_{f_{\sol}}^{\ast}\eta)^{2}\right)d\nu_{\sol} = \int_{M_{\sol}} |d\eta|_{g_{\sol}}^2 d\nu_{\sol}.
\end{equation*}
Next, because 
\begin{equation*} |\nabla\eta|_{g_{\sol}}^{2}=\frac{1}{2}|d\eta|_{g_{\sol}}^{2}+\frac{1}{4}|\mathcal{L}_{\eta^{\sharp}}g_{\sol}|_{g_{\sol}}^{2}, \end{equation*}
we can combine the above inequalities to obtain
\begin{equation*}
(\lambda-1)\int_{M_{\sol}}|\eta|_{g_{\sol}}^{2}d\nu_{\sol}=\int_{M_{\sol}}\left(\frac{1}{2}|d\eta|_{g_{\sol}}^{2}+\frac{1}{4}|\mathcal{L}_{\eta^{\sharp}}g_{\sol}|_{g_{\sol}}^{2}\right)d\nu_{\sol}=\frac{\lambda}{2}\int_{M_{\sol}}|\eta|_{g_{\sol}}^{2}d\nu_{\sol}+\frac{1}{4}\int_{M_{\sol}}|\mathcal{L}_{\eta^{\sharp}}g_{\sol}|_{g_{\sol}}^{2}d\nu_{\sol},
\end{equation*}
or equivalently 
\begin{equation*}
\left(\frac{\lambda}{2}-1\right)\int_{M_{\sol}}|\eta|_{g_{\sol}}^{2}d\nu_{\sol}=\frac{1}{4}\int_{M_{\sol}}|\mathcal{L}_{\eta^{\sharp}}g_{\sol}|^{2}d\nu_{\sol}.
\end{equation*}
In particular, we have $\lambda\geq 2$, with equality if and
only if $\eta^{\sharp}$
is a Killing field. If $\eta^{\sharp}$ is a Killing field, then combining this with our assumption $d_{f_{\sol}}^{\ast}\eta=0$ implies
\begin{equation*}
    0=d_{f_{\sol}}^{\ast}\eta= -\operatorname{tr}_{g_{\sol}}(\nabla^{g_{\sol}}\eta)+\iota_X \eta = -\frac{1}{2}\operatorname{tr}_{g_{\sol}}(\mathcal{L}_{\eta^{\sharp}}g)+\iota_X \eta = \iota_X\eta.
\end{equation*}

For assertion \ref{lem:lineartheory-assertion (iii)}, suppose $u\in C^{\infty}(M_{\sol})\cap L^2(M_{\sol},\nu_{\sol})$ satisfies $-\Delta_{f_{\sol}}u=\lambda u$. Then an integration by parts
argument as in \ref{lem:lineartheory-assertion (i)} yields 
\begin{equation*}
\int_{M_{\sol}}|\partial  u|_{g_{\sol}}^{2}d\nu_{\sol}=\lambda\int_{M_{\sol}}u^{2}d\nu_{\sol}.
\end{equation*}
By Lemma \ref{lem:closedrange}, if $T:W^{1,2}(M_{\sol},\nu_{\sol})\to L^2(M_{\sol},\nu_{\sol};T^{\ast}M_{\sol})$ denotes the map $w\mapsto dw$, and $T^{\ast}$ is the Hilbert space adjoint, then
\begin{equation*}
L^2(M_{\sol},\nu_{\sol};T^{\ast}M_{\sol})=T\left( W^{1,2}(M_{\sol},\nu_{\sol}) \right) \oplus \ker(T^{\ast}),
\end{equation*}
so we can write 
\begin{equation*}
d^{c}u=dv+\alpha
\end{equation*}
for some $v\in W^{1,2}(M_{\sol},\nu_{\sol})$ and some $\alpha \in \ker(T^{\ast})$. For any $w\in C_c^{\infty}(M_{\sol})$, we then have
\begin{equation*}
    0 = \langle T^{\ast}\alpha,w\rangle_{W^{1,2}(M_{\sol},\nu_{\sol})} = \langle \alpha,Tw\rangle_{L^2(M_{\sol},\nu_{\sol})} = \int_{M_{\sol}} \langle \alpha,dw\rangle_{g_{\sol}}d\nu_{\sol},
\end{equation*}
where $T^{\ast}$ is the Hilbert-space adjoint, so that $d_{f_{\sol}}^{\ast}\alpha=0$ in the sense of distributions. Because $d\alpha=dd^c u \in \mathcal{A}^2(M_{\sol})$ and $d_{f_{\sol}}^{\ast}dv=d_{f_{\sol}}^{\ast}d^c u \in C^{\infty}(M_{\sol})$, local elliptic regularity gives $v\in C^{\infty}(M_{\sol})$ and $\alpha \in \mathcal{A}^1(M_{\sol})$. 

From the Bochner formula for 1-forms, we have
\begin{align*}
(\nabla^{g_{\sol}})^*\nabla^{g_{\sol}} d^cu=&(d_{g_{\sol}}^{\ast}d+dd_{g_{\sol}}^{\ast})d^c u-\Ric_{g_{\sol}}(d^cu)\\
=& d^c(d_{g_{\sol}}^{\ast}du)-\Ric_{g_{\sol}}(d^c u)\\
=&-2d^c (\Delta_{g_{\sol}}u)-\Ric_{g_{\sol}}(d^c u),
\end{align*}
where we used that $d_{g_{\sol}}^{\ast}d+dd_{g_{\sol}}^{\ast}$ commutes with both $\partial$ and $\overline{\partial}$. Combining this with
\begin{equation*}
    d^c(X u)=(\nabla^{g_{\sol}} \nabla^{g_{\sol}}f_{\sol})(d^cu)+\iota_X{\nabla}^{g_{\sol}}d^cu.
\end{equation*}
and applying \eqref{eq:weightedweitzenbock} gives
\begin{equation*}
    \begin{split}
        (dd_{f_{\sol}}^*+d_{f_{\sol}}^*d)(d^cu)&=(\nabla^{g_{\sol}})^*\nabla^{g_{\sol}}( d^cu)+\iota_X \nabla^{g_{\sol}}(d^cu)+d^cu\\
        &=-2d^c(\Delta_{f_{\sol}}u)-(\Ric_{g_{\sol}}+\nabla^{g_{\sol}}\nabla^{g_{\sol}}f_{\sol})(d^cu)+d^cu\\
        &=2\lambda d^cu.
    \end{split}
\end{equation*}
Recalling that $d^c u=dv+\alpha$, we have
\begin{equation} \label{eq:lineartheoryseparation}
     dd_{f_{\sol}}^*dv+d_{f_{\sol}}^*d\alpha=2\lambda dv+2\lambda\alpha,
\end{equation}
and thus
\begin{equation*}
    \Delta_{f_{\sol}}(d_{f_{\sol}}^*dv-2\lambda v)=0.
\end{equation*}
Because $d^c u$ is an eigenfunction of the weighted Hodge Laplacian, we can integrate by parts as in \ref{lem:lineartheory-assertion (i)} to conclude that $d_{f_{\sol}}^{\ast}dv=d_{f_{\sol}}^{\ast}(d^c u) \in L^2(M_{\sol},\nu_{\sol})$.
Suppose $\lambda \in (0,1)$. Because the only $\Delta_{f_{\sol}}$-harmonic functions in $L^2(M_{\sol},\nu_{\sol})$ are the constants, we may therefore add a constant to $v$ in order to assume that
\begin{equation*}
    d_{f_{\sol}}^*dv=2\lambda v.
\end{equation*}
Combining with \eqref{eq:lineartheoryseparation} then implies
\begin{equation*}
    d_{f_{\sol}}^{\ast}d\alpha=2\lambda \alpha.
\end{equation*}
By \ref{lem:lineartheory-assertion (ii)}, we must have $\alpha=0$,
hence $dd^{c}u=d^{2}v=0$; in other words, $u$ is pluriharmonic.
If $\lambda=1$, then $\alpha^{\sharp}$ is a real-holomorphic Killing
field with $\alpha(X)=0$. We thus have
\begin{equation*}
0=\mathcal{L}_{\alpha^{\sharp}}\omega_{\sol}=d\iota_{\alpha^{\sharp}}\omega_{\sol}=-\frac{1}{2}d\left(\iota_{J_{\sol}\nabla^{g_{\sol}} u+2\nabla^{g_{\sol}} v}\omega_{\sol}\right)=-\frac{1}{2}d^2 u-2dd^c v=-2dd^{c}v,
\end{equation*}
so that $v$ is pluriharmonic. By applying $J$ to both sides of $d^{c}u=dv+\alpha$,
we get
\begin{equation*}
du=-4d^{c}v+2J_{\sol}\alpha.
\end{equation*}
Then applying $d$ and using $dd^{c}v=0$ yields $d(J_{\sol}\alpha)=0$. Because $M_{\sol}$ is simply connected by \cite{esparzafundamental,SunZhang}, we can therefore find $h\in C^{\infty}(M_{\sol})$ satisfying $2dh=J_{\sol}\alpha$,
or equivalently $2\nabla^{g_{\sol}} h=-J_{\sol}\alpha^{\sharp}$.

Moreover, $\nabla^{g_{\sol}} h$
is real-holomorphic, and $J_{\sol}\nabla^{g_{\sol}} h=\alpha^{\sharp}$ is a real-holomorphic
Killing field with 
\begin{equation*}
0=\alpha(X)=\langle J_{\sol}\nabla^{g_{\sol}} h,X\rangle_{g_{\sol}}=-\mathcal{L}_{J_{\sol}X}h.    
\end{equation*}
In particular, we have
\begin{equation*}
    2d^ch=2\alpha=d_{f_{\sol}}^*d\alpha=d^*dd^ch + \iota_X d\alpha=d^{\ast}dd^ch+\mathcal{L}_X\alpha
\end{equation*}
since $\alpha(X)=0$. From
\begin{equation*}
    d^{\ast}d^ch=\frac{1}{4}d^{\ast}\alpha=\frac{1}{4}d_{f_{\sol}}^{\ast}\alpha-\frac{1}{4}\alpha(X)=0,
\end{equation*}
we can compute
\begin{equation*}
    d^{\ast}dd^ch=(d^{\ast}d+dd^{\ast})d^ch=d^c(d^{\ast}d+dd^{\ast})h=-2d^c(\Delta_{g_{\sol}}h).
\end{equation*}
Combining expressions and using
\begin{equation*}
    \mathcal{L}_X \alpha = 4\mathcal{L}_X(d^c h)=4d^c(\mathcal{L}_Xh),
\end{equation*}
we conclude that 
\begin{equation*}
    2d^c h =-2d^c(\Delta_{g_{\sol}}h)+d^c(\mathcal{L}_Xh)=-2d^c(\Delta_{f_{\sol}}h).
\end{equation*}
Therefore, by adding an appropriate constant to $h$, we have $-\Delta_{f_{\sol}}h=h$. Since $|\partial h|_{g_{\sol}}=\frac{1}{2\sqrt{2}}|\alpha|_{g_{\sol}}\in L^2(M_{\sol},\nu_{\sol})$ and  $\Delta_{f_{\sol}}h=-h,$ a standard cut-off argument shows that $h\in L^2(M_{\sol},\nu_{\sol}),$ hence $h\in H$. Finally, $d(u-h)=-4d^c v$ implies $u_{PH}:=u-h$ is pluriharmonic, and because $\Delta_{f_{\sol}}u=-u$ and $\Delta_{f_{\sol}}h=-h$, we obtain $\Delta_{f_{\sol}}u_{PH}=-u_{PH}$, hence $\mathcal{L}_{\frac{1}{2}X}u_{PH}=u_{PH}$. 

\ref{lem:lineartheory-assertion (iv)} Setting $w:=\langle \nabla^{g_{\sol}}u_{PH},\nabla^{g_{\sol}}h\rangle_{g_{\sol}}$, Bochner's formula gives
\begin{align*}
    \Delta_{f_{\sol}}\langle \nabla^{g_{\sol}}u_{PH},\nabla^{g_{\sol}}h\rangle_{g_{\sol}} =& \langle \nabla^{g_{\sol}}\Delta_{f_{\sol}}u_{PH},\nabla^{g_{\sol}}h\rangle_{g_{\sol}}+\langle \nabla^{g_{\sol}}u_{PH},\nabla^{g_{\sol}}\Delta_{f_{\sol}}h\rangle_{g_{\sol}} +\langle \nabla^{g_{\sol}}u_{PH},\nabla^{g_{\sol}}h\rangle_{g_{\sol}}\\
    &+\langle \nabla^{g_{\sol}}\nabla^{g_{\sol}}u_{PH},\nabla^{g_{\sol}}\nabla^{g_{\sol}}h\rangle_{g_{\sol}}\\
    =&-\langle \nabla^{g_{\sol}}u_{PH},\nabla^{g_{\sol}}h\rangle_{g_{\sol}},
\end{align*}
where we used that $\langle \nabla^{g_{\sol}}\nabla^{g_{\sol}}u_{PH},\nabla^{g_{\sol}}\nabla^{g_{\sol}}h\rangle_{g_{\sol}}=0$ since $\nabla^{g_{\sol}}\nabla^{g_{\sol}}h$ has bidegree $(1,1)$ (because $\nabla^{g_{\sol}}h$ is real-holomorphic) and $\nabla^{g_{\sol}}\nabla^{g_{\sol}}u_{PH}$ is pointwise orthogonal to terms of bidegree $(1,1)$ (because it is pluriharmonic). By the proof of \ref{lem:lineartheory-assertion (i)}, we have $|\partial u_{PH}|_{g_{\sol}},|\partial h|_{g_{\sol}}\in L^2(M_{\sol},\nu_{\sol})$, and thus $w \in L^1(M_{\sol},\nu_{\sol})$. Choose $\chi \in C^{\infty}(\mathbb{R},[0,1])$ such that $\chi|_{(-\infty,1]}\equiv 1$, $\chi|_{[2,\infty)}\equiv 0$, we set 
$\chi_r(\cdot):=\chi(r^{-2}f_{\sol}(\cdot))$ for all $r\in (1,\infty)$. We then integrate by parts to obtain
\begin{equation}
    \begin{split}
        \int_{M_{{\sol}}}\chi_rwd\nu_{\sol}&=-\int_{M_{{\sol}}}\chi_r\Delta_{f_{\sol}}wd\nu_{\sol}\\
        &=-\int_{M_{{\sol}}}(\Delta_{f_{\sol}}\chi_r)wd\nu_{\sol}\\
        &=-\int_{M_{\sol}}(\chi'r^{-2}\Delta_{f_{\sol}}f_{\sol}+\chi''r^{-4}|\partial f_{\sol}|_{g_{\sol}}^2)wd\nu_{\sol}\\
        &\le C\int_{\{r^2\le f_{\sol}\le 2r^2\}}|w|d\nu_{\sol}.
    \end{split}
\end{equation}
Taking $r\to \infty$ and using the dominated convergence theorem gives $\int_{M_{\sol}}wd\nu_{\sol}=0$, from which a final integration by parts yields
\begin{equation*}
    0=\int_{M_{\sol}} \langle \nabla^{g_{\sol}}u_{PH},\nabla^{g_{\sol}}h\rangle_{g_{\sol}} d\nu_{\sol} = -2\int_{M_{\sol}} u_{PH}\Delta_{f_{\sol}} hd\nu_{\sol} = 2\int_{M_{\sol}} u_{PH}hd\nu_{\sol},
\end{equation*}
and the claim follows.
\end{proof}

Because any K\"ahler-Ricci shrinker satisfies a log Sobolev inequality, the spectrum of the drift Laplacian $\Delta_{f_{\sol}}$ is discrete \cite{XZ}. Moreover, we can use Lemma \ref{lem:lineartheory} to give a description of those eigenfunctions corresponding to small eigenvalues.

\begin{lemma} \label{lem:eigenvalues} Suppose that $(M_{\sol},J_{\sol},g_{\sol},f_{\sol})$ is a K\"ahler-Ricci shrinker which does not split a factor of $\mathbb{C}^k$ isometrically. Then there exist nonnegative integers $N_- \leq  N_0$ and an orthonormal basis $(\Psi_{\alpha})_{\alpha \in \mathbb{N}}$ of $L^2(M_{\sol},\nu_{\sol})$ such that 
\begin{equation*} \Delta_{f_{\sol}}\Psi_j =-\lambda_j \Psi_j, \end{equation*} where 
\begin{equation*} 0=\lambda_0<\lambda_1 \leq \cdots \leq \lambda_{N_-} \leq 1 < \lambda_{N_0+1} \leq  \lambda_{N_0+2}\leq \cdots, \end{equation*}
and $\lim_{\alpha \to \infty} \lambda_{\alpha}=\infty$. In addition, $\sqrt{-1}\partial \overline{\partial}\Psi_{\alpha}=0$ for $\alpha \leq N_-$, whereas $\lambda_{\alpha}=1$ and $\Psi_{\alpha} \in H$ whenever $N_- < \alpha \leq N_0$. 
\end{lemma}

\begin{proof}
     The existence of the orthonormal eigenbasis $(\Psi_{\alpha})_{\alpha \in \mathbb{N}}$ with $\lim_{\alpha \to \infty} \lambda_{\alpha}=\infty$ follows from \cite[Theorem 4]{XZ}, and the statements concerning $\Psi_{\alpha}$ for $\alpha \leq N_0$ follow from Lemma \ref{lem:lineartheory}.
\end{proof}

Using estimates from \cite{ConlonDeruelleSun,ColdingMinicozzi25}, we now obtain pointwise upper bounds for eigenfunctions of $\Delta_{f_{\sol}}$. 

\begin{lemma} \label{lem:eigenfunctiongrowth} Suppose that for some $\lambda \in [0,\infty)$, $v\in C^{\infty}(M_{\sol})\cap L^2(M_{\sol},\nu_{\sol})$ satisfies
\begin{equation} \label{eq:eigenfunction} \Delta_{f_{\sol}} v = -\lambda v .
\end{equation}
\begin{enumerate}
    \item If $(M_{\sol},g_{\sol})$ has bounded Ricci curvature and $\lambda \in [0,1]$, then there exists $C_0(g_{\sol}) \in (0,\infty)$ such that
    \label{lem:eigenfunctiongrowth0}
    \begin{equation*}
        |v|\leq C_0(g_{\sol})\left( \int_{M_{\sol}}v^2 d\nu_{\sol} \right)^{\frac{1}{2}}b_{\sol}^{2\lambda}.
    \end{equation*}

    \item If $(M_{\sol},g_{\sol})$ has bounded curvature, then there exists $C_0(g_{\sol},\lambda) \in (0,\infty)$ such that \label{lem:eigenfunctiongrowth1} 
    \begin{equation*} |v| +b_{\sol}^{-1}|\partial v|_{g_{\sol}} 
    \leq C_0(g_{\sol},\lambda) \left( \int_{M_{\sol}}v^2 d\nu_{\sol} \right)^{\frac{1}{2}}b_{\sol}^{2n-1+2\lambda}.
    \end{equation*}

    \item \label{lem:eigenfunctiongrowth2} If $\lambda \in [0,1]$ and $(M_{\sol},g_{\sol})$ is either AC or compact, then for each $k\in \mathbb{N}$, there exists $C_k(g_{\sol}) \in (0,\infty)$ such that the following hold:
    \begin{equation*}
    \sup_{M_{\sol}}\sum_{j=0}^k b_{\sol}^{j-2\lambda}|(\nabla^{g_{\sol}})^j v|_{g_{\sol}}\leq C_k(g_{\sol}) \left( \int_{M_{\sol}} v^2 d\nu_{\sol}\right)^{\frac{1}{2}}.
\end{equation*}
\end{enumerate}
\end{lemma}
\begin{proof}
\ref{lem:eigenfunctiongrowth0}
By Lemma \ref{lem:lineartheory}, we can write $v=u_{PH}+h$, where $h\in H$, $u_{PH}$ is pluriharmonic, $\Delta_{f_{\sol}}u_{PH}=-\lambda u_{PH}$, $\Delta_{f_{\sol}}h=-h$, and $h=0$ unless $\lambda=1$. Setting $X:= \nabla^{g_{\sol}}f_{\sol}$, we then have
\begin{equation*}
    -\lambda u_{PH}=\Delta_{f_{\sol}} u_{PH} = -\frac{1}{2}\mathcal{L}_X u_{PH}
\end{equation*}
 
 \begin{equation*}
     \mathcal{L}_{\frac{1}{2}X} \frac{u_{PH}}{f_{\sol}^{\lambda}} = \lambda \frac{u_{PH}}{f_{\sol}^{\lambda+1}}\left(f_{\sol}-|\partial f_{\sol}|_{g_{\sol}}^2 \right) = \lambda \frac{u_{PH}}{f_{\sol}^{\lambda+1}}\operatorname{R}_{g_{\sol}},
 \end{equation*}
 from which we can estimate
 \begin{equation*}
     \left|  \mathcal{L}_{\frac{1}{2}X} \frac{u_{PH}}{f_{\sol}^{\lambda}} \right| \leq \frac{C}{f_{\sol}}\left|\frac{u_{PH}}{f_{\sol}^{\lambda}}\right|.
 \end{equation*}
 It follows that for any $x\in b_{\sol}^{-1}(r_0)$, the function $\psi(s):= \frac{u_{PH}}{f_{\sol}^{\lambda}}(\phi_s(x))$ satisfies
 \begin{equation*}
     |\psi'(s)|\leq C(g_{\sol})e^{-\frac{s}{2}}|\psi(s)|.
 \end{equation*}
 It follows from integration in $s$ that 
 \begin{equation*}
     |u_{PH}|\leq C(g_{\sol})f_{\sol}^{\lambda}\sup_{\Omega(r_0)}|u_{PH}| \leq C(g_{\sol},u_{PH})f_{\sol}^{\lambda}.
 \end{equation*}
By \cite[Claim A.14]{ConlonDeruelleSun}, we similarly have  
\begin{equation*}
    |h|\leq C(g_{\sol},h)f_{\sol}.
\end{equation*}
It follows that $\|\cdot\|_{L^2(M_{\sol},\nu_{\sol})}$ and
\begin{equation*}
    \|u\|' := \sup_{M_{\sol}} \frac{|u|}{f_{\sol}^{\lambda}}
\end{equation*}
are two norms on the vector space
\begin{equation*}
    V_{\lambda} := \{u\in C^{\infty}(M_{\sol})\cap L^2(M_{\sol},\nu_{\sol}) \: | \: -\Delta_{f_{\sol}}u=\lambda u\}. 
\end{equation*}
However, each $V_{\lambda}$ is finite-dimensional by Lemma \ref{lem:eigenvalues}, so the claim follows from the fact that any two norms on a finite-dimensional vector space are equivalent.

\ref{lem:eigenfunctiongrowth1}
Because $(M_{\sol},g_{\sol})$ has bounded curvature, we can choose $r_0 \in (0,\infty)$ so that 
\begin{equation} \label{eq:nablababout1}
    \frac{1}{2} \leq |\nabla^{g_{\sol}}b_{\sol}|_{g_{\sol}} \leq 2 \qquad \text{ on } M_{\sol}\setminus \Omega(r_0).
\end{equation}
We recall the weighted $L^2$ average defined in \cite[(4.2)]{ColdingMinicozzi25} by 
\begin{equation*}
I(r)= \frac{1}{r^{2n-1}}\int_{b_{\sol}^{-1}(r)}v^2|\nabla b_{\sol}| d\sigma_{r}
\end{equation*}
for all $r\in [r_0,\infty)$, 
where $\sigma_r$ denotes the Riemannian volume measure of $g_{\sol}|_{b_{\sol}^{-1}(r)}$. 
After possibly further increasing $r_0$, \cite[Theorem 4.1]{ColdingMinicozzi25} gives
\begin{equation}\label{eq; CM25 average bound}
I(r_2)\leq C(\lambda,g_{\sol})\left( \frac{r_2}{r_1} \right)^{4\lambda}I(r_1)
\end{equation}
for all $r_2 \geq r_1 \geq r_0$ and all $v \in C^{\infty}(M_{\sol})\cap L^2(M_{\sol},\nu_{\sol})$ satisfying \eqref{eq:eigenfunction}.

Fix $x\in M_{\sol}$, and suppose $R_x:=b_{\sol}(x)\geq r_0$. Let also $r_x=b_{\sol}(x)^{-1}.$ Because $(M_{\sol},g_{\sol})$ has bounded curvature and $\sup_{B_{g_{\sol}}(x,r_x)}|\nabla f_{\sol}|_{g_{\sol}}\leq 2r_x^{-1}$, local elliptic regularity applied at scale $r_x$ (see \cite[Lemma 4.2]{LiZhang}) gives
\begin{equation*}
\sup_{B_{g_{\sol}}(x,\frac{r_x}{2})}v^2 +r_x^2 \sup_{B_{g_{\sol}}(x,\frac{r_x}{2})}|\partial v|_{g_{\sol}}^2\leq \frac{C(g_{\sol},\lambda)}{r_x^{2n}}\int_{B_{g_{\sol}}(x,r_x)}v^2 \omega_{\sol}^n.
\end{equation*}
If we possibly further increase $r_0=r_0(g_{\sol})$, then $B_{g_{\sol}}(x,r_x)\subseteq \Omega(R_x+2r_x)\setminus \overline{\Omega}(R_x-2r_x)$. Then \eqref{eq:nablababout1}, \eqref{eq; CM25 average bound}, and the co-area formula yield
\begin{align*}
    \frac{C}{r_x^{2n}}\int_{B_{g_{\sol}}(x,r_x)}v^2 \omega_{\sol}^n &\leq  \frac{C}{r_x^{2n}}\int_{\Omega(R_x+2r_x)\setminus \overline{\Omega}(R_x-2r_x)}v^2 \omega_{\sol}^n \leq \frac{C}{r_x^{2n}}\int_{R_x-2r_x}^{R_x+2r_x}\left( \int_{b_{\sol}^{-1}(s)}v^2|\nabla b_{\sol}|d\sigma_s\right)ds \\
    &\leq \frac{C}{r_x^{2n}}\int_{R_x-2r_x}^{R_x+2r_x} s^{2n-1}I(s)ds  \le C\frac{(R_x+2r_x)^{2n-1}}{r_x^{2n-1}}I(R_x+2r_x) \leq C\frac{R_x^{2n-1+4\lambda}}{r_x^{2n-1}}I(r_0)\\& \leq R_x^{2(2n-1+2\lambda)}I(r_0).
\end{align*}
Because $\inf_{\Omega(2r_0)}e^{-f_{\sol}} \geq C(g_{\sol})^{-1}$, a similar argument yields
\begin{equation*}
I(r_0)\leq C(g_{\sol},\lambda)\int_{\Omega(2r_0)}|v|^2 d\nu_{\sol}\leq C(g_{\sol},\lambda)\|v\|^2_{L^2(M_{\sol},\nu_{\sol})},
\end{equation*}
and the claim follows by combining estimates.

 \ref{lem:eigenfunctiongrowth2} 
We proceed by induction on $k\in \mathbb{N}$, with the case $k=0$ following from \ref{lem:eigenfunctiongrowth0}. Since the shrinker is now assumed to be AC or compact, there exist constants $C_m<\infty$ such that
\begin{equation} \label{eq:eigenvalueACcurvature}
|(\nabla^{g_{\sol}})^m \Rm(g_{\sol}) |_{g_{\sol}} \leq \frac{C_m}{b_{\sol}^{2+m}}
\end{equation}
on $M_{\sol}$. By combining \ref{lem:eigenfunctiongrowth0} and \eqref{eq:eigenvalueACcurvature} with standard elliptic estimates at scale $b_{\sol}(x)^{-1}$, we get that for every $j\in \N,$ there exists $N_j \in \mathbb{N}$ such that
\begin{equation} \label{eq:eigenfunctionverycoarse}
|(\nabla^{g_{\sol}})^j v|_{g_{\sol}}\leq C \| v\|_{L^2(M_{\sol},\nu_{\sol})}b_{\sol}^{N_j}.
\end{equation}
For $j\geq 1$, the Bochner formula yields (see \cite[Section 2]{ColdingMinicozzi25}):
\begin{align*}
\Delta_{f_{\sol}}(\nabla^{g_{\sol}})^j v
&=\left(\frac j2 -\lambda\right)(\nabla^{g_{\sol}})^jv + \sum_{p=0}^{j-2}(\nabla^{g_{\sol}})^p\Rm(g_{\sol}) \ast(\nabla^{g_{\sol}})^{j-p} v\\
&=  \left(\frac j2 -\lambda\right)(\nabla^{g_{\sol}})^j v  + \Rm(g_{\sol})\ast (\nabla^{g_{\sol}})^jv + E_j,
\end{align*}
where $|E_j|\leq C_j b_{\sol}^{2\lambda-j-2}\|v\|_{L^2(M_{\sol},\nu_{\sol})}$ by our induction hypothesis, and the term $\Rm(g_{\sol}) \ast (\nabla^{g_{\sol}})^j v$ is not present if $j=1$. Using Cauchy's inequality, we estimate
\begin{equation} \label{eq:ellipticineqhigher}
\Delta_{f_{\sol}} |(\nabla^{g_{\sol}})^j v|_{g_{\sol}}^2\geq 2|(\nabla^{g_{\sol}})^{j+1} v|_{g_{\sol}}^2+\left(j-2\lambda- \frac{C}{f_{\sol}} \right)|(\nabla^{g_{\sol}})^j v|_{g_{\sol}}^2 -C_j f_{\sol}^{2\lambda-j-1}\|v\|_{L^2(M_{\sol},\nu_{\sol})}^2.
\end{equation}
Combining \eqref{eq:ellipticineqhigher} with 
\begin{align*} \Delta_{f_{\sol}}f_{\sol}^{j-2\lambda} &= -(j-2\lambda)f_{\sol}^{j-2\lambda-1}(f_{\sol}-n) + (j-2\lambda)(j-2\lambda-1)|\partial f_{\sol}|_{g_{\sol}}^2 f_{\sol}^{j-2\lambda-2} \\
&\geq -(j-2\lambda)f_{\sol}^{j-2\lambda} -C_j f_{\sol}^{j-2\lambda-1}
\end{align*}
and using Cauchy's inequality again, we estimate
\begin{align*} 
    \Delta_{f_{\sol}} \left( f_{\sol}^{j-2\lambda}|(\nabla^{g_{\sol}})^j v|_{g_{\sol}}^2 \right) \geq & 2f_{\sol}^{j-2\lambda}|(\nabla^{g_{\sol}})^{j+1}v|_{g_{\sol}}^2 - C_jf_{\sol}^{j-2\lambda-1}|(\nabla^{g_{\sol}})^j v|_{g_{\sol}}^2 - C_j f_{\sol}^{-1}\|v\|_{L^2(M_{\sol},\nu_{\sol})}^2 \\
    &-C_j|(\nabla^{g_{\sol}})^{j+1}v|_{g_{\sol}}|(\nabla^{g_{\sol}})^jv|_{g_{\sol}} f_{\sol}^{j-2\lambda-\frac{1}{2}} \\
    \geq & f_{\sol}^{j-2\lambda}|(\nabla^{g_{\sol}})^{j+1}v|_{g_{\sol}}^2- C_jf_{\sol}^{j-2\lambda-1}|(\nabla^{g_{\sol}})^j v|_{g_{\sol}}^2 - C_j f_{\sol}^{-1}\|v\|_{L^2(M_{\sol},\nu_{\sol})}^2.
\end{align*}
Thus, for $A_j \in (1,\infty)$ chosen sufficiently large, the quantity
\begin{align*}
    Q_j:= f_{\sol}^{j-2\lambda}|\nabla^j v|^2 + A_j f_{\sol}^{j-1-2\lambda}|\nabla^{j-1}v|^2.
\end{align*}
satisfies (by the induction hypothesis)
\begin{align*}
\Delta_{f_{\sol}}Q_j \geq & f_{\sol}^{j-2\lambda-1}|(\nabla^{g_{\sol}})^j v|
_{g_{\sol}}^2 - A_jC_{j-1}f_{\sol}^{j-2\lambda-2}|(\nabla^{g_{\sol}})^{j-1}v|_{g_{\sol}}^2-C_j f_{\sol}^{-1}\|v\|_{L^2(M_{\sol},\nu_{\sol})}^2\\
\geq & \frac{Q_j}{f_{\sol}} - \frac{C(j)}{f_{\sol}}\|v\|_{L^2(M_{\sol},\nu_{\sol})}^2.
\end{align*}
On the other hand, for any $a>1$, we can estimate
\begin{equation*}
    \Delta_{f_{\sol}}f_{\sol}^{a} \leq -af_{\sol}^{a-1}(f_{\sol}-n) +a(a-1)f_{\sol}^{a-1} \leq -(a-1) f_{\sol}^{a}+C(a) f_{\sol}^{a-1}.
\end{equation*}
Choose $p_j > \max \{N_j+j-2\lambda,N_{j-1}+j-1-2\lambda\}$. 
By \eqref{eq:eigenfunctionverycoarse}, the quantity $Q_{j,\varepsilon}:=Q_j-\varepsilon f_{\sol}^{p_j}$ then achieves a global maximum at some point $x^{\ast}$, and the above inequalities yield
\begin{equation*}
   0\geq \Delta_{f_{\sol}}Q_{j,\epsilon}(x^{\ast})\geq \frac{Q_{j,\epsilon}(x^{\ast})}{f_{\sol}(x^{\ast})} - \frac{C(j)}{f_{\sol}(x^{\ast})}\|v\|_{L^2(M_{\sol},\nu_{\sol})}^2 + \epsilon(f_{\sol}^{p_j}(x^{\ast})-C_j f_{\sol}^{p_j-1}(x^{\ast})).
\end{equation*}
Rearranging terms, we obtain
\begin{equation*}
    \sup_{M_{\sol}} Q_{j,\epsilon} = Q_{j,\epsilon}(x^{\ast}) \leq C(j)\|v\|_{L^2(M_{\sol},\nu_{\sol})}^2 + \epsilon \sup_{M_{\sol}}(C_j f_{\sol}^{p_j}-f_{\sol}^{p_j+1}).
\end{equation*}
Letting $\varepsilon\to 0$ yields $Q_j \leq C_j\|v\|_{L^2(M_{\sol},\nu_{\sol})}^2$ on $M_{\sol}$, from which the claim follows.
\end{proof}

Lemma \ref{lem:eigenvalues} will be combined with the following $L^2$ expansion theorem to prove a decay estimate for the drift heat equation modulo gauge terms.

\begin{lemma} 
\label{lem:decayfordriftheateq}
Suppose $(M_{\sol},g_{\sol})$ has bounded curvature and $(\psi_t)_{t\in (t_-,t_+]}$ is a solution of 
\begin{equation*}
    \partial_t \psi_t = \Delta_{f_{\sol}}\psi_t + \psi_t
\end{equation*}
on $M_{\sol} \times (t_-,t_+)$ which satisfies
\begin{equation*}
    \sup_{t\in (t_-,t_+]} \int_{M_{\sol}} \psi_t^2 d\nu_{\sol} <\infty.
\end{equation*}
\begin{enumerate}
    \item \label{lem:decayfordriftheateq1} For any $t_0 \in (t_-,t_+]$, there exist $b_{\alpha} \in \mathbb{R}$ satisfying $\sum_{\alpha=0}^{\infty} b_{\alpha}^2 = \int_{M_{\sol}} \psi_{t_0}^2 d\nu_{\sol}$ and
\begin{equation*}
    \lim_{N\to \infty} \int_{M_{\sol}} \left|\psi_t - \sum_{\alpha=0}^{N}b_{\alpha}e^{(1-\lambda_{\alpha})(t-t_0)} \Psi_{\alpha} \right|^2 d\nu_{\sol} =0
\end{equation*}
for all $t\in (t_-,t_+]$.

    \item \label{lem:decayfordriftheateq2} The map $(t_-,t_+] \to (0,\infty),$  $t\mapsto \int_{M_{\sol}}\psi_t^2 d\nu_{\sol}$ is continuous. 
\end{enumerate}

\end{lemma}
\begin{proof} \ref{lem:decayfordriftheateq1} Because $\psi_{t} \in L^2(M_{\sol},\nu_{\sol})$ for each $t\in (t_-,t_+]$, it follows that
\begin{equation*}
    b_{\alpha}(t):= \int_{M_{\sol}} \psi_t \Psi_{\alpha} d\nu_{\sol}
\end{equation*}
satisfy $\sum_{\alpha=0}^{\infty}b_{\alpha}^2(t) =\int_{M_{\sol}} \psi_{t}^2 d\nu_{\sol}$ as well as 
\begin{equation*}
   \lim_{N\to \infty} \int_{M_{\sol}} \left| \psi_{t}-\sum_{\alpha=0}^N b_{\alpha}(t) \Psi_{\alpha}\right|^2 d\nu_{\sol}=0.
\end{equation*}
It remains to show that $b_{\alpha}(t)=e^{(1-\lambda_{\alpha})(t-t_0)}b_{\alpha}(t_0)$ for all $t\in (t_-,t_+]$. 

By local parabolic regularity theory, we have $\psi \in C^{\infty}(M_{\sol}\times (t_-,t_+])$. Choose a cutoff function $\eta \in C^{\infty}(\mathbb{R})$ such that $\eta|_{(-\infty,1]}\equiv 1$, $\eta|_{[2,\infty)}\equiv 0$, and $-2 \leq \eta '\leq 0$. We then define $\chi_r \in C_c^{\infty}(M_{\sol})$ for each $r>0$ by $\chi_r(x):=\eta(r^{-1}b_{\sol}(x))$. Recalling that $|\partial b_{\sol}|_{g_{\sol}}\leq 1$ and
\begin{equation*}
    |\Delta_{f_{\sol}}b_{\sol}| \leq \frac{|\Delta_{f_{\sol}}f_{\sol}|}{b_{\sol}} + \frac{1}{4}\frac{|\partial b_{\sol}^2|}{b_{\sol}^3} \leq C(g_{\sol})(b_{\sol}+1),
\end{equation*}
we obtain 
\begin{equation} \label{eq:boundsforcutoffs} r|\partial \chi_r|_{g_{\sol}}+|\Delta_{f_{\sol}} \chi_r|\leq C(g_{\sol})
\end{equation}
for all $r\in (0,\infty)$. Setting $\Lambda := \sup_{t\in (t_-,t_+]}\int_{M_{\sol}}\psi_t^2 d\nu_{\sol}$, we combine \eqref{eq:boundsforcutoffs} with Lemma \ref{lem:eigenfunctiongrowth} \ref{lem:eigenfunctiongrowth1} to estimate
\begin{align*}
    e^{-(\lambda_{\alpha}-1)(t-t_0)}\left| \frac{d}{dt}\left(e^{(\lambda_{\alpha}-1)(t-t_0)}\int_{M_{\sol}} \chi_r\psi_t\Psi_{\alpha}d\nu_{\sol} \right) \right|= &  \left| \int_{M_{\sol}} \psi_t \Psi_{\alpha}\Delta_{f_{\sol}}\chi_r d\nu_{\sol} +2\operatorname{Re}\int_{M_{\sol}}\psi_t \langle \partial \chi_r,\overline{\partial}\Psi_{\alpha}\rangle d\nu_{\sol}\right|\\
    \leq & C(g_{\sol})\sqrt{\Lambda} \sqrt{\nu_{\sol}(M_{\sol}\setminus \Omega(r))} \\& + C(g_{\sol},\lambda_{\alpha})\sqrt{\Lambda}r^{2n-1+2\lambda_{\alpha}}\sqrt{\nu_{\sol}(M_{\sol}\setminus\Omega(r))}\\
    \leq & \Psi(r^{-1}|\lambda_{\alpha},\Lambda).
\end{align*}
Integrating in time and taking $r\to \infty$, the dominated convergence theorem yields
\begin{equation*}
    e^{(\lambda_{\alpha}-1)(t-t_0)}b_{\alpha}(t)=b_{\alpha}(t_0)
\end{equation*}
for all $t\in (t_-,t_+]$, so the claim follows.

\ref{lem:decayfordriftheateq2} By \ref{lem:decayfordriftheateq1}, we have
\begin{equation*}
    \int_{M_{\sol}}\psi_t^2 d\nu_{\sol} = \sum_{\alpha=0}^{\infty} e^{2(1-\lambda_{\alpha})(t-t_0)}b_{\alpha}^2
\end{equation*}
for all $t\in (t_-,t_+]$. Because $t_0$ can be chosen arbitrarily close to $t_-$, and $\sum_{\alpha=0}^{\infty} b_{\alpha}^2<\infty$, the claim follows from the dominated convergence theorem for series.
\end{proof}

\section{Model metrics}

\label{section:modelmetrics}

Throughout this section, we assume that $(M_{\sol},g_{\sol},J_{\sol},f_{\sol})$
is K\"ahler-Ricci shrinker which is either AC or compact, and write $X:=\nabla^{g_{\sol}}f_{\sol}$.

Lemma \ref{lem:eigenfunctiongrowth} shows that for any $h\in H$, $|\nabla^{g_{\sol}}h|_{g_{\sol}}$ grows at most linearly, thereby ensuring the completeness of the vector field $\nabla^{g_{\sol}}h$. For any $h\in H$, we let $(\zeta_s^h)_{s\in \mathbb{R}}$ be the flow of $\frac{1}{2}\nabla^{g_{\sol}}h$. Following the strategy of \cite{ChiuSzek}, we now define an important collection of metrics which differ from $(M_{\sol},J_{\sol},g_{\sol},f_{\sol})$ by a holomorphic gauge change.

\begin{definition}[Model metrics]\label{defn; model metrics}
The family $\mathcal{F}=\{\omega_{h}\}_{h\in H}$
of model metrics corresponding to $(M_{\sol},J_{\sol},g_{\sol},f_{\sol})$ consists of the tuples $(M_{\sol},J_{\sol},g_h,f_h)_{h\in H}$ defined by
\begin{equation*}
    g_h := (\zeta_1^h)^{\ast}g_{\sol}, \qquad f_h := (\zeta_1^h)^{\ast}f_{\sol}.
\end{equation*}
\end{definition}

Next, we summarize the basic identities satisfied by model metrics, including the fact that any model metric is a K\"ahler-Ricci shrinker with soliton vector field $X$.

\begin{prop} \label{prop:modelidentities} Suppose $\{g_h,f_h\}_{h\in H}$ is the family of model metrics corresponding to $(M_{\sol},J_{\sol},g_{\sol},f_{\sol})$, and let $\omega_h$ be the corresponding K\"ahler form. For each $h\in H$, there then exists a unique $u_h \in C^{\infty}(M_{\sol})$ satisfying the following:
\begin{enumerate}

\item \label{prop:modelidentities1} $\omega_h = \omega_{\sol}+\sqrt{-1}\partial \overline{\partial}u_h$, 

\item \label{prop:modelidentities1.5}
$f_h = f_{\sol}+\frac{1}{2}X\cdot u_h$,

\item \label{prop:modelidentities2} $\mathcal{L}_{JX}u_h=\mathcal{L}_{JX}f_h=0$,

\item \label{prop:modelidentities3} $\nabla^{g_{h}}f_{h} = X$, 

\item \label{prop:modelidentities4} $e^{u_{h}-\frac{1}{2}Xu_{h}}\omega_{h}^{n}=\omega_{\sol}^{n}$,

\item \label{prop:modelidentities5}$\Ric(\omega_h)+\mathcal{L}_{\frac{1}{2}X}\omega_h = \omega_h$.

\end{enumerate}
\end{prop}
\begin{proof}
By $[X,\nabla^{g_{\sol}}h]=0$,
we get $(\zeta_{s}^{h})^{\ast}X=X$ for all $h\in H$ and $s\in\mathbb{R}$. For all $s\in \R$,
we compute
\begin{align*}
\frac{d}{ds}\omega_{sh}=  \frac{d}{ds}(\zeta_{s}^{h})^{\ast}\omega_{\sol}=&\frac{1}{2}(\zeta_{s}^{h})^{\ast}\mathcal{L}_{\nabla^{g_{\sol}}h}\omega_{\sol}\\
= & \frac{1}{2}(\zeta_{s}^{h})^{\ast}d\iota_{\nabla^{g_{\sol}}h}\omega_{\sol}\\
= & -\frac{1}{2}(\zeta_{s}^{h})^{\ast}dJdh\\
= & \sqrt{-1}\partial\overline{\partial}((\zeta_{s}^{h})^{\ast}h),
\end{align*}
so that integrating from $s=0$ to $s=1$ yields \ref{prop:modelidentities1}, 
where $u_{h}:=\int_{0}^{1}(\zeta_{s}^{h})^{\ast}hds$. Moreover, from this formula, we also have
\begin{equation*}
\begin{split}
    f_h-f_{\sol}&=(\zeta^h_1)^*f_{\sol}-f_{\sol}=\frac{1}{2}\int_{0}^1(\zeta^h_t)^*\langle\nabla^{g_{\sol}}h,\nabla^{g_{\sol}}f_{\sol}\rangle_{g_{\sol}}dt\\
    &=\int_{0}^1(\zeta^h_t)^*(\frac{X}{2}\cdot h)dt\\
    &=\frac{X}{2}\cdot \int_0^1(\zeta_t^h)^*hdt
    =\frac{X}{2}\cdot u_h.
    \end{split}
\end{equation*}
Since $(JX)\cdot h=0$, we have $\mathcal{L}_{JX}u_h=0$, hence also
\begin{equation*}
    \mathcal{L}_{JX}f_h = \frac{1}{2}\mathcal{L}_{JX}(Xu_h)=\frac{1}{2}[JX,X]u_h+\frac{1}{2}\mathcal{L}_X(\mathcal{L}_{JX}u_h)=0,
\end{equation*}
and \ref{prop:modelidentities2} follows.
We also have 
\begin{align*}
\iota_{JX}\omega_{h}= & \iota_{JX}(\omega_{\sol}+\sqrt{-1}\partial\overline{\partial}u_{h})=-df_{\sol}+\iota_{JX}dd^{c}u_{h}\\
= & -df_{\sol}+\mathcal{L}_{JX}(d^{c}u_{h})-d\iota_{JX}d^{c}u_{h}\\
= & -d\left(f_{\sol}+\frac{1}{2}X\cdot u_{h}\right)\\
= & -df_{h},
\end{align*}
so that \ref{prop:modelidentities3} holds. To prove \ref{prop:modelidentities4}, we use
\begin{equation*}
\partial_{s}u_{sh}=\partial_{s}\int_{0}^{1}(\zeta_{rs}^{h})^{\ast}(sh)dr=\partial_{s}\left(\int_{0}^{s}(\zeta_{\tau}^{h})^{\ast}hd\tau\right)=(\zeta_{s}^{h})^{\ast}h
\end{equation*}
to compute
\begin{align*}
\frac{d}{ds}(e^{u_{sh}-f_{sh}}\omega_{sh}^{n})= & ((\zeta_{s}^{h})^{\ast}h)e^{u_{sh}-f_{sh}}\omega_{sh}^{n}-\frac{1}{2}((\zeta_{s}^{h})^{\ast}(X\cdot h))e^{u_{sh}-f_{sh}}\omega_{sh}^{n}+\frac{1}{2}e^{u_{sh}-f_{sh}}(\zeta_{s}^{h})^{\ast}(\mathcal{L}_{\nabla^{g_{\sol}}h}\omega_{\sol}^{n})\\
= & (\zeta_{s}^{h})^{\ast}\left(\Delta_{g_{\sol}}h-\frac{1}{2}\langle\nabla^{g_{\sol}}f_{\sol},\nabla^{g_{\sol}}h\rangle_{g_{\sol}}+h\right)e^{u_{sh}-f_{sh}}\omega_{sh}^{n}
=  0.
\end{align*}
Integrating from $s=0$ to $s=1$ and using $u_{0}=0$, $f_{0}=f_{\sol}$
gives
\begin{equation*}
e^{u_{h}-f_{h}}\omega_{h}^{n}=e^{-f_{\sol}}\omega_{\sol}^{n},
\end{equation*}
which implies \ref{prop:modelidentities4}. Finally, \ref{prop:modelidentities5} follows by combining \ref{prop:modelidentities4} with the Ricci soliton equation for $\omega_{\sol}$. 
\end{proof}
\begin{corollary}\label{coro: uniform value of min f_h}
    If $x_0\in\arg\min f_{\sol}$, then $f_h(x_0)=f_{\sol}(x_0)$ for all $h\in H$.
\end{corollary}
\begin{proof}
    Since $x_0\in\arg\min f$, we have $X(x_0)=0$. By Proposition \ref{prop:modelidentities} \ref{prop:modelidentities1.5}, we have $f_h=f_{\sol}+\frac{1}{2}X\cdot u_h$ holds for all $h\in H$. Evaluating at $x_0$, we obtain $f_h(x_0)=f_{\sol}(x_0)$
for all $h\in H$.
\end{proof}
\begin{definition}
    For $h\in H$, we define $\nu_h := (2\pi)^{-n}e^{-f_h-W}dg_h$.
\end{definition}

The following lemma shows that the potential function $f_h$ is uniformly comparable to the background potential $f_{\sol}$ whenever $h$ is in a bounded region of $H$. 
\begin{lemma}\label{lem; equiv of potentials}
    There exists $C(g_{\sol}) \in (1,\infty)$ such that the following hold for all $h\in H$:
\begin{enumerate}
    \item \label{lem; equiv of potentials1}
    $e^{-C(g_{\sol})\|h\|_H} f_{\sol} \leq f_{h} \leq e^{C(g_{\sol})\|h\|_H}f_{\sol}$,
    \item \label{lem; equiv of potentials2} for all $x_0\in \arg\min f_{h}$, 
\begin{equation*}
    \frac{1}{2}(d_{g_{h}}(x_0,\cdot)-C(g_{\sol}))_+^2 \le f_{h}\le \frac{1}{2}(d_{g_{h}}(x_0,\cdot)+C(g_{\sol}))^2.
\end{equation*}
\end{enumerate}     
\end{lemma}
\begin{proof}
\ref{lem; equiv of potentials1}
By Lemma \ref{lem:eigenfunctiongrowth}, there exists $C(g_{\sol}) \in (1,\infty)$ such that for all $h\in H$,
\begin{align*}
    |\nabla^{g_{\sol}}h|_{g_{\sol}}\le C(g_{\sol}) \|h\|_H \sqrt{f_{\sol}}.
\end{align*}
Combining this with $|X|^2_{g_{\sol}}=|\nabla^{g_{\sol}}f_{\sol}|_{\sol} ^2\le 2f_{\sol}$ and the Cauchy-Schwarz inequality yields 

\begin{align*}
    \left| \frac{d}{dt}f_{\sol}(\zeta_t^h(x))\right|\leq  \frac{1}{2}|\langle X,\nabla^{g_{\sol}}h\rangle_{g_{\sol}}|(\zeta_t^h(x))\leq C(g_{\sol})\|h\|_H f_{\sol}(\zeta_t^h(x))
\end{align*}
for any fixed $x \in M_{\sol}$. Upon integration, we obtain
\begin{align*}
    f_h(x)=f_{\sol}(\zeta_1^h(x)) \leq e^{C(g_{\sol})\|h\|_H }f_{\sol}(x),
\end{align*}
and the remaining inequality follows by pulling back this inequality by $\zeta_{-1}^h$.

\ref{lem; equiv of potentials2}
For any $x_0 \in \operatorname{argmin}_{M_{\sol}}f_{\sol}$, Corollary \ref{coro: uniform value of min f_h} gives $\min_{M_{\sol}}f_h=f_h(x_0)=f_{\sol}(x_0)$ for all $h\in H$. The remaining claim thus follows from Theorem \ref{thm:basicKRSfacts} \ref{thm:basicKRSfacts1}, after pulling back by $\zeta_1^h$. 
\end{proof}

In the following, we summarize some important properties of model metrics which will be used in the sequel. We first begin with a technical lemma:
\begin{lemma}\label{tech lemma on lie derivative}
    Suppose $(Y_t)_{t\in [0,1]}$ is a smooth family of complete vector fields on a smooth manifold $M$, and for each $t\in [0,1]$, let $(\Phi_s^{Y_t})_{s\in\R}$ denote the flow of $Y_t$. For all $t\in [0,1]$ and $s\in \mathbb{R}$, define
    \begin{equation*}
        Z_s^t=\frac{\partial}{\partial t}\Phi_s^{Y_t}\circ \Phi_{-s}^{Y_t}.
    \end{equation*}
    Then we have
    \begin{equation*}
        Z_s^t=\int_0^s(\Phi_{s-\tau}^{Y_t})_{\ast}(\partial_t Y_t) d\tau.
    \end{equation*}
\end{lemma}
\begin{proof}
Define $\Phi: [0,1]\times \R\mapsto M$ by $\Phi(t,s)=\Phi_s^{Y_t}.$
Let $D_s,D_t$ denote covariant differentiation along $\Phi$ with respect to a Levi-Civita connection $\nabla$. Since the connection is torsion-free, we have $D_s(\frac{\partial}{\partial t}\Phi)=D_t(\frac{\partial}{\partial s}\Phi)$.
    Differentiating $\displaystyle{\frac{\partial}{\partial t}\Phi_s^{Y_t}=Z_s^t\circ\Phi_s^{Y_t}}$ with respect to $s$ gives us $D_s(\frac{\partial}{\partial t}\Phi)= D_s(Z_s^t \circ \Phi_s^{Y_t}).$
    On the other hand, we have
   \begin{equation*}
       D_t(\frac{\partial}{\partial s}\Phi)= D_t(Y_t\circ \Phi_s^{Y_t}).
   \end{equation*}
   Combining these gives us
   \begin{equation*}
       \frac{\partial }{\partial s}Z_s^t+[Y_t,Z_s^t]=\partial_tY_t, \quad \text{and}\quad    \frac{\partial}{\partial s}((\Phi_s^{Y_t})^*Z_s^t)=(\Phi_s^{Y_t})^*\partial_tY_t.
   \end{equation*}
   Therefore, we have
   \begin{equation*}
       (\Phi^{Y_t}_s)^*Z_s^t-Z_0^t=\int_{0}^s(\Phi_{\tau}^{Y_t})^*(\partial_tY_t)d\tau.
   \end{equation*}
   Since $Z_0^t=0$, the claim follows.
\end{proof}
\begin{prop} \label{prop:modelmetrics} Assume that $(M_{\sol},g_{\sol})$ is AC or compact.

\begin{enumerate}

    \item \label{modelmetrics0} For all $h\in H$, 
\begin{equation*} 
   e^{-C(g_{\sol})\|h\|_H} \omega_{\sol} \leq  \omega_h \leq e^{C(g_{\sol})\|h\|_H}\omega_{\sol}
\end{equation*}

\item \label{modelmetrics1} For all $h,k\in H$ with $\|h\|_H,\|k\|_H \leq 1$, 
    \begin{equation*} |u_h-u_k|\le C(g_{\sol})\|h-k\|_H f_{\sol},\end{equation*}

\item \label{modelmetrics2.5} For any $\ell_1,\ell_2 \in \mathbb{N}$, $m\in \mathbb{N}$, and $\mu \in \mathbb{R}$, there exists $C_{\ell_1,\ell_2,m}(g_{\sol},\mu) \in (1,\infty)$ such that for any smooth tensor $T$ of type $(\ell_1,\ell_2)$ on $M_{\sol}$, any $s\in [0,1]$, and any $h\in H$ with $\|h\|_H \leq 1$, we have 
\begin{equation*}
    \sum_{j=0}^m \sup_{M_{\sol}} f_{\sol}^{j-\mu}|(\nabla^{g_{\sol}})^j (\zeta_s^h)^{\ast}T|_{g_{\sol}}^2 \leq C_{\ell_1,\ell_2,m}(g_{\sol},\mu)\sum_{j=0}^{m} \sup_{M_{\sol}} f_{\sol}^{j-\mu}|(\nabla^{g_{\sol}})^jT|_{g_{\sol}}^2.
\end{equation*}

\item \label{modelmetrics3} 
for each $j\in\N$ and $h,k\in H$ with $\|h\|_H,\|k\|_H \leq 1$, 
\begin{equation*}
    \sup_{M_{\sol}}f_{\sol}^{\frac{j}{2}}|(\nabla^{g_{\sol}})^j(\omega_{h}-\omega_{k})|_{\omega_{\sol}}\leq C_j(g_{\sol})\|h-k\|_{H},
\end{equation*}

\item \label{modelmetrics4} there exists $C(g_{\sol})\in (1,\infty)$ such that for all $h,k\in H$ with $\|h\|_H,\|k\|_H \leq 1$,
\begin{equation*}
    |(u_{h}-u_{k})-(h-k)| \leq C(g_{\sol})(\|h\|_H + \|k\|_H)\|h-k\|_H f_{\sol}.
\end{equation*}
\end{enumerate}
\end{prop}
\begin{proof} \ref{modelmetrics0}
For any $t\in [0,1]$, we have
\begin{equation*} \partial_t \omega_{th}=(\zeta_t^h)^{\ast}\mathcal{L}_{\frac{1}{2}\nabla^{g_{\sol}}h}\omega_{\sol}=(\zeta_t^h)^{\ast}\sqrt{-1}\partial \overline{\partial}h .\end{equation*}
By Lemma \ref{lem:eigenfunctiongrowth} \ref{lem:eigenfunctiongrowth2}, we have
\begin{equation*}
     -C(g_{\sol})\|h\|_H\omega_{\sol}\le \sqrt{-1}\partial\bar\partial h\le C(g_{\sol})\|h\|_H\omega_{\sol}
\end{equation*}
for all $h\in H$, so combining expressions yields 
\begin{equation*}
   \partial_t (e^{-C(g_{\sol})\|h\|_Ht}\omega_{th}) \leq 0.
\end{equation*}
Upon integrating over $t\in [0,1]$, this yields $\omega_h \leq e^{C(g_{\sol})\|h\|_H}\omega_{\sol}$. After pulling back by $\zeta_1^h$, we obtain the remaining inequality.  

\ref{modelmetrics1}:
Recall that $u_h=\int_0^1(\zeta^h_s)^*hds$, where $(\zeta_s^h)_{s\in \mathbb{R}}$ is the flow of $\frac{1}{2}\nabla^{g_{\sol}}h$. Now we compute 
\begin{equation*}
   \begin{split}
        |u_h-u_k|&=\left|\int_0^1\left((\zeta^h_s)^*h-(\zeta^h_s)^*k+(\zeta^h_s)^*k-(\zeta^k_s)^*k\right)ds\right|\\
        &\le \int_0^1|(\zeta^h_s)^*(h-k)|ds+\int_0^1|(\zeta^h_s)^*k-(\zeta^k_s)^*k|ds\\
   \end{split}
\end{equation*}
By Lemma \ref{lem:eigenfunctiongrowth} and Lemma \ref{lem; equiv of potentials} \ref{lem; equiv of potentials1}, there is a uniform constant $C(g_{\sol})\in (1,\infty)$ such that for all $h\in H$ with $\|h\|_H \leq 1$, we have
\begin{equation*}
    |h|+\sqrt{f_{\sol}}|\nabla^{g_{\sol}}h|_{g_{\sol}}\le C\|h\|_{H}f_{\sol},\qquad \frac{1}{C(g_{\sol})}f_{\sol}\le f_h\le C(g_{\sol})f_{\sol}.
\end{equation*}
Hence, for all $t\in[0,1]$, we have
\begin{equation*}
    (\zeta^h_t)^*|h-k|\le C(g_{\sol})\|h-k\|_{H}f_{th}\le C(g_{\sol})\|h-k\|_H f_{\sol}.
\end{equation*}
To estimate the other term, we use that
\begin{equation*}
        (\zeta_s^h)^*k-(\zeta_s^k)^*k=\int_0^1\frac{\partial}{\partial \tau}\left((\zeta^{k+\tau(h-k)}_s)^*k\right)d\tau = \int_0^1 (\zeta^{k+\tau(h-k)}_s)^*(Z^\tau_s\cdot k)d\tau,
\end{equation*}
where $Z_s^{\tau}:=\frac{\partial}{\partial \tau}\zeta^{k+\tau(h-k)}_s\circ\zeta^{k+\tau(h-k)}_{-s}$. By Lemma \ref{tech lemma on lie derivative}, we have
\begin{equation} \label{eq:variation vector field model metrics}
    Z_s^\tau=\frac{1}{2}\int_0^s(\zeta_{s-\sigma}^{k+\tau(h-k)})_*(\nabla^{g_{\sol}}(h-k))d\sigma .
\end{equation}
 As a consequence of \ref{modelmetrics0}, for any tensor $T$ on $M_{\sol}$, and any $h\in H$, we can estimate
\begin{equation} \label{eq:usefuldoodad}
    |(\zeta_t^h)^{\ast}T|_{g_{\sol}}\leq C(g_{\sol}) |(\zeta_t^h)^{\ast}T|_{(\zeta_t^h)^{\ast}g_{\sol}}=C(g_{\sol})(\zeta_t^h)^{\ast}|T|_{g_{\sol}}.
\end{equation}
If $\|h\|_H,\|k\|_H \leq 1$, we thus have
\begin{align*}
    |(\zeta_{s-\sigma}^{k+\tau(h-k)})_{\ast}(\nabla^{g_{\sol}}(h-k))|_{g_{\sol}} (x) &\le C(g_{\sol})|\nabla^{g_{\sol}}(h-k)|_{g_{\sol}}(\zeta_{\sigma-s}^{k+\tau(h-k)}(x))\\
    & \leq C(g_{\sol})\|h-k\|_H \sqrt{f_{\sol}},
\end{align*}
where, again, we used Lemma \ref{lem:eigenfunctiongrowth} and Lemma \ref{lem; equiv of potentials} \ref{lem; equiv of potentials1} for the second inequality. Integrating yields

\begin{equation*}
    |Z_s^\tau|_{g_{\sol}}\leq 
    \frac{1}{2}
    Cs\|h-k\|_H \sqrt{f_{\sol}}. 
\end{equation*}
Therefore, for all $s\in[0,1]$ we have
\begin{equation*}
     |(\zeta_s^h)^*k-(\zeta_s^k)^*k|\le Cs\|k\|_H\|h-k\|_H f_{\sol}.
\end{equation*}
Hence, if $\|h\|_H,\|k\|_H\le 1$ we conclude that
\begin{equation*}
    |u_h-u_k|\le C\|h-k\|_H f_{\sol}.
\end{equation*}

\ref{modelmetrics2.5} We proceed by induction, with the case $m=0$ immediate from \eqref{eq:usefuldoodad}. Suppose now that $m\in\N^\times$ and that the claim holds with
$m$ replaced by $m-1$. Note that the same estimate holds for $s\in [-1,1]$ by replacing $h$ with $-h$. By homogeneity, we may assume that
\begin{equation}\label{eq:normalized weighted tensor norm}
\sum_{j=0}^{m}
\sup_{M_{\sol}}
f_{\sol}^{j-\mu}
\left|
    \bigl(\nabla^{g_{\sol}}\bigr)^jT
\right|_{g_{\sol}}^2
=1.
\end{equation}

Fix $x\in M_{\sol}$, set $x_s := \zeta_{-s}^h(x)$, $V:=\frac{1}{2}\nabla^{g_{\sol}}h$, $T_s:= (\zeta_s^h)^{\ast}T$, and define
\begin{equation*}\eta(s):=
\left|(\nabla^{g_{\sol}})^mT_s\right|_{g_{\sol}}^2(x_s)\end{equation*}
for all $s\in [0,1]$. 
Since $\partial_sT_s=\mathcal{L}_VT_s$, we compute
\begin{align*}
\eta'(s)
={}&
2\left\langle
    (\nabla^{g_{\sol}})^m\mathcal{L}_VT_s,
     (\nabla^{g_{\sol}})^mT_s
\right\rangle_{g_{\sol}}(x_s)
-
\left(
    V\cdot \left| (\nabla^{g_{\sol}})^mT_s\right|_{g_{\sol}}^2
\right)(x_s)
\\
={}&
2\left\langle
     (\nabla^{g_{\sol}})^m\mathcal{L}_V T_s
    -
     \nabla^{g_{\sol}}_V  (\nabla^{g_{\sol}})^mT_s,
     (\nabla^{g_{\sol}})^mT_s
\right\rangle_{g_{\sol}}(x_s).
\end{align*}
As usual, let $T'\ast T''$ denote a linear combination of
contractions of $T'$ and $T''$ with $g_{\sol}$ and
$g_{\sol}^{-1}$. Commuting covariant derivatives gives
\begin{align} (\nabla^{g_{\sol}})^m\mathcal{L}_VT_s
={}&
\nabla^{g_{\sol}}_V (\nabla^{g_{\sol}})^mT_s
+
\sum_{j=0}^{m}
(\nabla^{g_{\sol}})^{j+2}h
\ast
(\nabla^{g_{\sol}})^{m-j}T_s
\nonumber\\
&+
\sum_{j=0}^{m-1}
(\nabla^{g_{\sol}})^j
\left(
    \nabla^{g_{\sol}} h
    \ast
    \Rm(g_{\sol})
    \ast
    (\nabla^{g_{\sol}})^{m-1-j}T_s
\right),
\label{eq:commuted lie derivative}
\end{align}
so that $\xi(s):=\sqrt{\eta(s)}$ is a Lipschitz function satisfying the following for a.e. $s \in [0,1]$:
\begin{equation}
\label{eq:eta derivative expanded}
\begin{aligned}
|\xi'(s)|
\leq{}&
C_m
\sum_{j=0}^{m}
\left|(\nabla^{g_{\sol}})^{j+2}h\right|_{g_{\sol}}(x_s)
\left|(\nabla^{g_{\sol}})^{m-j}T_s\right|_{g_{\sol}}(x_s)
\\
&+
C_m
\sum_{j=0}^{m-1}
\sum_{\substack{a,b,c\geq 0\\a+b+c=j}}
\left|(\nabla^{g_{\sol}})^{a+1}h\right|_{g_{\sol}}(x_s)
\left|(\nabla^{g_{\sol}})^b\Rm(g_{\sol})\right|_{g_{\sol}}(x_s)
\left|
    (\nabla^{g_{\sol}})^{m-1-j+c}T_s
\right|_{g_{\sol}}(x_s).
\end{aligned}
\end{equation}
By Lemma
\ref{lem:eigenfunctiongrowth} \ref{lem:eigenfunctiongrowth2}, Lemma \ref{lem:basicACfacts} \ref{lem:basicACfacts3} and $\|h\|_H\leq 1$, we have
\begin{align*}
\left|(\nabla^{g_{\sol}})^{j+2}h\right|_{g_{\sol}}
\leq
C_j(g_{\sol})f_{\sol}^{-\frac{j}{2}}, \quad 
\left|(\nabla^{g_{\sol}})^{a+1}h\right|_{g_{\sol}}
\leq
C_a(g_{\sol})f_{\sol}^{\frac{1-a}{2}}, \quad
\left|(\nabla^{g_{\sol}})^b\Rm(g_{\sol})\right|_{g_{\sol}}
\leq
C_b(g_{\sol})f_{\sol}^{-\frac{b+2}{2}}.
\end{align*}
Moreover, the induction hypothesis and
\eqref{eq:normalized weighted tensor norm} imply that
\begin{equation}\label{eq:lower derivative pullback bound}
\left|(\nabla^{g_{\sol}})^rT_s\right|_{g_{\sol}}(x_s)
\leq
C_r(g_{\sol})
f_{\sol}(x_s)^{\frac{\mu-r}{2}}
\end{equation}
for every $r\in \{0,...,m-1\}$. The $j=0$ term in the first sum in
\eqref{eq:eta derivative expanded} is bounded by $C_m\xi(s)$, so combining expressions yields
\begin{equation*}
|\xi'(s)|
\leq
C_m(g_{\sol})
\left( \xi(s) + f_{\sol}^{\frac{\mu-m}{2}}(x_s)
\right) = C_m(g_{\sol})\left( \xi(s)+f_{-sh}^{\frac{\mu-m}{2}}(x) \right).
\end{equation*}
Integrating over $s\in [0,1]$ and using Lemma
\ref{lem; equiv of potentials} \ref{lem; equiv of potentials1}
gives
\begin{equation*}
\left|(\nabla^{g_{\sol}})^mT_s\right|_{g_{\sol}}(x_s)
\leq
C_m(g_{\sol})
f_{\sol}^{\frac{\mu-m}{2}}(x).
\end{equation*}
Since $x\in M_{\sol}$ was arbitrary, another application of Lemma \ref{lem; equiv of potentials}  \ref{lem; equiv of potentials1} proves the estimate
at order $m$, so the claim follows by induction.

\ref{modelmetrics3} For $t\in[0,1]$, set $h_t := k+t(h-k)$, and let $Z_s^t$ be as in the proof of \ref{modelmetrics1}. 
Lemma
\ref{lem:eigenfunctiongrowth}\ref{lem:eigenfunctiongrowth2}
then gives
\begin{equation}\label{eq:weighted derivatives gradient h minus k}
\sum_{j=0}^{m+1}
\sup_{M_{\sol}}
f_{\sol}^{j-1}
\left|
    \bigl(\nabla^{g_{\sol}}\bigr)^j
    \nabla^{g_{\sol}}(h-k)
\right|_{g_{\sol}}^2
\leq
C_{m+1}(g_{\sol})
\|h-k\|_H^2.
\end{equation}
Applying \ref{modelmetrics2.5} with
$\mu=1$, and then using Jensen's inequality in
\eqref{eq:variation vector field model metrics} yields
\begin{equation}\label{eq:weighted lie derivative estimate}
\sum_{j=0}^{m}
\sup_{M_{\sol}}
f_{\sol}^{j}
\left|
    \bigl(\nabla^{g_{\sol}}\bigr)^j
    \mathcal{L}_{Z_1^t}\omega_{\sol}
\right|_{g_{\sol}}^2\leq \sum_{j=0}^{m+1}
\sup_{M_{\sol}}
f_{\sol}^{j-1}
\left|
    \bigl(\nabla^{g_{\sol}}\bigr)^jZ_1^t
\right|_{g_{\sol}}^2
\leq
C_{m+1}(g_{\sol})
\|h-k\|_H^2,
\end{equation}
uniformly for $t\in[0,1]$. 
Applying \ref{modelmetrics2.5} with
$\mu=0$ and using
\eqref{eq:weighted lie derivative estimate}, we obtain
\begin{align*}
f_{\sol}^{\frac{m}{2}}
\left|
    \frac{\partial}{\partial t}
    \bigl(\nabla^{g_{\sol}}\bigr)^m\omega_{h_t}
\right|_{g_{\sol}}
=
f_{\sol}^{\frac{m}{2}}
\left|
    \bigl(\nabla^{g_{\sol}}\bigr)^m
    \left(\zeta_1^{h_t}\right)^*
    \mathcal{L}_{Z_1^t}\omega_{\sol}
\right|_{g_{\sol}}\leq
C_m(g_{\sol})\|h-k\|_H,
\end{align*}
so the claim follows by integrating over $t\in[0,1]$.

\ref{modelmetrics4} For $h\in H$, we compute
\begin{equation*}
    u_h-h=\int_0^1((\zeta_t^h)^*h-h)dt=\int_0^1\int_0^t(\zeta_s^h)^*|\partial  h|_{g_{\sol}}^2dsdt,
\end{equation*}
so that for any $h,k\in H$, 
\begin{equation*}
  \begin{split}
        (u_h-h)-(u_k-k)&=\int_0^1\int_0^t\left((\zeta_s^h)^*|\partial h|_{g_{\sol}}^2-(\zeta_s^k)^*|\partial k|_{g_{\sol}}^2\right)dsdt\\
        &=\int_0^1\int_0^t\left((\zeta_s^h)^*(|\partial h|_{g_{\sol}}^2-|\partial k|_{g_{\sol}}^2)+(\zeta_s^h)^*|\partial k|_{g_{\sol}}^2-(\zeta_s^k)^*|\partial k|_{g_{\sol}}^2\right)dsdt
  \end{split}
\end{equation*}
For the first term, we have
\begin{equation*}
\begin{split}
     \left|  |\partial h|_{g_{\sol}}^2-|\partial k|^2_{g_{\sol}}\right|&\le ( |\partial h|_{g_{\sol}}+ |\partial k|_{g_{\sol}})|\partial(h-k)|_{g_{\sol}}\\
     &\le C(\|h\|_H+\|k\|_H)\|h-k\|_H f_{\sol}
\end{split}
\end{equation*}
which implies the following for $\|h\|_H \leq 1$:
\begin{equation*}
    \int_0^1\int_0^t|(\zeta_s^h)^*|\partial h|_{g_{\sol}}^2-(\zeta_s^h)^*|\partial k|_{g_{\sol}}^2|dsdt\le C(\|h\|_H+\|k\|_H)\|h-k\|_H f_{\sol}.
\end{equation*}
For the remaining term, we use that if $\|k+r(h-k)\|_H \leq 1$ for all $r\in[0,1]$, then
\begin{equation*}
    \begin{split}
       \left| (\zeta_s^h)^*|\partial k|_{g_{\sol}}^2-(\zeta_s^k)^*|\partial k|_{g_{\sol}}^2\right|&=\left|\int_0^1\frac{\partial}{\partial r}\left((\zeta_s^{k+r(h-k)})^*|\partial k|_{g_{\sol}}^2\right)dr\right|\\
        &=\left|\int_0^1(\zeta_s^{k+r(h-k)})^*(Z_s^r\cdot|\partial k|_{g_{\sol}}^2)dr\right|\\
        &\le Cs \|h-k\|_H\|k\|_H^2 f_{\sol}.    
    \end{split}
\end{equation*}
If $h,k\in H$ satisfy $\|h\|_H,\|k\|_H \leq 1$, it follows that
\begin{equation*}
      |(u_h-h)-(u_k-k)|\le C\|h-k\|_H(\|h\|_H+\|k\|_H)f_{\sol}.
\end{equation*}
This completes the proof of \ref{modelmetrics4}.
\end{proof}

The following lemma will be used to estimate how $\Omega_h(D)$ changes under the flow of $\nabla^{g_{\sol}}f_{\sol}$ and when $h$ changes.

\begin{lemma} \label{lem:Dandrho} There is a function $D:(0,1)\to (1,\infty)$, depending only on $g_{\sol}$, such that for all $h,h'\in H$ with $\max\{\|h\|_H,\|h'\|_H\}<D(\rho)^{-1}$, the following hold on $M_{\sol}\setminus \Omega_{h'}(D(\rho))$:
\begin{enumerate}
    \item \label{lem:Dandrho1} $(1-\rho) f_{h} \leq |\partial f_h|_{g_h}^2 \leq (1+\rho)f_h$,

    \item \label{lem:Dandrho2} $\sqrt{-1}\partial \overline{\partial}f_h \geq (1-\rho)\omega_h$,

    \item \label{lem:Dandrho3} $(1-\rho)d_{g_{h'}}(x_0,\cdot) \leq b_h \leq (1+\rho)d_{g_{h'}}(x_0,\cdot)$.
\end{enumerate}
Moreover, for any $D\geq D(\rho)$ and $s\geq 0$, we have
\begin{equation} \label{eq:containmentofsublevelsets}
    \Omega_{h}(e^{\frac{1-\rho}{2}s}D) \subseteq \phi_s(\Omega_h(D)) \subseteq \Omega_{h}(e^{\frac{1+\rho}{2}s}D), \qquad \Omega_{h'}((1-\rho)D) \subseteq \Omega_h(D) \subseteq \Omega_{h'}((1+\rho)D).
\end{equation}
\end{lemma}
\begin{proof} Let $X = \nabla^{g_{\sol}}f_{\sol}$ be the soliton vector field, and let $(\phi_s)_{s\in \mathbb{R}}$ be the flow of $\frac{1}{2}X$. Recall that for the model metrics, we have 
\begin{equation*}
g_h= (\zeta_1^h)^{\ast}g_{\sol}, \qquad  f_h=(\zeta_1^h)^{\ast}f_{\sol}, \qquad X=\nabla^{g_h}f_h.
\end{equation*} Combining \eqref{eq:solitonnormalization}, Lemma \ref{lem:basicACfacts} \ref{lem:basicACfacts3}, and Theorem \ref{thm:basicKRSfacts} \ref{thm:basicKRSfacts1}
we have
\begin{equation*}
\lim_{x\to \infty}\frac{|\partial f_{\sol}|^2}{f_{\sol}}(x) = 1.
\end{equation*}
Therefore, given any $\mu \in (0,\frac{1}{2})$, we can find $D_0(\mu)\ge 1$ large so that 
\begin{equation*}
(1-\mu) f_{\sol} \leq |\partial f_{\sol}|_{g_{\sol}}^2 \leq (1+\mu)f_{\sol}
\end{equation*}
holds on $M_{\sol}\setminus \Omega(D_0(\mu))$. Also using the quadratic curvature decay, it is then straightforward to get
\begin{equation*}
\sqrt{-1}\partial\bar\partial f_{\sol} = \omega_{\sol} - \Ric(\omega_{\sol})\ge (1-\mu)\omega_{\sol}
\end{equation*}
on the same region. Pulling back by $\zeta_1^h,$ these identities also hold for $(g_h,f_h)$ on a region where $f_h$ is sufficiently big. To prove this on a region defined by $f_{h'},$ we use Lemma \ref{lem; equiv of potentials}: for $\|h\|_H$ and $\|h'\|_H$ sufficiently small, we have
\begin{equation*}
(1-C\delta)f_{\sol} \leq f_h \leq (1+C\delta)f_{\sol}, 
\end{equation*}
with the same holding for $f_{h'}.$ 
Given the above, given $\rho\in (0,1)$ and $D_0(\rho)>1$ large enough, we can find $D(\rho) \in (1,\infty)$ such that for $\max\{\|h\|_H,\|h'\|_H\}<D(\rho)^{-1}$, $f_{h'}\geq \frac{1}{2}D(\rho)^2,$ implies $f_h\geq \frac{1}{2}D_0(\mu)^2,$ with $\mu=\mu(\rho)>0$ small enough so that \ref{lem:Dandrho1}, \ref{lem:Dandrho2} hold. 
Similarly, we may combine Lemma \ref{lem; equiv of potentials} \ref{lem; equiv of potentials1}, \ref{lem; equiv of potentials2} to obtain \ref{lem:Dandrho3}, after possibly further increasing $D(\rho)$.

We now prove \eqref{eq:containmentofsublevelsets}. Since
\(X=\nabla^{g_h} f_h\) and \(\phi_s\) is generated by \(\frac{1}{2} X\), we
have
\begin{equation} \label{eq:fhunderflow}
\frac{d}{ds} f_h(\phi_s(x))
=
\frac{1}{2} X \cdot f_h(\phi_s(x))
=
|\partial f_h|_{g_h}^2(\phi_s(x)).
\end{equation}
By \ref{lem:Dandrho1}, after increasing \(D(\rho)\) if necessary, whenever
\(b_h(\phi_s(x))\ge D_0(\rho)\) we have
\[
(1-\rho) f_h(\phi_s(x))
\le
\frac{d}{ds} f_h(\phi_s(x))
\le
(1+\rho) f_h(\phi_s(x)).
\]
Thus, for all \(D\ge D_0(\rho)\) and all \(x\in \Omega_h(D)\), applying Gronwall's
inequality, 
\[
f_h(\phi_s(x))
\le
e^{(1+\rho)s} f_h(x)
\le
e^{(1+\rho)s}\frac{D^2}{2},
\]
which gives \(b_h(\phi_s(x)) \le e^{\frac{1+\rho}{2}s}D.\) Hence,
\begin{equation}
    \phi_s(\Omega_h(D))
\subset
\Omega_h\left(e^{\frac{1+\rho}{2}s}D\right).
\end{equation}
For the reverse inclusion, suppose \(x\notin \Omega_h(D)\), so that by \eqref{eq:fhunderflow}, we have \(b_h(\phi_\tau(x))\ge b_h(x)\ge D(\rho)\) for all $\tau \in [0,s]$. Hence, we can use \ref{lem:Dandrho1} again to obtain
\begin{equation*}
\frac{d}{d\tau} f_h(\phi_\tau(x))
\ge
(1-\rho)f_h(\phi_\tau(x)).
\end{equation*}
Integrating from 0 to $s$ gives $f_h(\phi_s(x)) \ge e^{(1-\rho)s}f_h(x)$, so that $\phi_s(x)\notin
\Omega_h\left(e^{\frac{1-\rho}{2}s}D\right)$, hence
\begin{equation*}
\Omega_h\left(e^{\frac{1-\rho}{2}s}D\right)
\subset
\phi_s(\Omega_h(D)).
\end{equation*}
Applying Lemma \ref{lem; equiv of potentials}, the second set of inclusions in \eqref{eq:containmentofsublevelsets} then follows analogously.
\end{proof}

\section{K\"ahler-Ricci flow near asymptotically conical shrinkers}

\label{section:KRFnearAC}

Throughout this section, we assume that $(M_{\sol},J_{\sol},g_{\sol},f_{\sol})$ is a K\"ahler-Ricci shrinker which is either AC or compact, and write $X:= \nabla^{g_{\sol}}f_{\sol}$.

Let $(\phi_s)_{s\in \mathbb{R}}$ be the flow of $\frac{1}{2}X$, and let $(g_h)_{h\in H}$ be the collection of model metrics corresponding to $(M_{\sol},J_{\sol},g_{\sol},f_{\sol})$. Then $g_{h,t}:= |t|\phi_{\log \frac{1}{|t|}}^{\ast}g_{h}$ are the corresponding K\"ahler-Ricci flows. Recall that $\omega_{h}=\omega_{\sol}+\sqrt{-1}\partial \overline{\partial} u_h$ for some $u_h \in C^{\infty}(M_{\sol})$, and $f_h = f_{\sol}+\frac{1}{2}X\cdot u_h$ is the corresponding soliton potential. Setting $u_{h,t}:= |t|\phi_{\log \frac{1}{|t|}}^{\ast}u_h$, it follows that $\omega_{h,t}=\omega_{\sol,t}+\sqrt{-1}\partial \overline{\partial}u_{h,t}$. We also set $f_{h,t}:=\phi_{\log \frac{1}{|t|}}^{\ast}f_{h}$, $b_{h,t}:=\sqrt{2|t|f_{h,t}}$, and
\begin{equation*}
    \Omega_{h,t}(D):= \{x\in M_{\sol} \: | \: b_{h,t}(x)\leq D\}.
\end{equation*}

We will later need the following to compare $b_h$ and $b_{h,t}$ for $t\in [-1,0)$. 

\begin{lemma} \label{lem:comparingbandd}
  There exists $c_1 \in \mathbb{R}$ such that for all $t\in [-1,0)$ and $h\in H$, the following holds on $M_{\sol}$:  \begin{equation*}
    b_{h}\le d_{g_{h,t}}(x_0,\cdot)+c_1.
\end{equation*}
\end{lemma}
\begin{proof} This follows from Lemma \ref{lem:primitivecomparingbandd} after pulling back by $\zeta_1^h$. 
\end{proof}

The proof of the following lemma is similar to \cite[Proposition 6.2]{FangLi}, with some modifications for our specific situation.

\begin{lemma} \label{lem:unnormalizedpseudolocality} For any $\epsilon\in (0,1)$, the following holds if
$\delta \leq \overline{\delta}(\epsilon,g_{\sol})$ and $D\geq \underline{D}(\epsilon,g_{\sol})$. Suppose $(M,(g_t)_{t\in [-\delta^{-1},0]})$ is a compact solution of the K\"ahler-Ricci flow, $h\in H$, and $\psi:\Omega_h(D) \to M$ is an open holomorphic embedding satisfying 
\begin{equation*}
    \|\psi^{\ast}g_{-1}-g_{h,-1}\|_{C^2(\Omega_h(D),g_{h,-1})} <\delta.
\end{equation*}
Then for all $t\in [-\epsilon^{-1},-\epsilon]$, we have 
\begin{equation*}
\|\psi^{\ast}g_t-g_{h,t}\|_{C^{\lfloor \epsilon^{-1} \rfloor}(\Omega_h((1-\epsilon)D),g_{h,t})}<\epsilon.
\end{equation*}
\end{lemma}

\begin{proof}
We first assume that the lemma is true when $h=0$, and show that this implies the result for any $h\in H$. Recalling that $g_{h,t}= (\zeta_1^h)^*g_{\sol,t}$ and $f_{h,t} = (\zeta_1^h)^{\ast}f_{\sol,t}$ for all $t<0$, we may define $\Phi:= \psi \circ (\zeta_1^h)^{-1}:\Omega(D)\to M$. Note that we have
\begin{align*}
    \Phi^* g_{-1} - g_{\sol,-1}= ((\zeta_1^h)^{-1})^*(\psi^*g_{-1}-g_{h,-1}),
\end{align*}
and since $\zeta_1^h:(M_{\sol},g_{h,-1})\to (M_{\sol},g_{\sol,-1})$ is an isometry, we have
\begin{align*}
    \| \Phi^*g_{-1}- g_{\sol,-1}\|_{C^2(\Omega(D),g_{\sol,-1})}\leq \delta.
\end{align*}
By our assumption that the lemma holds when $h=0$, 
\begin{align*}
    \sup_{t\in [-\varepsilon^{-1},-\varepsilon]}\|\Phi^{\ast}g_t-g_{\sol,t}\|_{C^{\lfloor \epsilon^{-1} \rfloor}(\Omega((1-\epsilon)D),g_{\sol,t})}<\epsilon.
\end{align*}
Now observing that $\psi^*g_t= (\zeta_1^{h})^*\Phi^*g_t$ and $g_{h,t}=(\zeta_1^{h})^*g_{\sol,t},$ we get
\begin{align*}
    \sup_{t\in [-\varepsilon^{-1},-\varepsilon]}\|\psi^*g_t-g_{h,t}\|_{C^{\lfloor \varepsilon^{-1}\rfloor}(\Omega_h((1-\varepsilon)D),g_{h,t})} \leq \varepsilon.
\end{align*}

It therefore suffices to prove the Lemma for $h=0$. Suppose by way of contradiction that the lemma does not hold, so that for some $\varepsilon>0$, there are $D_i\to \infty,$ $\delta_i\searrow 0,$ compact K\"ahler-Ricci flows $(M_i,(g_{i,t})_{t\in [-\delta_i^{-1},0]})$, and open holomorphic embeddings $\psi_i:\Omega(D_i)\to M_i$ such that $\widetilde{g}_{i,t}:=\psi_i^{\ast}g_{i,t}$ satisfy 
\begin{align} \label{eq:assumeclose}
\|\widetilde{g}_{i,-1}-{g}_{\sol,-1}\|_{C^2(\Omega(D_i),g_{\sol,-1})}&<\delta_i, \\ \label{eq:contradictfar} \sup_{t\in [-\epsilon^{-1},-\epsilon]} \|\widetilde{g}_{i,t}-g_{\sol,t}\|_{C^{\lfloor \epsilon^{-1} \rfloor}(\Omega((1-\epsilon)D_i),g_{\sol,t})}&\geq \epsilon.
\end{align}
By \eqref{eq:assumeclose} and \eqref{eq:quadraticcurvaturedecay}, there exists a constant $C>0$ such that for all $x\in \Omega((1-\epsilon^2)D_i)$, we have 
\begin{equation} \label{eq:nearbymetriccurvaturedecay}
    \sup_{B_{\widetilde{g}_{i,-1}}(x,r_x)}|\Rm(\tilde g_{i,-1})|_{\tilde g_{i,-1}}\le C\left( \delta_i+\frac{1}{b_{\sol}^2(x)}\right),\quad r_x:=\frac{\varepsilon^2}{2}\min\{b_{\sol}(x),\delta_i^{-1/2}\}
\end{equation}
as well as 
\begin{equation}
     \label{eq:nearbymetricvolume} \frac{\operatorname{Vol}_{\widetilde{g}_{i,-1}}(B_{\widetilde{g}_{i,-1}}(x,r_x))}{r_x^{2n}} \geq C(g_{\sol})^{-1}.
\end{equation}
We therefore apply \cite[Theorem 2.47]{Bam3} and \cite[Theorem 1.2]{LuPseudolocal} to obtain 
\begin{equation} \label{eq:smallcurvatureonlargeannulus}
     \sup_{t\in [-2\epsilon^{-1},-\frac{1}{2}\epsilon]} \sup_{\Omega((1-\epsilon^2)D_i)\setminus \overline{\Omega}(\Lambda)} |\Rm(\tilde g_{i,t})|_{\tilde g_{i,t}}\le \Psi(i^{-1},\Lambda^{-1}|\epsilon),
\end{equation}
for sufficiently large $\Lambda>1$.
By combining \eqref{eq:smallcurvatureonlargeannulus} with Shi's local derivative estimates for Ricci flow (for example, applying \cite[Theorem 14.14]{ChowII} to parabolic rescalings), we obtain
\begin{equation}\label{eq:smallderivativesonlargeannulus}
\sup_{t\in [-\epsilon^{-1},-\epsilon]} \sup_{\Omega((1-\epsilon^{\frac{3}{2}})D_i)\setminus \overline{\Omega}(\Lambda)}
      |(\nabla^{\tilde g_{i,t}})^k\Rm(\tilde g_{i,t})|_{\tilde g_{i,t}}\le \Psi(i^{-1},\Lambda^{-1}|k,\epsilon).
\end{equation}
Integrating \eqref{eq:smallcurvatureonlargeannulus} yields
\begin{equation}\label{eq:closemetricsonlargeannulus} \sup_{t\in [-2\epsilon^{-1},-\frac{1}{2}\epsilon]} \sup_{\Omega((1-\epsilon^2)D_i)\setminus \overline{\Omega}(\Lambda)} |\widetilde{g}_{i,t}-g_{\sol,t}|_{\tilde g_{i,t}}\le \Psi(i^{-1},\Lambda^{-1}|\epsilon).
\end{equation}
By \eqref{eq:smallderivativesonlargeannulus} and \eqref{eq:closemetricsonlargeannulus}, we can apply \cite[Lemma A.2]{BamKlein} to
$\widetilde{g}_i$ to obtain
\begin{equation}\label{eq: behavior at infinity lemma 5.1}
\sup_{t\in [-\epsilon^{-1},-\epsilon]} \sup_{\Omega((1-\epsilon)D_i)\setminus \overline{\Omega}(\frac{3}{2}\Lambda)}
    |(\nabla^{g_{\sol,t}})^k (\widetilde g_{i,t}-g_{\sol,t})|_{{g_{\sol,t}}}\le \Psi(i^{-1},\Lambda^{-1}|k,\epsilon).
\end{equation}
By \eqref{eq: behavior at infinity lemma 5.1} and \eqref{eq:contradictfar}, we can choose $\Lambda \geq \underline{\Lambda}(\epsilon)$ to guarantee that for $i\geq \underline{i}(\epsilon)$ sufficiently large, 
\begin{equation} \label{eq:contradictinginequality}
\sup_{t\in [-\epsilon^{-1},-\epsilon]}\|\widetilde{g}_{i,t}-g_{\sol,t}\|_{C^{\lfloor \epsilon^{-1} \rfloor}(\Omega(2\Lambda),g_{\sol,t})}\geq \epsilon.
\end{equation}
Define
\begin{equation*}
    \tau := \inf \left\{ s \in (0,1] \: | \: \limsup_{i \to \infty} \sup_{t \in [-s^{-1},-s]} \sup_{\Omega(2\Lambda)}|\Rm(\widetilde{g}_{i,t})|_{\widetilde{g}_{i,t}}<\infty\right\} .
\end{equation*}
By \cite[Theorem 2.47]{Bam3}, \cite[Theorem 10.3]{Perelman1}, \eqref{eq:nearbymetriccurvaturedecay}, and \eqref{eq:nearbymetricvolume}, we obtain $\tau <1$. If $\tau=0$, then we are done, so assume $\tau>0$. Integrating in time and arguing as in the previous paragraph then gives
\begin{equation} \label{eq:weakderivativeestimates} \sup_{t\in [-(\tau')^{-1},-\tau']} \sup_{\Omega((1-\epsilon)D_i)\setminus \overline{\Omega}(\Lambda)}
    |(\nabla^{g_{\sol,t}})^k\tilde g_{i,t}|_{{g_{\sol,t}}}\le C_k(\tau')
\end{equation}
for each $\tau' \in (\tau,1]$. 
By \eqref{eq: behavior at infinity lemma 5.1}, \eqref{eq:weakderivativeestimates}, and the Arzel\`a-Ascoli theorem (see, for instance, \cite[Chapter 3]{ChowI}), we may pass to a subsequence in order to assume that $\widetilde{g}_{i,t} \to \widetilde{g}_{\infty,t}$ in $C_{\operatorname{loc}}^{\infty}(M_{\sol} \times (-\tau^{-1},-\tau))$, where $(\widetilde{g}_{\infty,t})_{t\in (-\tau^{-1},-\tau)}$ is a complete K\"ahler-Ricci flow on $M_{\sol}$ with bounded curvature on compact time intervals. Because $\widetilde{g}_{i,-1}\to g_{\sol,-1}$ by \eqref{eq:assumeclose}, we have $\widetilde{g}_{\infty,-1}=g_{\sol,-1}$, hence the 
two-sided uniqueness of Ricci flows (see \cite{ChenZhu} and \cite{Kotschwar1}) implies $\widetilde{g}_{\infty,t}=g_{\sol,t}$ for all $t\in (-\tau^{-1},-\tau)$. By smooth convergence, for any $\tau'\in (\tau,1]$, we have 
\begin{equation} \label{eq:consequenceofclosenesstosoliton}
    \limsup_{i\to \infty} \sup_{t\in [-(\tau')^{-1},-\tau']} \sup_{\Omega(\frac{3}{2}\Lambda)} |\Rm(\widetilde{g}_{i,t})|_{\widetilde{g}_{i,t}} \leq (\tau')^{-1}\sup_{M_{\sol}}|\Rm(g_{\sol})|_{g_{\sol}}.
\end{equation}
If $\tau>0$, we could choose $\tau'>\tau$ sufficiently close to $\tau$, and then combine \eqref{eq:consequenceofclosenesstosoliton} with \cite{Perelman1} and \cite{Bam3} to obtain a contradiction. Thus $\tau=0<\epsilon$, so that we can take $\tau'<\epsilon$ in \eqref{eq:consequenceofclosenesstosoliton}. Repeating the above argument shows that we can pass to a further subsequence such that $\widetilde{g}_{i,t} \to g_{\sol,t}$ in $C_{\operatorname{loc}}^{\infty}(M_{\sol} \times [-\epsilon^{-1},-\epsilon])$, finally contradicting \eqref{eq:contradictinginequality}.
\end{proof}

In the following, we use the argument of \cite{FangLi} to obtain bounds for the conjugate heat kernel for K\"ahler-Ricci flows near a model metric.

\begin{lemma} \label{lem:comparabilityofpotentials}
For any AC K\"ahler-Ricci shrinker $(M_{\sol},J_{\sol},g_{\sol},f_{\sol})$ and any $\sigma \in (0,1)$ and $A,Y \in (1,\infty)$, the following holds if $D \geq \underline{D}(g_{\sol},A,Y,\sigma)$ and $\delta \leq \overline{\delta}(g_{\sol},A,Y,\sigma)$. 

Suppose that for some $r>0$, $(M,(g_t)_{t\in [-\delta^{-1}r^2,0)})$ is a compact solution of the K\"ahler-Ricci flow and $(\nu_t)_{t\in [-\delta^{-1}r^2,0)}$ is a conjugate heat kernel based at the singular time with $\mathcal{N}_{\nu}(\delta^{-1}r^2)\geq -Y$. Also assume that $h\in H$ and $\psi: \Omega_h(D) \to M$ is an open holomorphic embedding satisfying 
\begin{equation*}
    \|r^{-2}\psi^{\ast}g_{-r^2}-g_{h,-1}\|_{C^2(\Omega_h(D),g_{h,-1})} < \delta, \qquad f_{-r^2}(\psi(x_0))\leq A.
\end{equation*}
Then the following hold:

\begin{enumerate}
    \item \label{eq:comparabilityofpotentials} for all $t\in [-9,-\frac{1}{9}]$ and $x\in \Omega_h((1-\sigma)D)$,
\begin{equation*}
(1+\sigma)^{-1}f_{h,t}(x) -C(\sigma,A,Y,g_{\sol}) \leq (\psi^{\ast}f_{r^2t})(x) \leq (1+\sigma)f_{h,t}(x) + C(\sigma,A,Y,g_{\sol}),  
\end{equation*}

\item \label{eq:comparabilitytodistance} for all $t\in [-9,-\frac{1}{9}]$ and $y\in \psi \left(\Omega_h((1-\sigma)D) \right)$,
\begin{equation*}
(1+\sigma)^{-1}\frac{d_{g_{r^2t}}^2(y,\psi(x_0))}{2r^2|t|}-C(\sigma,A,Y,g_{\sol}) \leq f_{r^2 t}(y) \leq (1+\sigma)\frac{d_{g_{r^2t}}^2(y,\psi(x_0))}{2r^2|t|}+C(\sigma,A,Y,g_{\sol}).
\end{equation*}

\item \label{eq:comparabilitytoheatreg} If $\overline{f} \in C^{\infty}(M\times [-9r^2,-\frac{1}{9}r^2])$ is defined by $\overline{f}_{-9r^2}=f_{-9r^2}$ and $(\partial_t -\Delta_{\omega_t})(|t|\overline{f}_t)=0$, then for all $t\in [-8,-\frac{1}{8}]$ and $y\in \psi(\Omega_h((1-\sigma)D))$, we have
\begin{equation*}
(1+\sigma)^{-1}f_{r^2t}(y) -C(\sigma,A,Y,g_{\sol}) \leq \overline{f}_{r^2 t}(y) \leq (1+\sigma)f_{r^2 t}(y)+C(\sigma,A,Y,g_{\sol}).
\end{equation*}
\end{enumerate}
\end{lemma}
\begin{proof} Write $d\nu_t(x)=K(x,t)dg_t(x)=(2\pi|t|)^{-n}e^{-f_t(x)}dg_t(x)$. By rescaling, we may assume $r=1$. We may also assume $h=0$ after composing with the biholomorphism $\zeta_1^h$ used to define $g_h$. By \cite[proof of Claim 6.4]{FangLi}, $(\psi(x_0),t)$ is an $H(A,Y,g_{\sol})$-center of $\nu$ for all $t\in [-9,-\frac{1}{9}]$. By \cite[Theorem 2.13]{FangLi}, we therefore have
\begin{align}\label{heat kernel bound; FangLi}
    K(x,t)\leq \frac{C(A,Y,\sigma,g_{\sol})}{|t|^{n}}\exp\left( -\frac{d_{g_t}^2(\psi(x_0),x)}{2(1+\sigma^4)|t|}\right),
\end{align}
for all $(x,t)\in M\times [-9, -\frac{1}{9}]$, or equivalently  
\begin{align*}
    f_t(x)\geq \frac{d_{g_t}^2(\psi(x_0),x)}{2(1+\sigma^4)|t|}-C(A,Y,\sigma,g_{\sol}).
\end{align*}
By Lemma \ref{lem:unnormalizedpseudolocality}, 
for any $(x,t)\in \Omega((1-\sigma^4)D) \times [-9,-\frac 19],$  we have
\begin{align*}
    (1-\sigma^4)d_{g_{\sol,t}}(x_0,x)\leq d_{g_t}(\psi(x_0),\psi(x))\leq (1+\sigma^4)d_{g_{\sol,t}}(x_0,x)
\end{align*}
if we choose $\delta \leq \overline{\delta}(g_{\sol},\sigma)$.
Finally, we can combine the above inequalities with \eqref{heat kernel bound; FangLi} and Lemma \ref{lem:basicACfacts} \ref{lem:basicACfacts2} to get the lower bound in \ref{eq:comparabilityofpotentials}:
\begin{align*}
    f_t(\psi(x))&\geq \frac{d_{g_t}^2(\psi(x_0),\psi(x))}{2(1+\sigma^4)|t|} -C(\sigma,A,Y,g_{\sol})\ge \frac{d_{g_{\sol,t}}^2(x_0,x)}{2(1+\sigma^4)|t|}(1-\sigma^4)^2 -C(\sigma,A,Y,g_{\sol})\\
    &\geq \frac{1}{1+\sigma}f_{\sol,t}(x) -C(\sigma,A,Y,g_{\sol}),
\end{align*}
where we have assumed (without loss of generality) $\sigma \leq \overline{\sigma}$. The upper bounds of \ref{eq:comparabilityofpotentials} and \ref{eq:comparabilitytodistance} then follow exactly as in \cite[Proposition 6.3]{FangLi}. Finally, \ref{eq:comparabilitytoheatreg} follows from \cite[Lemma 6.8]{FangLi}.
\end{proof}

In the following, we consider only comparison diffeomorphisms corresponding to a fixed conjugate heat kernel. In this case, we can make a crucial improvement on Lemma \ref{lem:comparabilityofpotentials} \ref{eq:comparabilityofpotentials} by removing the dependence of the constant $C$ on $A$.

\begin{lemma} \label{lem:improvementofcomparability} For any AC K\"ahler-Ricci shrinker $(M_{\sol},J_{\sol},g_{\sol},f_{\sol})$ and any $Y \in (1,\infty)$, the following holds if $D\geq \underline{D}(Y,g_{\sol})$, $\delta \leq \overline{\delta}(Y,g_{\sol})$.

Suppose $(M,(g_t)_{t\in [-T,0)})$ is a compact solution of the K\"ahler-Ricci flow, $\nu=(\nu_t)_{t\in [-T,0)}$ is a conjugate heat kernel based at the singular time $0$, and $r>0$ is such that $\mathcal{N}_{\nu}(\delta^{-1}r^2) \geq -Y$, and $\nu$ is $(\delta,r)$-self-similar.  
Suppose $h\in H$ and $\psi:\Omega_h(D)\to M$ is an open holomorphic embedding such that 
\begin{equation*}
    \|r^{-2}\psi^{\ast}g_{-r^2}-g_{h,-1}\|_{C^2(\Omega_h(D),g_{h,-1})} < \delta,  \qquad \|h\|_H < \delta, \qquad f_{-r^2}(\psi(x_0))\leq 1+f_{\sol}(x_0)+W,
\end{equation*}
where $W$ is the entropy of $(M_{\sol},g_{\sol},f_{\sol})$, $x_0 \in \operatorname{argmin}_{M_{\sol}}f_{\sol}$, and $f_t$ is defined by $d\nu_t = (2\pi|t|)^{-n}e^{-f_t}dg_t$. Then we have
\begin{equation}\label{eq:comparabilityofpotentials improved}
\sup_{t\in [-8,-\frac{1}{8}]} f_{r^2 t}(\psi(x_0))\leq \frac{1}{2}+f_{\sol}(x_0)+W.
\end{equation}
\end{lemma}
\begin{proof} 
By rescaling, we may assume $r=1$. Because $f_h(x_0)=f_{\sol}(x_0)$ for all $h\in H$, we may assume $h=0$ since the following proof is uniform in $h$. 

Suppose by way of contradiction that the claim fails for some $Y$, so that (after rescaling) for some $\delta_i \searrow 0$ and $D_i \to \infty$, there is a sequence of compact K\"ahler-Ricci flows $(M_i,(g_{i,t})_{t\in [-\delta_i^{-1},0)})$ and conjugate heat kernels $(\nu_t^i)_{t\in [-\delta_i^{-1},0)}$ based at the singular time, such that $\mathcal{N}_{\nu^i}^{g_i}(\delta_i^{-1})\geq -Y$, $(M_i,(g_{i,t})_{[-\delta_i^{-1},0)})$ is $(\delta_i,1)$-self-similar at $\nu^i$, and there exist smooth open embeddings $\psi_i:\Omega(D_i)\to M_i$ such that 
\begin{equation} \label{eq:comparabilityofpotentialsassumedC2close}
    \|\psi_i^{\ast}g_{i,-1}-g_{\sol,-1}\|_{C^2(\Omega(D_i),g_{\sol,-1})}<\delta_i 
\end{equation}
and such that if we write $d\nu_t^i = (2\pi|t|)^{-n}e^{-f_{i,t}}dg_{i,t}$, then 
\begin{equation*}
f_{i,-1}(\psi_i(x_0)) \leq 1+f_{\sol}(x_0)+W,
\end{equation*}
but also
\begin{equation*}
    \sup_{t\in [-8,-\frac{1}{8}]} f_{i,t}(\psi_i(x_0)) > \frac{1}{2}+f_{\sol}(x_0)+W.
\end{equation*}
By \eqref{eq:comparabilityofpotentialsassumedC2close} and Lemma \ref{lem:unnormalizedpseudolocality}, we may assume that in fact $\psi_i^{\ast}g_{i,t} \to g_{\sol,t}$ in $C_{\operatorname{loc}}^{\infty}(M_{\sol}\times [-9,-\frac{1}{9}])$. Because $\nu^i$ is $(\delta_i,1)$-self-similar, we also have 
\begin{equation*}
    \begin{split}
        &\quad\lim_{i\to \infty}  \int_{-10}^{-\frac{1}{10}}2|t| \int_{{\Omega (D_i)}} \left| \Ric(\psi_i^{\ast}g_{i,t})+(\nabla^{\psi_i^{\ast}g_{i,t}})^2f_{i,t}-\frac{1}{|t|}\psi_i^{\ast}g_{i,t}\right| ^2_{\psi_i^{\ast}g_{i,t}}\frac{e^{-\psi_i^{\ast}f_{i,t}}}{(2\pi|t|)^{n}}\psi_i^{\ast}dg_{i,t} dt=0
    \end{split}
\end{equation*}
By Lemma \ref{lem:comparabilityofpotentials}, there exists $C = C(Y) \in (1,\infty)$ such that
\begin{equation} \label{eq:weakboundsforfi}
    \frac{1}{2}f_{\sol,t}(x)-C\leq (\psi_i^{\ast}f_{i,t})(x)\leq \frac{3}{2}f_{\sol,t}(x)+C
\end{equation}
on $\Omega(\frac{1}{2}D_i)$ for all $t\in [-9,-\frac{1}{9}]$ and $i\in \mathbb{N}$. By combining these facts with local parabolic regularity and passing to a further subsequence, we may therefore assume that $\psi_i^{\ast}f_{i,t} \to f_{\infty,t}$ for some $f_{\infty,t} \in C_{\operatorname{loc}}^{\infty}(M_{\sol}\times (-9,-\frac{1}{9}))$ which satisfies
\begin{equation*}
    \Ric(g_{\sol,t})+(\nabla^{g_{\sol,t}})^2 f_{\infty,t} = \frac{1}{|t|}g_{\sol,t}.
\end{equation*}
Because $g_{\sol,t}$ does not isometrically split a factor of $\mathbb{R}$, it follows that for each $t\in (-9,-\frac{1}{9})$, we have $f_{\infty,t}-f_{\sol,t}=B(t)$ for some constant $B(t)\in \mathbb{R}$. By our choice of normalization, we have $\int_{M_{\sol}}(2\pi)^{-n}e^{-f_{\sol}}dg_{\sol} = e^{W}$, where $W$ is the entropy of $(M_{\sol},g_{\sol},f_{\sol})$. On the other hand, Lemma \ref{lem:Dandrho} implies 
\begin{equation*}
    \psi_i(\Omega(D)) \supseteq \psi_i \left( B_{g_{\sol}}(x_0,\frac{1}{2}D) \right) \supseteq B_{g_{i,-1}}(\psi_i(x_0),\frac{1}{4}D)
\end{equation*}
for $D$ and $i\in \mathbb{N}$ sufficiently large, so we can estimate
\begin{align*} \left|
    \int_{\Omega(D)} (2\pi|t|)^{-n} e^{-f_{\infty,t}}dg_{\sol,t} -1 \right| \leq & \left| \int_{\Omega(D)} (2\pi |t|)^{-n}e^{-f_{\infty,t}}dg_{\sol,t} - \int_{\psi_i(\Omega(D))} (2\pi |t|)^{-n} e^{-f_{i,t}}dg_{i,t} \right|  \\
    &+\int_{M_{i}\setminus \psi_i(\Omega(D))} (2\pi |t|)^{-n} e^{-f_{i,t}}dg_{i,t} \\
    \leq & \Psi(i^{-1}|D)+ \nu_{t}^i\left( M_{i}\setminus B_{g_{i,-1}}(\psi_i(x_0),\frac{1}{4}D) \right) \\
    \leq &\Psi(i^{-1}|D)+ \Psi(D^{-1}|Y)
\end{align*}
for all $t\in [-8,-\frac{1}{8}]$,
where for the last inequality, we used \cite[Propositions 3.12, 3.13]{Bam1} and 
the fact that $(\psi_i(x_0),t)$ is a $C(Y,g_{\sol})$-center by \eqref{eq:weakboundsforfi}. 
Therefore, $B(t)=W$ for all $t\in [-8,-\frac{1}{8}]$, so for sufficiently large $i\in \mathbb{N}$, we have
\begin{equation*}
    f_{i,t}(\psi_i(x_0)) = (\psi_i^{\ast} f_{i,t})(x_0) \leq f_{\infty,t}(x_0) + \frac{1}{2} = f_{\sol}(x_0)+W+\frac{1}{2}
\end{equation*}
for all $t\in [-8,-\frac{1}{8}]$,
a contradiction.
\end{proof}

\section{Non-concentration and decay estimates}\label{Non-concentration and decay estimates}

\label{section:decay}

Throughout this section, we fix a compact K\"ahler-Ricci flow $(M^n,(\widetilde{\omega}_t)_{t\in [-T,0)})$ with $T\geq 1$, a conjugate heat kernel $(\widetilde{\nu}_t)_{t\in [-T,0)} \in M_0$ based at the singular time, and write $d\widetilde{\nu}_t = (2\pi|t|)^{-n}e^{-\widetilde{f}_t}d\widetilde{g}_t$. We also fix a K\"ahler-Ricci shrinker $(M_{\sol},J_{\sol},g_{\sol},f_{\sol})$ which is either AC or compact, and fix a basepoint $x_0 \in \operatorname{argmin}_{M_{\sol}}f_{\sol}$. Let $(\omega_h)_{h\in H}$ be the associated model family of K\"ahler-Ricci shrinkers.

\subsection{The non-concentration estimate}
In this subsection, we establish the key non-concentration lemma. Throughout this subsection, we will work under the following assumption.

\begin{assumption} \label{assume:existsholomap} For some $\rho \in (0,\frac{1}{5})$, there is an open holomorphic embedding
$F: \Omega(D(\rho))\to M$.
\end{assumption}

We may therefore define 
\begin{equation} \label{eq:modifyingtheflow}
    \omega_s := e^s (F\circ \phi_{-s})^{\ast}\widetilde{\omega}_{-e^{-s}}, \qquad f_s:= (F\circ \phi_{-s})^{\ast}\widetilde{f}_{-e^{-s}}
\end{equation}
on $\Omega(e^{\frac{1-\rho}{2}s}D(\rho))$ for all $s\geq 0$. Setting $\nu_s := (F\circ \phi_{-s})^{\ast} \widetilde{\nu}_{-e^{-s}}$, we then have $d\nu_s = (2\pi)^{-n}e^{-f_s}dg_s$. Moreover, in the open set $\mathcal{U}:= \cup_{s\geq 0} (\Omega (e^{\frac{2}{5}s}D(\rho)) \times \{s\})$, the flow $(\omega_s)_{s\geq 0}$ satisfies the following modified K\"ahler-Ricci flow equation:
\begin{equation} \label{eq:modifiedflow}
    \partial_s \omega_s = -\left( \Ric(\omega_s) + \frac{1}{2}\mathcal{L}_{X}\omega_s -\omega_s\right).
\end{equation}

In the following, we use Lemma \ref{lem:unnormalizedpseudolocality} to show that closeness of the normalized flow $(\omega_s)_{s\geq 0}$ at some time to a model metric propagates forwards in time, and outwards in space.

\begin{lemma} \label{lem:pseudoapplication} Under Assumption \ref{assume:existsholomap}, for any $\epsilon \in (0,\frac{1}{10})$, there exist $\delta_0(\epsilon,\rho)>0$ and $D_0(\epsilon,\rho)\in (1,\infty)$ such that if $D \geq D_0$, $h\in H$, $s_0 \geq \delta_0^{-1}$ satisfy 
\begin{equation} \label{eq:pseudohypothesis}
\|h\|_H <\delta_0, \qquad \Omega_h(D)\times \{s_0\} \subseteq \mathcal{U}, \qquad \|\omega_{s_0} - \omega_h\|_{C^2(\Omega_h(D),g_h)}<\delta_0,
\end{equation}
then the following hold:
\begin{enumerate}
    \item \label{lem:pseudoapplication1} For all $s\in [s_0,s_0+1]$, we have
\begin{equation*}
\|\omega_{s}-\omega_h\|_{C^{\lfloor \epsilon^{-1}\rfloor }(\Omega_h(e^{\frac{1-\rho}{2}(s-s_0)}e^{-\rho}D),g_h)} < \epsilon,
\end{equation*}

    \item \label{lem:pseudoapplication2} If $\varphi \in C^{\infty}(\mathcal{U})$ satisfies $\omega_s=\omega_h+\sqrt{-1}\partial \overline{\partial} \varphi_s$ and
\begin{equation} \label{modifiedPCMA}
    \partial_s \varphi_s = \log \left( \frac{(\omega_h+\sqrt{-1}\partial \overline{\partial} \varphi_s)^n}{\omega_h^n}\right) -\frac{1}{2}X\cdot \varphi_s + \varphi_s
\end{equation}
on $\mathcal{U}$, then for all $s\in [s_0,s_0+1]$, we have
\begin{equation*}
    \sup_{\Omega_h(e^{\frac{1-\rho}{2}(s-s_0)}e^{-\rho}D)}|\varphi_s| <\epsilon + 10\sup_{\Omega_h(D)}|\varphi_{s_0}|.
\end{equation*}
\end{enumerate}
\end{lemma}
\begin{proof} Define $\widehat{\omega}_t :=e^{s_0}(F\circ \phi_{-s_0})^{\ast}\widetilde{\omega}_{e^{-s_0}t}$, so that $(\widehat{\omega}_t)_{t\in [-e^{s_0}T,0)}$ is a solution of the unnormalized K\"ahler-Ricci flow satisfying 
\begin{equation*}
    \widehat{\omega}_{-1} = \omega_{s_0}, \qquad \omega_s = e^{s-s_0} \phi_{-(s-s_0)}^{\ast}\widehat{\omega}_{-e^{-(s-s_0)}}.
\end{equation*}
and therefore our hypotheses give
\begin{equation*}
    \|\widehat{\omega}_{-1}-\omega_{h,-1}\|_{C^2 (\Omega_h(D),g_{h,-1})}<\delta_0.
\end{equation*}
If $D_0\geq {D}_0(\epsilon,\rho)$ and $\delta_0 \leq \overline{\delta}_0(\epsilon,\rho)$, then we can apply Lemma \ref{lem:unnormalizedpseudolocality} to conclude
\begin{equation*}
    \sup_{t\in [-1,-\epsilon]} \|\widehat{\omega}_t-\omega_{h,t}\|_{C^{\lfloor \epsilon^{-1}\rfloor}(\Omega_h(e^{-\rho}D)),g_{h,t})}<\epsilon^2 e^{-\frac{\lfloor\epsilon^{-1}\rfloor}{2}}.
\end{equation*}
If we also require $\epsilon<e^{-1}$, this implies that for all $s\in [s_0,s_0+1]$, we have
\begin{align*}
    \|\omega_s-\omega_h\|_{C^{\lfloor \epsilon^{-1} \rfloor}(\phi_{s-s_0}(\Omega_h(e^{-\rho}D)),g_{h})} &= \|e^{s-s_0}\phi_{-(s-s_0)}^{\ast}(\widehat{\omega}_{-e^{-(s-s_0)}}-\omega_{h,-e^{-(s-s_0)}})\|_{C^{\lfloor \epsilon^{-1} \rfloor}(\phi_{s-s_0}(\Omega_h(e^{-\rho}D)),g_{h})}\\
    &\le  e^{\frac{\lfloor\epsilon^{-1}\rfloor}{2}} \|\widehat{\omega}_{-e^{-(s-s_0)}}-\omega_{h,-e^{-(s-s_0)}}\|_{C^{\lfloor \epsilon^{-1}\rfloor}(\Omega_h(e^{-\rho}D),g_{h,-e^{-(s-s_0)}})}\\
    &<\epsilon^2.
\end{align*}
This and \eqref{eq:containmentofsublevelsets} imply \ref{lem:pseudoapplication1}. For $t\in [-1,-e^{-1}]$, we then set $\widehat{\varphi}_t:= |t|\phi_{\log\frac{1}{|t|}}^{\ast}\varphi_{s_0+\log \frac{1}{|t|}}$, and integrate
\begin{equation*}
    |\partial_t \widehat{\varphi}_t| = \left|\log \left( \frac{\widehat{\omega}_t^n}{\omega_{h,t}^n} \right)\right| \leq C|\widehat{\omega}_t-\omega_{h,t}|_{\omega_{h,t}} \leq C\epsilon
\end{equation*}
in time, yielding \ref{lem:pseudoapplication2}.
\end{proof}

We now combine the estimates of Section \ref{section:KRFnearAC} with the general strategy of \cite{LSS} to prove the following non-concentration estimate for a weighted $L^2$ norm of $\varphi$, which ensures that the $L^2$ norm does not concentrate near spatial infinity. 
 In more detail, the proof proceeds as follows.
Using our assumption that some time slice $g_{s_0}$ of our flow is close in $C^{\lfloor \epsilon_1^{-1}\rfloor}$ to a model metric $g_h$, we may convert a weighted $L^2$ estimate for the relative K\"ahler potential $\varphi_{s_0}$ to a weak pointwise estimate \eqref{eq:NC1} on a slightly smaller region. An application of pseudolocality in the form of Lemma \ref{lem:pseudoapplication} allows us to propagate this to a weak pointwise estimate \eqref{eq:strongC0bound} in spacetime, in a region which grows exponentially fast. From here, we must estimate the negative and positive parts of $\varphi$ separately. 

An elementary computation shows $e^{-s}(\varphi_s)_+$ is a subsolution of the heat flow of $\Delta_{f_h}$. If $\varphi$ were globally defined, an $L^{2p}$ estimate for $(\varphi_{s_0+1})_+$ would follow immediately from the hypercontractivity of the associated semigroup. Instead, we must graft $e^{-s}\varphi_s$ onto an appropriate normalization of $e^{-s}f_h$, thereby obtaining a global subsolution $\psi_s$ of the $\Delta_{f_h}$-heat flow on $M_{\sol}$. Our assumption on $D$ ensures that the error introduced by this procedure is negligible relative to the weighted $L^2$ norm of $\varphi_{s_0}$. 

The bound for $(\varphi_s)_-$ follows from a somewhat similar idea, using the fact that, the K\"ahler potential $\widehat{\varphi}_t$ of the un-modified flow $(\widehat{\omega}_t)$, satisfies $(\partial_t -\Delta_{\widehat{g}_t})(\widehat{\varphi}_t)_- \leq 0$ where it is defined. To apply hypercontractivity along the Ricci flow, we must therefore graft $\widehat{\varphi}_-$ to a global solution of the heat flow of $\widehat{g}_t$ which is closely comparable to the soliton potential in the region of interest. Such a solution is provided by the work of \cite{FangLi}, so the $L^{2p}$ bound for $\varphi_-$ at a later time follows after translating between the modified and un-normalized flows.

For the remainder of the section, we fix the following constants: $p_1:= \frac{1+e^{\frac{1}{4}}}{2}$, $\theta_1:= e^{\frac{1}{200}}-1$, and $\rho_1 := \frac{1}{800}$.

\begin{lemma} \label{lem:nonconcentration} (Non-concentration estimate) For any $Y \in (0,\infty)$, there exist $\epsilon_1 = \epsilon_1(g_{\sol},Y) \in (0,\frac{1}{10^2})$, $C_1 =C_1(g_{\sol},Y)\in (1,\infty)$, and $D_1=D_1(g_{\sol},Y) \in (1,\infty)$ such that the following holds under Assumption \ref{assume:existsholomap} with $\rho \in (0,\rho_1]$, whenever $\mathcal{N}_{\widetilde{\nu}}(1)\geq -Y$, $D \geq D_1$ and $\Omega_h(D) \times \{s_0\}\subseteq \mathcal{U}$. Suppose $h\in H$, and for some $s_0\geq \epsilon_1^{-1}$, we have
\begin{align} \label{eq:nonconcentrationcourse} \|h\|_H <\epsilon_1, \qquad \|g_{s_0} - g_h\|_{C^{\lfloor \epsilon_1^{-1} \rfloor}(\Omega_h(D),g_h)}<\epsilon_1, \qquad f_{s_0}(x_0) \leq f_{\sol}(x_0)+W+1,
\end{align}
and that $\varphi \in C^{\infty}(\mathcal{U})$ satisfies $\omega_{s} = \omega_h + \sqrt{-1}\partial \overline{\partial}\varphi_{s}$ and 
\begin{equation} \label{modifiedPCMA2}
    \partial_s \varphi_s = \log \left( \frac{(\omega_h+\sqrt{-1}\partial \overline{\partial} \varphi_s)^n}{\omega_h^n}\right) -\frac{1}{2}X\cdot\varphi_s + \varphi_s
\end{equation}
on $\mathcal{U}$. If in addition
\begin{equation*}
    \left( \int_{\Omega_h(D)} \varphi_{s_0}^2 d\nu_h+ e^{-(1+\theta_1)\frac{D^2}{2}}\right)^{\frac{1}{2}}<\epsilon
\end{equation*}
for some $\epsilon \in (0,\epsilon_1]$, then  
\begin{equation*}
    \left(\int_{\Omega_h(e^{\frac{1}{4}}\widetilde{D}_{\epsilon})} |\varphi_{s_0+1}|^{2p_1} d\nu_h \right)^{\frac{1}{2p_1}} < C_1\epsilon,
\end{equation*}
where $\widetilde{D}_{\epsilon}$ is defined by $e^{-\frac{\widetilde{D}_{\epsilon}^2}{4}}=\epsilon$.
\end{lemma}
\begin{proof} Set $\lambda:= e^{-\frac{1}{100}}$. We will determine $\epsilon_1,C_1,D_1$ in the course of the proof. By hypothesis, $D > \frac{1}{\sqrt{1+\theta_1}} \widetilde{D}_{\epsilon}=\lambda^{\frac{1}{4}}\widetilde{D}_{\epsilon}$. For any $x\in \Omega_h(\lambda^{\frac{1}{3}} \widetilde{D}_{\epsilon})$, $|\nabla b_h|_{g_h}\leq 1$ implies 
\begin{equation*} B_{g_h}(x,1)\subseteq \Omega_h(\lambda^{\frac{1}{3}} \widetilde{D}_{\epsilon}+1) \subseteq \Omega_h(\lambda^{\frac{1}{4}}\widetilde{D}_{\epsilon}) \end{equation*} if $\epsilon_1 \leq \overline{\epsilon}_1$.
We can thus estimate
\begin{equation} \label{eq:NC1}
    \int_{B_{g_h}(x,1)} \varphi_{s_0}^2 \omega_h^n \leq \epsilon^2 e^{\frac{1}{2}(\lambda^{\frac{1}{4}} \widetilde{D}_{\epsilon})^2}\leq e^{-\frac{1}{2}(1-\lambda^{\frac{1}{2}})\widetilde{D}_{\epsilon}^2},
\end{equation}
By \eqref{eq:NC1} and $\Delta_{g_h}\varphi_{s_0}=\operatorname{tr}_{g_h}(g_{s_0}-g_h)$, we can apply Lemma \ref{lem:L2toLinfty} and then Lemma \ref{lem: schauder estimates} to obtain
\begin{equation} \label{eq:5derivativesonvarphi} 
\sup_{\Omega_h(\lambda^{\frac{1}{3}}\widetilde{D}_{\epsilon}) }|\varphi_{s_0}|<10^{-20},
\end{equation}
when $\epsilon_1 \leq \overline{\epsilon}_1(g_{\sol})$.  
If $\epsilon_1 \leq \overline{\epsilon}_1(g_{\sol})$ and $D_1 \geq \underline{D}_1(g_{\sol})$,
we can apply Lemma \ref{lem:pseudoapplication} \ref{lem:pseudoapplication2} with $D$ replaced by $\lambda^{\frac{1}{3}}\widetilde{D}_{\epsilon}$ and $\rho$ replaced by $\rho_1$ to obtain
\begin{equation} \label{eq:strongC0bound}
    \sup_{s\in [s_0,s_0+1]} \sup_{\Omega_h(e^{\frac{(1-\rho_1)(s-s_0)}{2}}\sqrt{\lambda} \widetilde{D}_{\epsilon} ) } |\varphi_s| \leq 10^{-10}.
\end{equation}
Define
\begin{equation*}
    \psi_s(x) := \left\{ \begin{array}{cc}
        \max \{e^{-(s-s_0)}\varphi_s(x)_+, e^{-(s-s_0)}f_h(x) - \frac{1}{2}\lambda^3 \widetilde{D}_{\epsilon}^2 \}, &
        b_h(x)\leq e^\frac{s-s_0}{2}\lambda  \widetilde{D}_{\epsilon}\\
        e^{-(s-s_0)}f_h(x) - \frac{1}{2}\lambda^3 \widetilde{D}_{\epsilon}^2,  & b_h(x)\geq e^{\frac{s-s_0}{2}}\lambda \widetilde{D}_{\epsilon}
    \end{array} \right. .
\end{equation*}
For $(x,s) \in M_{\sol}\times [s_0,s_0+1]$ in a sufficiently small neighborhood of $\cup_{s' \in [s_0,s_0+1]} (\partial \Omega_h(e^{\frac{s'-s_0}{2}}\lambda \widetilde{D}_{\epsilon})\times \{s'\})$, we have
\begin{equation*}
    e^{-(s-s_0)}f_h(x)-\frac{1}{2}\lambda^3 \widetilde{D}_{\epsilon}^2 \geq \frac{1}{2}\lambda^{2} \widetilde{D}_{\epsilon}^2-\frac{1}{2}\lambda^3 \widetilde{D}_{\epsilon}^2 \geq \frac{1}{2}(\lambda^{2}-\lambda^3)\widetilde{D}_{\epsilon}^2 \geq 1.
\end{equation*}
On the other hand, $\varphi_s(x)_+ <1$ since 
\begin{equation*}
    \Omega_h(e^{\frac{s-s_0}{2}}\lambda \widetilde{D}_{\epsilon}) \subseteq \Omega_h(e^{\frac{(1-\rho_1)(s-s_0)}{2}}\sqrt{\lambda}\widetilde{D}_{\epsilon})
\end{equation*}
for all $s\in [s_0,s_0+1]$. 

Next, we compute
\begin{equation*}
    (\partial_s - \Delta_{f_h})f_h = -\Delta_{\omega_h}f_h+\frac{1}{2}X\cdot f_h = f_h-n \leq f_h,
\end{equation*}
\begin{equation*}
    (\partial_s - \Delta_{f_h})\varphi_s = \log \left( \frac{(\omega_h+\sqrt{-1}\partial \overline{\partial}\varphi_s)^n}{\omega_h^n} \right)+\varphi_s-\operatorname{tr}_{\omega_h}(\sqrt{-1}\partial \overline{\partial}\varphi_s) \leq \varphi_s.
\end{equation*}
It follows that $\psi_s$ is a continuous subsolution of a drift heat equation on $M_{\sol}\times [s_0,s_0+1]$ in the sense of distributions:
\begin{equation} \label{eq:psiissubsolution}
    (\partial_s - \Delta_{f_h})\psi_s \leq 0.
\end{equation}
By our definition of $\psi_s$, we moreover have $\psi_s = e^{-(s-s_0)}(\varphi_s)_+$ on $\Omega_h(e^{\frac{s-s_0}{2}}\lambda^{\frac{3}{2}} \widetilde{D}_{\epsilon})$ for all $s\in [s_0,s_0+1]$. 

By \eqref{eq:psiissubsolution} and the hypercontractivity \cite[Theorem 5.2.3 and Proposition 5.7.1]{Bakry} of the semigroup corresponding to $\Delta_{f_h}$, if we set $p:=\frac{1+\sqrt{e}}{2}$ and $q:=\frac{1+\frac{1}{\sqrt{e}}}{2}$, then
\begin{align*}
    \left( \int_{M_{\sol}} \psi_{s_0+1}^{2p} d\nu_h \right)^{\frac{1}{2p}} &\leq \left( \int_{M_{\sol}} \psi_{s_0}^{2q}d\nu_h \right)^{\frac{1}{2q}} \\&\leq \left( \int_{\Omega_h(D) } |\varphi_{s_0}|^{2q} d\nu_h + \int_{M_{\sol}\setminus \Omega_h(\lambda^{\frac{3}{2}} \widetilde{D}_{\epsilon})} f_h^{2q} d\nu_h\right)^{\frac{1}{2q}}.\\
\end{align*}
H\"older's inequality yields
\begin{equation*}
    \begin{split}
        \int_{\Omega_h(D)} |\varphi_{s_0}|^{2q} d\nu_h &\le \left(\int_{\Omega_h(D)} \varphi_{s_0}^{2} d\nu_h \right)^{q}\nu_h(\Omega_h(D))^{1-q}\\
        &\le C(g_{\sol})\varepsilon^{2q}.
    \end{split}
\end{equation*}
 Moreover, for any $\sigma>0$, we can argue as in  \cite[Proof of Proposition 3.1]{FangLi} to get
\begin{align*}
    \int_{M_{\sol}\setminus \Omega_h(\lambda^{\frac{3}{2}} \widetilde{D}_{\epsilon})} f_h^{2q} d\nu_h  \leq C(\sigma,g_{\sol})e^{-\frac{\lambda^3 \widetilde{D}_{\epsilon}^2}{2+\sigma}} = C(\sigma,g_{\sol})\epsilon^{\frac{4\lambda^3}{2+\sigma}}.
\end{align*}
Since $\lambda>q^{\frac{1}{3}}$, we can define $\sigma>0$ by $\frac{2\lambda^3}{2+\sigma}=q$, yielding
\begin{equation*}
     \int_{M_{\sol}\setminus \Omega_h(\lambda^{\frac{3}{2}} \widetilde{D}_{\epsilon})} f_h^{2q} d\nu_h  \leq C(g_{\sol})\varepsilon^{2q}.
\end{equation*}
Combining expressions yields 
\begin{align}\label{Lp bound on positive part}
 \left(\int_{\Omega_h(e^{\frac{1}{2}}\lambda^{\frac{3}{2}} \widetilde{D}_{\epsilon})} (\varphi_{s_0+1})_+^{2p} d\nu_h \right)^{\frac{1}{2p}} \leq e\left( \int_{M_{\sol}} \psi_{s_0+1}^{2p} d\nu_h \right)^{\frac{1}{2p}}  \leq C(g_{\sol})\varepsilon.
\end{align}

Next, we prove the analogous bound for $(\varphi_{s_0+1})_-$. Set $S:=-e^{-1}$, $t_0 := -e^{-s_0}$, $\widehat{F}:= F\circ \phi_{-s_0}$, and
\begin{align*}
    \widehat{\omega}_t := |t_0|^{-1}\widetilde{\omega}_{|t_0|t}, \qquad \widehat{\varphi}_t := |t|(\widehat{F}^{-1})^{\ast}\phi_{\log \frac{1}{|t|}}^{\ast}\varphi_{s_0+\log \frac{1}{|t|}}, \qquad \widehat{\nu}_t := \widetilde{\nu}_{|t_0|t},
\end{align*}
so that $\widehat{\omega}_{-1}=(\widehat{F}^{-1})^{\ast}\omega_{s_0}$,  $\widehat{\omega}_S = e^{-1}(\widehat{F}^{-1})^{\ast}\phi_1^{\ast}\omega_{s_0+1}$, and 
\begin{equation*} \widehat{F}^{\ast}\widehat{\omega}_t = \omega_{h,t}+\sqrt{-1}\partial \overline{\partial}(\widehat{F}^{\ast}\widehat{\varphi}_t) \end{equation*}
for all $t\in [-1,S]$. 
From \eqref{eq:nonconcentrationcourse}, we have
\begin{equation*}
    \|\widehat{F}^{\ast}\widehat{g}_{-1}-g_{h,-1}\|_{C^2(\Omega_h(D),g_{h,-1})} \leq \epsilon_1,
\end{equation*}
hence Lemma \ref{lem:unnormalizedpseudolocality} implies that
\begin{equation} \label{eq:nonconcentrationpseudoapp1}
  \|\widehat{F}^{\ast}\widehat{g}_t - g_{h,t}\|_{C^5(\Omega_h(e^{-\frac{\rho_1}{8}}D),g_{h,t})}\leq \rho_1^2,
\end{equation}
for all $t\in[-1,S]$ and $\epsilon_1 \leq \overline{\epsilon}_1(g_{\sol})$. If $\epsilon_1 \leq \overline{\epsilon}_1(g_{\sol})$ and $D_1 \geq \underline{D}_1(g_{\sol})$,   
we can apply Lemma \ref{lem:Dandrho} to obtain
\begin{equation*}
    B_{g_{h,t}}(x_0,e^{-\frac{\rho_1}{4}}D)\subset \Omega_h(e^{-\frac{\rho_1}{8}}D), \qquad 
     \|\widehat{F}^{\ast}\widehat{g}_t - g_{h,t}\|_{C^5\left(B_{g_{h,t}}(x_0,e^{-\frac{\rho_1}{4}}D),g_{h,t}\right)}\le \rho_1^2
\end{equation*}
so that
\begin{equation} \label{eq:NCmetricsclose}
    (1-\rho_1^{\frac{3}{2}})g_{h,t} \leq \widehat{F}^*\widehat{g}_t \leq (1+\rho_1^{\frac{3}{2}})g_{h,t}
\end{equation}
on $B_{g_{h,t}}(x_0,e^{-\frac{\rho_1}{4}}D)$ for all $t\in[-1,S]$. Standard distance comparison arguments then yield
\begin{equation} \label{eq:NC2}
B_{\widehat{g}_t}(\widehat{F}(x_0),e^{-\frac{\rho_1}{2}}D)\subset \widehat{F}\left(B_{g_{h,t}}(x_0,e^{-\frac{\rho_1}{4}}D)\right) \subseteq \widehat{F}\left(\Omega_h(e^{-\frac{\rho_1}{8}}D) \right). 
\end{equation}
For any $y\in B_{\widehat{g}_t}(\widehat{F}(x_0),\lambda \widetilde{D}_{\epsilon})$, \eqref{eq:NCmetricsclose} gives
\begin{equation*}
    d_{g_{h,t}}(x_0,\widehat{F}^{-1}(y)) \leq (1-\rho_1^{\frac{3}{2}})^{-\frac{1}{2}}d_{\widehat{g}_t}(\widehat{F}(x_0),y)<(1-\rho_1^{\frac{3}{2}})^{-\frac 12}\lambda\widetilde{D}_{\epsilon}.
\end{equation*}
Combined with Lemma \ref{lem:comparingbandd}, this implies $\widehat{F}^{-1}(y) \in \Omega_h(e^{\rho_1}\lambda \widetilde{D}_{\epsilon})$ if we assume $\epsilon_1 \leq \overline{\epsilon}_1(g_{\sol})$, so that
\begin{equation*}
    \phi_{\log\frac{1}{|t|}}(\widehat{F}^{-1}(y)) \in \Omega_h(e^{\frac{1-\rho_1}{2}\log\frac{1}{|t|}}\sqrt{\lambda}\widetilde{D}_{\epsilon})
\end{equation*}
for all $t\in [-1,S]$. Along with \eqref{eq:strongC0bound}, this implies
\begin{equation} \label{eq:NC3}
    \sup_{t\in [-1,S]} \sup_{B_{\widehat{g}_t}(\widehat{F}(x_0),\lambda \widetilde{D}_{\epsilon})} |\widehat{\varphi}_t| \leq \frac{1}{2}.
\end{equation}
We can moreover compute
\begin{equation*}
    (\partial_t - \Delta_{\widehat{\omega}_t})\widehat{\varphi}_t  = -\left( \log \left( \frac{(\widehat{\omega}_t + \sqrt{-1}\partial \overline{\partial}(-\widehat{\varphi}_t))^n}{\widehat{\omega}_t^n} \right)-\operatorname{tr}_{\widehat{\omega}_t}(\sqrt{-1}\partial \overline{\partial}(-\widehat{\varphi}_t)) \right) \geq 0,
\end{equation*}
so that $(\widehat{\varphi}_t)_- = \max\{-\widehat{\varphi}_t,0\}$ satisfies
\[
(\partial_t-\Delta_{\widehat{\omega}_t})(\widehat{\varphi}_t)_{-} \leq 0
\]
in the sense of distributions. 
Write $d\widehat{\nu}_t = (2\pi|t|)^{-n}e^{-\widehat{f}_t}d\widehat{g}_t$, and let $\widehat{q}_t \in C^{\infty}(M \times [-1,S])$ be the solution of 
\begin{equation*} (\partial_t -\Delta_{\widehat{\omega}_t})\widehat{q}_t = 0 .\end{equation*}
satisfying $\widehat{q}_{-1}=\widehat{f}_{-1}$. 
If $\epsilon_1 \leq \overline{\epsilon}_1$, then we can use $\widehat{f}_{-1}(\widehat{F}(x_0))\leq f_{\sol}(x_0)+W+1$, \eqref{eq:nonconcentrationcourse},\eqref{eq:nonconcentrationpseudoapp1}, and Lemma \ref{lem:comparabilityofpotentials} to obtain
$C_0=C_0(g_{\sol},Y) \in (1,\infty)$ such that
\begin{equation} \label{eq:comparabilityapproximatepotential}
(1-10^{-3})\widehat{F}_{\ast}f_{h,t} -C_0 \leq \widehat{f}_t\leq (1+10^{-3})\widehat{F}_{\ast}f_{h,t} + C_0,
\end{equation}
\begin{equation}
\label{eq:comparabilitywithdistancefunction}
(1-10^{-3})\frac{d_{\widehat{g}_t}^2(\widehat{F}(x_0),\cdot)}{2} -C_0 \leq \widehat{q}_t \leq (1+10^{-3})\frac{d_{\widehat{g}_t}^2(\widehat{F}(x_0),\cdot)}{2}+C_0,
\end{equation}
on $B_{\widehat{g}_t}(\widehat{F}(x_0),\lambda^{\frac{1}{2}}\widetilde{D}_{\epsilon})$ for all $t\in [-1,S]$. 

We now define a function $\widehat{\psi}_t,$ analogous to the function $\psi_s$ used in the first part of this proof. Define
\begin{equation*}
     \widehat{\psi}_t(x) :=  \left\{ \begin{array}{cc}
         \max \{\widehat\varphi_t(x)_-, \widehat{q}_t(x)- \frac{1}{2}\lambda^3 \widetilde{D}_{\epsilon}^2 \}, & d_{\widehat{g}_t}(x,\widehat{F}(x_0)) \leq \lambda \widetilde{D}_{\epsilon}\\
      (\widehat{q}_t(x)-\frac{1}{2}\lambda^3 \widetilde{D}_{\epsilon}^2)_+,  & d_{\widehat{g}_t}(x,\widehat{F}(x_0))\geq \lambda  \widetilde{D}_{\epsilon}
     \end{array} \right. .
 \end{equation*}
For $(x,t)\in M\times [-1,S]$ sufficiently close to $\cup_{t'\in [-1,S]}(\partial B_{\widehat{g}_{t'}}(\widehat{F}(x_0),\lambda \widetilde{D}_{\epsilon})\times \{t'\})$, \eqref{eq:comparabilitywithdistancefunction} gives
\begin{equation*}
    \widehat{q}_t(x) -\frac{1}{2}\lambda^3 \widetilde{D}_{\epsilon}^2 \geq \frac{1-10^{-3}}{2}(\lambda \widetilde{D}_{\epsilon}-1)^2-C_0 - \frac{1}{2}\lambda^3 \widetilde{D}_{\epsilon}^2 \geq 1 > \widehat{\varphi}_t(x)_-
\end{equation*}
if we choose $\epsilon_1 \leq \overline{\epsilon}_1(g_{\sol},Y)$. 
It follows that $\widehat{\psi}_t$ is a continuous, nonnegative subsolution of the heat equation coupled to the Ricci flow: 
\begin{equation*}
(\partial_t-\Delta_{\widehat{g}_t})\widehat{\psi}_t\leq 0
\end{equation*}
in the sense of distributions. Moreover, on $B_{\widehat{g}_t}(\widehat{F}(x_0),\lambda^2 \widetilde{D}_{\epsilon})$, \eqref{eq:comparabilitywithdistancefunction} gives
\begin{equation} \label{eq:NC4}
   \begin{split}
     \widehat{q}_t- \frac{1}{2}\lambda^3 \widetilde{D}_{\epsilon}^2\le \frac{1}{2}(1+10^{-3})\lambda^4 \widetilde{D}_\varepsilon^2- \frac{1}{2}\lambda^3 \widetilde{D}_{\epsilon}^2+C_0 \le -\frac{1}{2} \le \widehat{\varphi}_t,
   \end{split}
\end{equation}
if we choose 
$\epsilon_1 \leq \overline{\epsilon}_1(g_{\sol},Y)$. 
By hypercontractivity \cite[Theorem 12.1]{Bam1} along the Ricci flow, 
\begin{align*}
    \left( \int_{M}\widehat{\psi}_S^{2p}d\widehat{\nu}_S \right)^{\frac{1}{2p}}&\leq   \left( \int_{M}\widehat{\psi}_{-1}^{2q}d\widehat{\nu}_{-1} \right)^{\frac{1}{2q}} \\ &\leq \left(\int_{B_{\widehat{g}_{-1}}(\widehat{F}(x_0),\lambda \widetilde{D}_{\epsilon})} |\widehat{\varphi}_{-1}|^{2q}d\widehat{\nu}_{-1}+\int_{M\setminus B_{\widehat{g}_{-1}}(\widehat{F}(x_0),\lambda^2\widetilde{D}_{\epsilon})} (\widehat{f}_{-1})_+^{2q} d\widehat{\nu}_{-1}\right)^{\frac{1}{2q}},
\end{align*}
Our hypothesis $f_{s_0}(x_0)\leq f_{\sol}(x_0)+W+1$ implies that $\widehat{f}_{-1}(\widehat{F}(x_0)) \leq f_{\sol}(x_0)+W+1$, hence $\widehat{F}(x_0)$ is a $C(g_{\sol})$-center of $\widehat{\nu}$. For $\sigma'>0$ to be determined, we can therefore combine H\"older's inequality, \cite[Theorem 2.13]{FangLi}, and \cite[Proposition 6.2]{Bam3} to estimate
\begin{align*}
    \int_{M\setminus B_{\widehat{g}_{-1}}(\widehat{F}(x_0),\lambda^2 \widetilde{D}_{\epsilon})} (\widehat{f}_{-1})_+^{2q} d\widehat{\nu}_{-1} &\leq \left( \int_M (\widehat{f}_{-1})_+^{\frac{2q}{\sigma'}}d\widehat{\nu}_{-1}\right)^{\sigma'}\widehat{\nu}_{-1}(M\setminus B_{\widehat{g}_{-1}}(\widehat{F}(x_0),\lambda^2 \widetilde{D}_{\epsilon}))^{1-\sigma'}\\
    &\leq C(\sigma',Y)e^{-\frac{(1-\sigma')\lambda^4 \widetilde{D}_{\epsilon}^2}{2(1+\sigma')}}\\
    &\leq C(\sigma',Y)\epsilon^{2q},
\end{align*}
where in the last line, we defined $\sigma'>0$ by $\frac{(1-\sigma')\lambda^4}{1+\sigma'}=q$. By \eqref{eq:NCmetricsclose} and Lemma \ref{lem:Dandrho}, we have
\begin{equation*}
    B_{\widehat{g}_{-1}}(\widehat{F}(x_0),\lambda \widetilde{D}_{\epsilon}) \subseteq \widehat{F}(\Omega_h(\sqrt{\lambda}\widetilde{D}_{\epsilon})),
\end{equation*}
so we can also use \eqref{eq:comparabilityapproximatepotential} and H\"older's inequality to estimate
\begin{align*}
    \int_{B_{\widehat{g}_{-1}}(\widehat{F}(x_0),\lambda \widetilde{D}_{\epsilon})} |\widehat{\varphi}_{-1}|^{2q} d\widehat{\nu}_{-1} &\leq C(g_{\sol})\int_{\Omega_h(\sqrt{\lambda} \widetilde{D}_{\epsilon})} |\varphi_{s_0}|^{2q}e^{10^{-3}f_h}d\nu_h \\& \leq C(g_{\sol})\left( \int_{\Omega_h(\sqrt{\lambda} \widetilde{D}_{\epsilon})}\varphi_{s_0}^2 e^{-f_h} \omega_h^n\right)^q \left( \int_{M_{\sol}}e^{\frac{1}{10^3(1-q)}f_h} d\nu_h \right)^{1-q} \\
    &\leq C(g_{\sol})\epsilon^{2q},
\end{align*}
By \eqref{eq:NC4}, we have $\widehat{\psi}_t = (\widehat{\varphi}_t)_-$ on $B_{\widehat{g}_t}(\widehat{F}(x_0),\lambda^2 \widetilde{D}_{\epsilon})$, hence we may combine the above estimates to obtain
\begin{equation*}
    \int_{B_{\widehat{g}_S}(\widehat{F}(x_0),\lambda^2 \widetilde{D}_{\epsilon})}(\widehat{\varphi}_S)_-^{2p}d\widehat{\nu}_S\le C(g_{\sol})\varepsilon^{2p}.
\end{equation*}
Moreover, \eqref{eq:comparabilityapproximatepotential} yields
    \begin{equation*}
C(g_{\sol})\int_{B_{\widehat{g}_S}(\widehat{F}(x_0),\lambda^2 \widetilde{D}_{\epsilon})}(\widehat{\varphi}_S)_-^{2p}d\widehat{\nu}_S\ge \int_{B_{\widehat{g}_S}(\widehat{F}(x_0),\lambda^2 \widetilde{D}_{\epsilon})}(\widehat{\varphi}_S)_-^{2p}e^{-(1+10^{-3})\widehat{F}_*f_{h,S}}\widehat{F}_*(\omega_{h,S})^n.
\end{equation*}
Setting $p_1 :=\frac{1+e^{\frac{1}{4}}}{2}$, 
we can then apply H\"older's inequality to obtain
\begin{equation*}
   \begin{split}
      &\int_{B_{\widehat{g}_S}(\widehat{F}(x_0),\lambda^2 \widetilde{D}_{\epsilon})}(\widehat{\varphi}_S)_-^{2p_1}e^{-\widehat{F}_{\ast}f_{h,S}}(\widehat{F}_{\ast}\omega_{h,S})^n \\
      \le& \left( \int_{B_{\widehat{g}_S}(\widehat{F}(x_0),\lambda^2 \widetilde{D}_{\epsilon})}(\widehat{\varphi}_S)_-^{2p}e^{-(1+10^{-3})\widehat{F}_{\ast}f_{h,S}}(\widehat{F}_{\ast}\omega_{h,S})^n\right)^{\frac{p_1}{p}}\left(\int_{M_{\sol}}e^{\frac{p_1}{10^3(p-p_1)}f_{h,S}}e^{-f_{h,S}}\omega_{h,S}^n\right)^{1-\frac{p_1}{p}}
       \\ \leq & C(g_{\sol})\varepsilon^{2p_1}
   \end{split}
\end{equation*}
Because $\widehat{F} \left(B_{g_{h,S}}(x_0,\lambda^{\frac{5}{2}}\widetilde{D}_{\epsilon}) \right)\subset B_{\widehat{g}_S}(\widehat{F}(x_0),\lambda^{2}\widetilde{D}_{\epsilon})$ if $\epsilon_1 \leq \overline{\epsilon}_1$ and $\varphi_{s_0+1}=e\phi_{-1}^{\ast}(\widehat{F}^{\ast}\widehat{\varphi}_S),$ we conclude that
\begin{equation*}
    \int_{B_{g_h}(x_0,\sqrt{e}\lambda^{\frac{5}{2}}\widetilde{D}_{\epsilon})}(\varphi_{s_0+1})_-^{2p_1}e^{-f_{h}}\omega_h^n\le C(g_{\sol})\varepsilon^{2p_1}.
\end{equation*}
If $\epsilon_1 \leq \overline{\epsilon}_1$, then Lemma \ref{lem:Dandrho} moreover yields
\begin{equation*}
    \Omega_h( \sqrt{e}\lambda^3\widetilde{D}_{\epsilon})\subset B_{g_h}(x_0,\sqrt{e}\lambda^{\frac{5}{2}}\widetilde{D}_{\epsilon}),
\end{equation*} and therefore
\begin{equation*}
 \int_{\Omega_h(\sqrt{e}\lambda^3\widetilde{D}_{\epsilon})}(\varphi_{s_0+1})_-^{2p_1}e^{-f_{h}}\omega_h^n\le C(g_{\sol})\varepsilon^{2p_1}.
\end{equation*}
Combined with \eqref{Lp bound on positive part}, this finishes the proof of the lemma.
\end{proof}

We also prove an analogous result for K\"ahler-Ricci flow near compact shrinkers, which will be used in the proof of Theorem \ref{thm:main3}. The proof is substantially easier since it does not require dealing with the behavior near spatial infinity, which was the major technical difficulty in proving Lemma \ref{lem:nonconcentration}.

\begin{lemma} \label{lem:nonconcentrationcompact} Suppose $(M_{\sol},J_{\sol},g_{\sol},f_{\sol})$ is a compact K\"ahler-Ricci shrinker and $(M,(\widetilde{g}_t)_{t\in [-1,0)})$ is a compact K\"ahler-Ricci flow with $\omega_{-1} \in c_1(M)$. Then there exist $\epsilon_c \in (0,1)$ and $p_c,A_c \in (1,\infty)$ such that the following holds.

Suppose there exists a biholomorphism $F:M_{\sol} \to M$ such that $g_s := e^s (F\circ \phi_{-s})^{\ast}\widetilde{\omega}_{-e^{-s}}$ satisfies 
\begin{equation*}
    \|h\|_H <\epsilon_c, \qquad \|g_{s_0}-g_h\|_{C^{\lfloor\epsilon_c^{-1} \rfloor}(M_{\sol},g_h)}<\epsilon_c
\end{equation*}
for some $s_0 \geq \epsilon_c^{-1}$ and $h\in H$,
and that $\varphi \in C^{\infty}(M_{\sol} \times [0,\infty))$ satisfies $\omega_s=\omega_h +\sqrt{-1}\partial \overline{\partial} \varphi_s$ and \eqref{modifiedPCMA2}, as well as 
\begin{equation*}
    \left( \int_{M_{\sol}} \varphi_{s_0}^2 d\nu_h \right)^{\frac{1}{2}} <\epsilon
\end{equation*}
for some $\epsilon \in (0,\epsilon_c]$. Then 
\begin{equation*}
    \left( \int_{M_{\sol}} |\varphi_{s_0+1}|^{2p_c}d\nu_h \right)^{\frac{1}{2p_c}} \leq A_c \epsilon.
\end{equation*}
\end{lemma}
\begin{proof}
    Because 
    \begin{equation*}
        (\partial_s -\Delta_{f_h})(e^{-s}\varphi_s)_+ \leq 0
    \end{equation*}
    in the sense of distributions, we can use the hypercontractivity of the semigroup associated to $(M_{\sol},g_{h},f_{h})$ to get
    \begin{equation*}
        \left( \int_{M_{\sol}} (\varphi_{s_0+1})_+^{2p_c} d\nu_h \right)^{\frac{1}{2p_c}} \leq e\left( \int_{M_{\sol}} \varphi_{s_0}^2 d\nu_h \right)^{\frac{1}{2}}
    \end{equation*}
    as in the proof of Lemma \ref{lem:nonconcentration}. Choose any $\widetilde{\nu} \in M_0$. Then by \cite{PerelmanKahler}, there exists $A \in (1,\infty)$ depending only on the flow $(M,(\widetilde{g}_t)_{t\in [-1,0)})$ such that if we write $d\widetilde{\nu}_t = (2\pi|t|)^{-n}e^{-\widetilde{f}_t}d\widetilde{g}_t$, then $|\widetilde{f}_t|\leq A.$ After possibly increasing $A$, we can also assume that $|f_h|\leq A$ for all $h\in H$. Given $s_0 \geq 0$, if we define $\widehat{\varphi},\widehat{\nu},S$ as in the proof of Lemma \ref{lem:nonconcentration} and choose $\epsilon_c \leq \overline{\epsilon}_c$, it therefore follows that
\begin{equation} \label{eq:coarsefcompact}
    C(A,p)^{-1}\int_{M_{\sol}} (\varphi_{s_0+\log\frac{1}{|t|}})_-^{2p}d\nu_h\leq \int_M (\widehat{\varphi}_t)_-^{2p} d\widehat{\nu}_t \leq C(A,p)\int_{M_{\sol}} (\varphi_{s_0+\log\frac{1}{|t|}})_-^{2p}d\nu_h
\end{equation}
for all $t\in [-1,S]$. Because 
\begin{equation*}
    (\partial_t - \Delta_{\widetilde{\omega}_t})(\widehat{\varphi}_t)_- \leq 0
\end{equation*}
in the sense of distributions, we apply hypercontractivity along the Ricci flow as in the proof of Lemma \ref{lem:nonconcentration} to obtain
\begin{equation*}
    \left( \int_{M} (\widehat{\varphi}_{S})_-^{2p_c} d\widehat{\nu}_S \right)^{\frac{1}{2p_c}} \leq \left( \int_{M} (\widehat{\varphi}_{-1})_-^2 d\widehat{\nu}_{-1} \right)^{\frac{1}{2}},
\end{equation*}
so the claim follows by combining with \eqref{eq:coarsefcompact}.
\end{proof}

\subsection{Decay proposition}

We now define
\begin{equation*}
    \mathcal{D}_h(\varphi,s):= \int_{\Omega_h(a\sqrt{s})} \varphi_s^2 d\nu_h + e^{-(1+\theta_1)\frac{a^2 s}{2}}, \qquad \widetilde{\mathcal{D}}_h(\varphi,s):= \sup_{s'\in [s-2,s]} \mathcal{D}_h(\varphi,s'),
\end{equation*}
with $a>0$ to be determined later.

The following proposition is a nonlinear analogue of the fact that any unstable mode of the linear equation \eqref{eq:intro:linear} can be removed by a change of gauge. It is proved by contradiction, following the general strategy of \cite[Proposition 3.5]{ChiuSzek}. 

In more detail, we suppose there are solutions $\varphi_i$ of the modified equation \eqref{eq:normalizedmodifiedflow} with respect to a sequence $h_i$ of reference metrics which converge to zero but whose corresponding functionals $\widetilde{\mathcal{D}}_{h_i}(\varphi_i,\cdot)$ do not decay after a change of gauge. We can then apply Lemma \ref{lem:nonconcentration} to the re-normalized potentials $\widehat{\varphi}_i$, allowing us to pass to the limit to obtain a solution of the drift heat equation \eqref{eq:intro:linear}, and where the weighted $L^2$ norms converge. By the analysis of the linear equation done in Section \ref{section:lineartheory}, the solution consisting only of stable modes has a definite decay rate. Subtracting unstable or neutral modes corresponding to pluriharmonic functions gives a new solution of \eqref{eq:normalizedmodifiedflow}, and up to first order, subtracting the remaining neutral modes corresponds to a change of model metric by the estimates of Section \ref{section:modelmetrics}. Modifying $\widehat{\varphi}_i$ accordingly gives the desired decay. The remaining assertions of the proposition follow by our almost-self-similarity assumption, elliptic estimates, and interpolation. These claims will be used to ensure we can iterate the decay lemma indefinitely in the proof of Theorem \ref{thm:main1}.

\begin{prop} \label{prop:decay} 
Fix an AC K\"ahler-Ricci shrinker $(M_{\sol},J_{\sol},g_{\sol},f_{\sol})$, and let $(\omega_{h})_{h\in H}$ be the corresponding family of model metrics. For any $Y \in (0,\infty)$, there exist $a=a(g_{\sol},Y)>0$ and $\epsilon_2=\epsilon_2(g_{\sol},Y)>0$ such that the following holds.

Suppose $(M^n,(\widetilde{\omega}_t)_{t\in [-T,0)})$ is a compact K\"ahler-Ricci flow with $T\geq 1$, and with reference conjugate heat kernel $(\widetilde\nu_t)_{t\in [-T,0)} \in M_0$ based at the singular time. Suppose that for some $s_0\geq \epsilon_2^{-1}$, there exists $h\in H$ and a solution $\varphi$ of 
\begin{equation}
\partial_{s}\varphi_{s}=\log\frac{(\omega_{h}+\sqrt{-1}\partial\overline{\partial}\varphi_{s})^{n}}{\omega_{h}^{n}}-\frac{1}{2}X\cdot\varphi_{s}+\varphi_{s},\label{eq:normalizedmodifiedflow}
\end{equation}
on $\mathcal{U}:=\cup_{s\in [0,\infty)}(\Omega(e^{\frac{2}{5}s})\times \{s\})$ such that for some open holomorphic embedding $F:\Omega(\epsilon_2^{-1})\to M$, we have:
\begin{enumerate}
    \item \label{prop:decayassumption1} $\omega_{s}:= e^{s}\phi_{-s}^{\ast}F^{\ast}\widetilde{\omega}_{-e^{-s}}=\omega_{h}+\sqrt{-1}\partial\overline{\partial}\varphi_{s}$ on $\mathcal{U}$;
    
    \item \label{prop:decayassumption2} $\|h\|_{H}<\epsilon_2$;

    \item \label{prop:decayassumption3} $\sup_{s\in [s_0-2,s_0]} \|\omega_{s}-\omega_h\|_{C^{\lfloor \epsilon_2^{-1} \rfloor}(\Omega_h(a\sqrt{s}),g_h)}<\epsilon_2$;
    
    \item \label{prop:decayassumption4} $\widetilde{\mathcal{D}}_h(\varphi,s_0)\le \exp \left(-\frac{2a^2s_0}{(2+\theta_1)^2} \right)$, where $\theta_1 \in (0,\frac{1}{10^2})$ is as in Lemma \ref{lem:nonconcentration};

    \item \label{prop:decayassumption5} $f_{s_0}(x_0) \leq f_{\sol}(x_0)+W+1$, where $W$ is the entropy of $(M_{\sol},g_{\sol})$ and $f_s = \phi_{-s}^{\ast}F^{\ast}\widetilde{f}_{-e^{-s}}$;

    \item \label{prop:decayassumption6} $\widetilde{\nu}$ is $(\epsilon_2,e^{-\frac{s_0}{2}})$-self-similar and $\mathcal{N}_{\widetilde{\nu}}(1)\geq -Y$.
\end{enumerate}
Let $(\Psi_\alpha)_{\alpha=0}^{N_-}$ be the orthonormal basis from Lemma \ref{lem:eigenvalues}.
\noindent Then there exist $h'\in H$ and $b_0,...,b_{N_-} \in \mathbb{R}$ such that
\begin{equation} \label{eq:decayboundsforcoeffsinstatement} \|h-h'\|^2_H +\sum_{\alpha=0}^{N_-} b_{\alpha}^2  \leq \widetilde{\mathcal{D}}_{h}(\varphi,s_0), 
\end{equation}
and such that
\begin{equation*}
    \varphi_s':= \varphi_s +(u_{h}-u_{h'})+\sum_{\alpha=0}^{N_-} e^{(1-\lambda_{\alpha})(s-s_0)}b_{\alpha}\Psi_{\alpha}
\end{equation*}
is a solution of 
\begin{equation} \label{eq:modifiedKRFprime}
    \partial_s \varphi_s' = \log \frac{(\omega_{h'}+\sqrt{-1}\partial \overline{\partial}\varphi_s')^n}{\omega_{h'}^n} - \frac{1}{2}X\cdot\varphi_s'+\varphi_s'    
\end{equation}
on $\mathcal{U}$ such that the following hold:
\begin{enumerate}[label=(\alph*)]
    \item \label{prop:decayconclusion1} $\omega_s = \omega_{h'}+\sqrt{-1}\partial \overline{\partial} \varphi_s'$ on $\mathcal{U}$

    \item \label{prop:decayconclusion2} $\widetilde{\mathcal{D}}_{h'}(\varphi',s_0+1)\leq e^{-\frac{a^2}{2}}\widetilde{\mathcal{D}}_h(\varphi,s_0)$,

    \item \label{prop:decayconclusion3} $f_{s_0+1}(x_0)\leq f_{\sol}(x_0)+W+\frac{1}{2}$.
    \item \label{prop:decayconclusion4} there exists $\delta_2=\delta_2(\epsilon_2,g_{\sol},Y)>0$ such that if in addition $\widetilde{\mathcal{D}}_h(\varphi,s_0)\leq \delta_2$, then 
    \begin{equation*}
        \sup_{s\in [s_0-1,s_0+1]}\|\omega_s-\omega_{h'}\|_{C^{\lfloor \epsilon_2^{-1} \rfloor}(\Omega_{h'}(a\sqrt{s}),g_{h'})}<\epsilon_2.
    \end{equation*}
\end{enumerate}
\end{prop}

\begin{proof} 
The assertion \ref{prop:decayconclusion3} follows directly by Corollary \ref{coro: uniform value of min f_h} and Lemma \ref{lem:improvementofcomparability} if $\varepsilon_2$ is sufficiently small. 

Now set $a^2:=\frac{1}{2}(\lambda_{N_0+1}-1)$, where $\lambda_{N_0+1}$ is as in Lemma \ref{lem:eigenvalues}. We prove assertions \ref{prop:decayconclusion1} and \ref{prop:decayconclusion2} by way of contradiction:
that is, suppose there exist compact K\"ahler-Ricci flows $(M_i,J_i,(g_{i,t})_{t\in [-T_i,0)})$ with $T_i \geq 1$ and reference measures $\widetilde{\nu}^i \in (M_i)_0$, and suppose there are $\epsilon_i \to 0$, $h_{i}\in H$, $s_{i} \geq \epsilon_i^{-1}$ such that 
\begin{equation} \label{eq:decaycontrahypothesis}\sup_{s\in [s_i-2,s_i]} \|\omega_{i,s} - \omega_{h_i}\|_{C^{\lfloor \epsilon_i^{-1} \rfloor }(\Omega_{h_i}(a\sqrt{s}),g_{h_i})}<\epsilon_i, \qquad \|h_i\|_H <\epsilon_i, \qquad f_{i,s_i}(x_0)\leq f_{\sol}(x_0)+W+1, 
\end{equation}
$(M_i,(g_{i,t})_{t\in [-T_i,0)})$ is $(\epsilon_i,e^{-\frac{s_i}{2}})$-self-similar at $\widetilde{\nu}^i$, $\mathcal{N}_{\widetilde{\nu}^i}(1)\geq -Y$ and that there are solutions $(\varphi_{i,s})_{s\in [0,\infty)}$ to (\ref{eq:normalizedmodifiedflow}) with $h$ replaced by $h_i$, which satisfy $\omega_{i,s} = \omega_{h_i}+\sqrt{-1}\partial \overline{\partial} \varphi_{i,s}$ on $\mathcal{U}$, and $\widetilde{\mathcal{D}}_{h_i}(\varphi_i,s_i) \leq e^{-\frac{2a^2 s_i}{(2+\theta_1)^2}},$ but such that there do not exist $h'\in H$ and $b_0,...,b_{N_-} \in \mathbb{R}$ such that conclusions \ref{prop:decayconclusion1},\ref{prop:decayconclusion2} and the bound \eqref{eq:decayboundsforcoeffsinstatement} hold.  

Set $\kappa_i:= \sqrt{\widetilde{\mathcal{D}}_{h_i}(\varphi_i,s_i)}$, and define $\widehat{\varphi}_{i,s}:=\kappa_{i}^{-1}\varphi_{i,s+s_i}$, so that 
\begin{equation} \label{eq:L2boundalongsequence}
\sup_{s\in[-2,0]}\left(\int_{\Omega_{h_i}(a\sqrt{s_i+s})}\widehat{\varphi}_{i,s}^{2}d\nu_{h_i} \right)^{\frac{1}{2}}\leq 1.
\end{equation}  
Applying Lemma \ref{lem:nonconcentration}, with $\varphi$ replaced by $\varphi_{i,s+s_i-1}$, $D$ replaced by $a\sqrt{s_i+s-1}$, and $\epsilon$ replaced by $\kappa_i$, noting that $f_{s+s_i-1}(x_0)\leq f_{\sol}(x_0)+W+\frac{1}{2}$ by Lemma \ref{lem:improvementofcomparability}, we then have 
\begin{equation} \label{eq:Lpboundalongsequence}
    \sup_{s\in [-1,1]} \left( \int_{\Omega_{h_i}(e^{\frac{1}{4}}D_i)} |\widehat{\varphi}_{i,s}|^{2p_1} d\nu_{h_i} \right)^{\frac{1}{2p_1}} \leq C
\end{equation}
where $p_1>1$ is as in Lemma \ref{lem:nonconcentration}, and 
\begin{equation*}
    D_i :=\sqrt{-4\log \kappa_i} \geq \frac{2}{2+\theta_1}a\sqrt{s_i}
\end{equation*}
by assumption \ref{prop:decayassumption4}.
It follows that
\begin{equation*}
    e^{\frac{1}{4}}D_i\ge \frac{(2+2\theta_1)a\sqrt{s_i}}{2+\theta_1}\ge a\sqrt{s_i+s}
\end{equation*}
for all $s\in [-1,1]$, when $i\in \mathbb{N}$ is sufficiently large.
We rewrite \eqref{eq:normalizedmodifiedflow} as
\begin{align*}
    \partial_s \widehat{\varphi}_{i,s} + \frac{1}{2}X\cdot\widehat{\varphi}_{i,s}-\widehat{\varphi}_{i,s} = & \frac{1}{\kappa_i} \int_0^1 \partial_\tau \log \left( \frac{(\omega_{h_i}+\sqrt{-1}\partial \overline{\partial} (\tau\varphi_{i,s_i+s}))^n}{\omega_{h_i}^n}\right) d\tau \\
    =& \mathcal{F}_{i,s}(\sqrt{-1}\partial \overline{\partial}\widehat{\varphi}_{i,s}),
\end{align*}
where we define the vector bundle homomorphism $\mathcal{F}_i$ by
\begin{equation} \label{eq:ellipticfunctional}
    \mathcal{F}_{i,s}(\alpha):= n\int_0^1 \frac{(\tau\omega_{i,s}+(1-\tau)\omega_{h_i})^{n-1}\wedge \alpha}{(\tau\omega_{i,s}+(1-\tau)\omega_{h_i})^n} d\tau.
\end{equation}
By \eqref{eq:decaycontrahypothesis}, Proposition \ref{prop:modelmetrics}\ref{modelmetrics3}, and Lemma \ref{lem:pseudoapplication} \ref{lem:pseudoapplication1}, we moreover have that $\mathcal{F}_{i,s}$ converges locally smoothly on $M_{\sol} \times [-2,1]$ to $\alpha \mapsto \operatorname{tr}_{\omega_{\sol}}(\alpha)$ as $i\to \infty$. For each compact set $K\subset M_{\sol}$, $\mathcal{F}_{i,s}$ is uniformly elliptic on $K\times [-2,1],$ so we may combine interior estimates \cite[Theorem 7.22]{Lieberman} for linear parabolic equations with \eqref{eq:L2boundalongsequence} and \eqref{eq:Lpboundalongsequence} to get 
\begin{equation*}
\sup_{s\in[-2+\sigma ,1]}\sup_{ \Omega(\sigma^{-1})}|(\nabla^{g_{\sol}})^{k}\widehat{\varphi}_{i,s}|_{g_{\sol}}\leq C_{k}(\sigma)
\end{equation*}
for any $\sigma \in (0,1)$ when $i\geq \underline{i}(\sigma,k)\in\mathbb{N}$. We may therefore pass to a subsequence so that 
\begin{equation} \label{eq:decaysmoothconvergence}
    \widehat{\varphi}_i \to \psi \qquad \text{ in } C_{\operatorname{loc}}^{\infty}(M_{\sol}\times (-2,1])
\end{equation}
for some solution $\psi\in C^{\infty}(M_{\sol}\times (-2,1])$ of the drift heat equation
\begin{equation} \label{eq:driftheateqdecayprop}
    \partial_s \psi_s = \Delta_{\omega_{\sol}}\psi_s - \frac{1}{2}X\cdot\psi_s + \psi_s.
\end{equation}
By \eqref{eq:L2boundalongsequence},\eqref{eq:Lpboundalongsequence}, and \eqref{eq:decaysmoothconvergence}, we have 
\begin{equation} \label{eq:decayLpboundsonlimit}
    \sup_{s\in (-2,0]} \int_{M_{\sol}} \psi_s^2 d\nu_{\sol} \leq 1, \qquad \sup_{s\in [-1,1]} \int_{M_{\sol}}|\psi_s|^{2p_1} d\nu_{\sol} \leq C.
\end{equation}
By \eqref{eq:decayLpboundsonlimit}, Lemma \ref{lem:lineartheory} and Lemma \ref{lem:decayfordriftheateq}, we can define 
\begin{equation*} \psi_s^{\leq 0} = k+\sum_{\alpha=0}^{N_-} e^{-(\lambda_{\alpha}-1)s}b_{\alpha}\Psi_{\alpha},\quad s\in (-2,1] \end{equation*}
where $k\in H$, $N_- \leq N_0$ satisfies $\lambda_{N_-} \leq1$, 
\begin{equation*}
    \|k\|_H^2 + \sum_{\alpha=0}^{N_-}b_{\alpha}^2\leq 1,
\end{equation*}
and $\Psi_{\alpha}$ are pluriharmonic functions satisfying $\mathcal{L}_{\frac{1}{2}X}\Psi_{\alpha} = \lambda_{\alpha} \Psi_{\alpha}$ such that $\psi^{>0}:=\psi-\psi^{\le0}$ satisfies:
\begin{equation}\label{eq: decay of L^2-psi>0}
    \left(\sup_{s\in [-1,1]} \int_{M_{\sol}} (\psi_s^{>0})^2 d\nu_{\sol} \right)^{\frac{1}{2}} \leq e^{-(\lambda_{N_0+1}-1)} \left( \sup_{s\in (-2,0]} \int_{M_{\sol}} (\psi_s^{>0})^2 d\nu_{\sol} \right)^{\frac{1}{2}} \leq e^{-(\lambda_{N_0+1}-1)}.
\end{equation}
 Define \begin{equation*} \psi_{i,s} := \varphi_{i,s} -\left( u_{h_i+\kappa_i k} -u_{h_i}+ \kappa_i\sum_{\alpha=0}^{N_-} e^{-(\lambda_{\alpha}-1)(s-s_i)}b_{\alpha}\Psi_\alpha  \right),\quad s\in [s_i-2,s_i+1]\end{equation*} 
so that
\begin{equation} \label{eq:changingmodels}
    \omega_{h_i+\kappa_i k} + \sqrt{-1}\partial \overline{\partial}\psi_{i,s} = \omega_{i,s},
\end{equation}
holds on $\mathcal{U}$.
By Proposition \ref{prop:modelidentities} \ref{prop:modelidentities4}, \eqref{eq:changingmodels}, and $\mathcal{L}_{\frac{1}{2}X}\Psi_{\alpha} = \lambda_{\alpha} \Psi_{\alpha}$, we moreover have
\begin{align*}
    \partial_s \psi_{i,s} =& \log\left( \frac{\omega_s^n}{\omega_{h_i+\kappa_i k}^n} \right)-\frac{1}{2}X\cdot\psi_{i,s}+\psi_{i,s}+\log \left( \frac{\omega_{h_i+\kappa_i k}^n}{\omega_{h_i}^n} \right) +\kappa_i \sum_{\alpha=0}^{N_-} (\lambda_{\alpha}-1)e^{-(\lambda_{\alpha}-1)(s-s_i)}b_{\alpha}\Psi_{\alpha} \\
    &-\frac{1}{2}X\cdot\left( u_{h_i+\kappa_i k}-u_{h_i} + \kappa_i \sum_{\alpha=0}^{N_-} e^{-(\lambda_{\alpha}-1)(s-s_i)}b_{\alpha}\Psi_{\alpha} \right)+  u_{h_i+\kappa_i k}-u_{h_i} + \kappa_i \sum_{\alpha=0}^{N_-} e^{-(\lambda_{\alpha}-1)(s-s_i)}b_{\alpha}\Psi_{\alpha} \\
    =& \log\left( \frac{\omega_s^n}{\omega_{h_i+\kappa_i k}^n} \right)-\frac{1}{2}X\cdot\psi_{i,s}+\psi_{i,s}
    + \log \left( \frac{e^{u_{h_i+\kappa_i k}-\frac{1}{2}Xu_{h_i+\kappa_i k}}\omega_{h_i+\kappa_i k}^n}{e^{u_{h_i}-\frac{1}{2}X\cdot u_{h_i}}\omega_{h_i}^n} \right) \\
    &+ \kappa_i \sum_{\alpha=0}^{N_-} e^{-(\lambda_{\alpha}-1)(s-s_i)}b_{\alpha}(\lambda_{\alpha} \Psi_{\alpha} - \frac{1}{2}X\cdot\Psi_{\alpha} )\\
    =& \log\left( \frac{(\omega_{h_i+\kappa_i k}+\sqrt{-1}\partial \overline{\partial}\psi_{i,s})^n}{\omega_{h_i+\kappa_i k}^n} \right)-\frac{1}{2}X\cdot\psi_{i,s}+\psi_{i,s}.
\end{align*}
Next, we define $\widehat{\psi}_{i,s}:= \kappa_i^{-1}\psi_{i,s+s_i}$, so that 
\begin{equation} \label{eq:decaysmoothconvergenceofmod}
    \widehat{\psi}_{i,s}-\psi_s^{>0} = (\widehat{\varphi}_{i,s}-\psi_s)-\left( \frac{u_{h_i+\kappa_i k}-u_{h_i}}{\kappa_i}-k \right) \to 0 \qquad \text{ in } C_{\operatorname{loc}}^{0}(M_{\sol} \times [-1,1])
\end{equation}
as $i\to \infty$ by \eqref{eq:decaysmoothconvergence} and Proposition \ref{prop:modelmetrics} \ref{modelmetrics4}. For any fixed $A \in (1,\infty)$, \eqref{eq:decaysmoothconvergenceofmod} yields
\begin{align} \label{eq:decayconvergencefixedradius}
    \sup_{s\in [-1,1]} \left| \int_{\Omega(A)} \widehat{\psi}_{i,s}^2 d\nu_{h_i+\kappa_i k} - \int_{\Omega(A)} (\psi_s^{>0})^2 d\nu_{\sol} \right| \leq \Psi(i^{-1}|A).
\end{align}
We are then left to check uniform control of the tails. Note that, on $\Omega_{h_i+\kappa_ik}(a\sqrt{s_i+s}),$ we have $f_{\sol}\leq Cs_i,$ while $\kappa_i\leq e^{-cs_i},$ so the pointwise estimates in Proposition \ref{prop:modelmetrics} \ref{modelmetrics1} yield
\begin{equation*}
\sup_{\Omega_{h_i+\kappa_i k}(a\sqrt{s_i+s})} |u_{h_i+\kappa_ik}-u_{h_i}|\leq Ce^{-cs_i}(1+s_i)\to 0
\end{equation*}
as $i\to \infty$. This, together with \eqref{eq:Lpboundalongsequence}, Proposition \ref{prop:modelidentities} \ref{prop:modelidentities4}, and H\"older's inequality, gives us
\begin{equation} \label{eq:decaytailalongsequence}
\begin{split}
     \hspace{-10 mm}\sup_{s\in [-1,1]} \int_{\Omega_{h_i+\kappa_i k}(a\sqrt{s_i+s})\setminus \Omega(A)} \widehat{\varphi}_{i,s}^2 d\nu_{h_i+\kappa_i k} \leq&C\left( \int_{M_{\sol}\setminus \Omega_{h_i}(\frac{1}{2}A)} e^{p'(u_{h_i}-u_{h_i+\kappa_ik})}d\nu_{h_i} \right)^{\frac{1}{p'}}  \leq  \Psi(A^{-1}),
\end{split}
\end{equation}
where $p':=\frac{p_1}{p_1-1}$. Now, given the correction term
\begin{equation*}
\Theta_{i,s}:=\left( \frac{u_{h_i+\kappa_i k}-u_{h_i}}{\kappa_i} \right) + \sum_{\alpha=0}^{N_-} e^{-(\lambda_{\alpha}-1)s}b_{\alpha}\Psi_\alpha,
\end{equation*}
we have $\widehat{\psi}_{i,s}=\widehat{\varphi}_{i,s}-\Theta_{i,s}.$ Hence, Proposition \ref{prop:modelmetrics} and Lemma \ref{lem:eigenfunctiongrowth} \ref{lem:eigenfunctiongrowth1} yield 
\begin{align*}
    |\Theta_{i,s}| \leq C\|k\|_H f_{\sol}+C\frac{1}{\kappa_i}\|\kappa_i k\|_H(\kappa_i\|k\|_H+\|h_i\|_H) f_{\sol}+ C\sum_{\alpha=0}^{N_-}f_{\sol}^{\lambda_{\alpha}}
    \leq C\left( \int_{M_{\sol}} \psi_s^2 d\nu_{\sol} \right)^{\frac{1}{2}} f_{\sol}\leq Cf_{\sol}
\end{align*}
for all $s\in [-1,1]$, we can argue as before to obtain
\begin{equation}\label{eq: decaytailcorrectionterm}
    \sup_{s\in [-1,1]} \int_{\Omega_{h_i+\kappa_i k}(a\sqrt{s_i+s})\setminus \Omega(A)} |\Theta_{i,s}|^2 d\nu_{h_i+\kappa_i k} \leq \Psi(A^{-1}).
\end{equation}
Similarly, passing \eqref{eq:decayLpboundsonlimit} to the limit and using H\"older's inequality yields
\begin{equation}\label{eq:decaytailonlimit} 
\sup_{s\in [-1,1]} \int_{M_{\sol}\setminus \Omega(A)} \psi_s^2 d\nu_{\sol} \leq \Psi(A^{-1}).
\end{equation}
Using $|\psi_s-\psi_s^{>0}|\leq Cf_{\sol}$ for all $s\in [-1,1]$, and arguing as above, we obtain from
\eqref{eq:decayconvergencefixedradius}, \eqref{eq:decaytailalongsequence}, \eqref{eq: decaytailcorrectionterm}, \eqref{eq:decayLpboundsonlimit}, and \eqref{eq:decaytailonlimit} that
\begin{equation}\label{eq:convergenceofmodifiedpotentials}
    \lim_{i \to \infty} \sup_{s\in [-1,1]}\left| \int_{\Omega_{h_i+\kappa_i k}(a\sqrt{s_i+s}) }\widehat{\psi}_{i,s}^2 d\nu_{h_i+\kappa_i k} - \int_{M_{\sol}}(\psi_s^{>0})^2 d\nu_{\sol} \right| =0.
\end{equation}
Combining \eqref{eq: decay of L^2-psi>0} and \eqref{eq:convergenceofmodifiedpotentials}, we conclude that
\begin{equation*}
  \widetilde{\mathcal{D}}_{h_i+\kappa_ik}(\psi_i,s_i+1)= \sup_{s\in [s_i-1,s_i+1]}\left( \int_{\Omega_{h_i+\kappa_i k}(a\sqrt{s})} \psi_{i,s}^2 d\nu_{h_i+\kappa_i k} +e^{-\frac{(1+\theta_1)a^2s}{2}}\right) \leq e^{-\frac{a^2}{2}} \kappa_i^2
\end{equation*}
for sufficiently large $i\in \mathbb{N}$, yielding the desired contradiction.

We have thus shown that there exists $\varphi'$ satisfying conclusions \ref{prop:decayconclusion1}-\ref{prop:decayconclusion3}. It remains to show that \ref{prop:decayconclusion4} is a consequence of \ref{prop:decayconclusion1}-\ref{prop:decayconclusion3}.
Applying Lemma \ref{lem:nonconcentration} as in the derivation of \eqref{eq:Lpboundalongsequence}, and using H\"older's inequality, we obtain
\begin{equation*}
   \sup_{s\in [s_0-1,s_0+1]}\int_{\Omega_h(e^{\frac{1}{4}}\widetilde{D})}|\varphi_s|^{2p_1}d\nu_h\le C(g_{\sol},Y)\widetilde{\mathcal{D}}_h^{p_1}(\varphi,s_0),
\end{equation*}
where $\widetilde{D}^2=-2\log\widetilde{\mathcal{D}}_h(\varphi,s_0)\ge \frac{4a^2s_0}{(2+\theta_1)^2}$. In this case, if $s_0>2$ is sufficiently large, for all $s\in [s_0-2,s_0]$, we then have
\begin{equation*}
   \Omega_h \left((1+\frac{\theta_1}{4+\theta_1})a\sqrt{s+1} \right)\subset \Omega_h \left(\frac{2+2\theta_1}{2+\theta_1}a\sqrt{s}\right)\subset  \Omega_h(e^{\frac{1}{4}}\widetilde{D}).
\end{equation*}
Setting $\ell=1+\frac{\theta_1}{4+\theta_1}$,
we conclude that
\begin{equation*} 
   \sup_{s\in[s_0-1,s_0+1]} \int_{\Omega_h(a\ell\sqrt{s})}|\varphi_s|^{2p_1}\omega_h^n\leq (2\pi)^n\sup_{s\in [s_0-1,s_0+1]}e^{\frac{a^2\ell^2}{2}s}\int_{\Omega_h(a\ell\sqrt{s})}|\varphi_s|^{2p_1}d\nu_h\le C(g_{\sol},Y)e^{\frac{a^2\ell^2}{2}s_0}\widetilde{\mathcal{D}}^{p_1}_h(\varphi,s_0).
\end{equation*}
Combining our hypotheses $\widetilde{\mathcal{D}}_h(\varphi,s_0)\le e^{-\frac{2a^2s_0}{(2+\theta_1)^2}}$ and $\widetilde{\mathcal{D}}_h(\varphi,s_0)\le \delta$ yields
\begin{equation*}
    \widetilde{\mathcal{D}}^{p_1}_h(\varphi,s_0)\le e^{-\frac{a^2 \ell^2s_0}{2}}\delta^{p_1-\frac{\ell^2(2+\theta_1)^2}{4}}.
\end{equation*}
Hence, we may combine expressions to obtain
\begin{equation} \label{eq:Lpnormsmallindecay}
     \sup_{s\in [s_0-1,s_0+1]}\int_{\Omega_h(a\ell \sqrt{s})}|\varphi_s|^{2p_1}\omega_h^n\le C\delta^{p_1-\frac{\ell^2(2+\theta_1)^2}{4}} \leq C(g_{\sol},Y)\delta^{\gamma_0},
\end{equation}
where $\gamma_0:=p_1-\frac{\ell^2(2+\theta_1)^2}{4}$. On the other hand, assumption \ref{prop:decayassumption3} and Lemma \ref{lem:pseudoapplication} give
\begin{equation} \label{eq:metriccloseindecay}
     \sup_{s\in[s_0-1,s_0+1]} \|\omega_{s}-\omega_h\|_{C^{10}(\Omega_h(a\ell \sqrt{s}),g_h)}\le  \sup_{s\in[s_0-1,s_0+1]} \|\omega_s-\omega_h\|_{C^{10}(\Omega_h(a(1+\theta_1)\sqrt{s-1}),g_h)}\le \Psi(\epsilon_2|g_{\sol},Y).
\end{equation}
Taking $\epsilon_2 \leq \overline{\epsilon}_2(g_{\sol},Y)$, we use \eqref{eq:Lpnormsmallindecay} and \eqref{eq:metriccloseindecay} to apply Lemma \ref{lem:L2toLinfty} with $r=(\epsilon_2^{-1}\delta^{\frac{\gamma_0}{2p_1}})^{\frac{1}{2+\frac{n}{p_1}}}$ to get
\begin{equation} \label{eq:pointwisesmallindecay}
    \sup_{s\in [s_0-1,s_0+1]} \sup_{\Omega_h(a\ell \sqrt{s}-\frac{1}{2})} |\varphi_s|\leq C(g_{\sol},Y)\left( \frac{\delta^{\frac{\gamma_0}{2p_1}}}{\epsilon_2}\right)^{\frac{2}{2+\frac{n}{p_1}}}.
\end{equation}
Next, we combine \eqref{eq:pointwisesmallindecay} and \eqref{eq:metriccloseindecay} with Lemma \ref{lem: schauder estimates} to obtain
\begin{equation} \label{eq:higherorderindecay} \sup_{s\in [s_0-1,s_0+1]} \|\varphi_s\|_{C^8(\Omega_h(a\ell \sqrt{s}-\frac{3}{4}),g_h)} \leq 1
\end{equation}
if $\epsilon_2 \leq \overline{\epsilon}_2(g_{\sol},Y).$ By \eqref{eq:higherorderindecay}, \eqref{eq:pointwisesmallindecay}, and Lemma \ref{lem:interpolation}, 
we conclude that
\begin{equation*}
   \begin{split}
        \sup_{s\in[s_0-1,s_0+1]}\|\varphi_s\|_{C^{4}(\Omega_h(a\ell \sqrt{s}-1),g_h)} \leq \Psi(\delta|g_{\sol},Y)
   \end{split}
\end{equation*}
if $\epsilon_2 \leq \overline{\epsilon}_2(g_{\sol},Y)$. In particular, we have
\begin{equation*}
   \begin{split}
        \sup_{s\in[s_0-1,s_0+1]}\|\omega_s-\omega_h\|_{C^{2}(\Omega_h(a\ell \sqrt{s}-1),g_{h})}
        \le \Psi(\delta).
   \end{split}
\end{equation*}
If $\epsilon_2 \leq \overline{\epsilon}_2(g_{\sol})$, then \eqref{eq:containmentofsublevelsets} gives $\Omega_{h'}(a\frac{1+\ell}{2}\sqrt{s})\subset \Omega_h( a\ell \sqrt{s}-1)$, hence
\begin{equation*}
   \begin{split}
        \sup_{s\in[s_0-1,s_0+1]}\|\omega_s-\omega_{h'}\|_{C^{2}(\Omega_{h'}(a\frac{1+\ell}{2}\sqrt{s}),g_{h'})}
        \le \Psi(\delta)+\sup_{s\in[s_0-1,s_0+1]}\|\omega_h-\omega_{h'}\|_{C^{2}(\Omega_{h'}(a\frac{1+\ell}{2}\sqrt{s}),g_{h'})}.
   \end{split}
\end{equation*}
Combining Proposition \ref{prop:modelmetrics} \ref{modelmetrics3} and $\|h-h'\|_H^2 \leq C\widetilde{\mathcal{D}}_h(\varphi,s_0) \leq C\delta$, we get
\begin{equation*}
   \begin{split}
        \sup_{s\in[s_0-1,s_0+1]}\|\omega_s-\omega_{h'}\|_{C^{2}(\Omega_{h'}(a\frac{1+l}{2}\sqrt{s}),g_{h'})}
        \le \Psi(\delta).
   \end{split}
\end{equation*}
The claim then follows by taking $\delta \leq \overline{\delta}(\epsilon_2)$ and applying Lemma \ref{lem:pseudoapplication} \ref{lem:pseudoapplication1}.
\end{proof}

We also prove an analogous result in the setting of compact Ricci shrinkers, whose proof is again similar to the asymptotically conical case, but easier. For $\psi \in C^{\infty}(M_{\sol})$, $\varphi \in C^{\infty}(M_{\sol}\times [0,\infty))$, we define
\begin{equation*}
    \mathcal{D}_{c,h}(\psi):=\int_{M_{\sol}} \psi^2 d\nu_{h}, \qquad \widetilde{\mathcal{D}}_{c,h}(\varphi,s):=\sup_{s'\in [s-2,s]}\mathcal{D}_{c,h}(\varphi_{s'}).
\end{equation*}

\begin{prop} \label{prop:compactdecay}
Suppose $(M^n,J,(\widetilde{\omega}_t)_{t\in [-1,0)})$ is a compact K\"ahler-Ricci flow with $\omega_{-1}\in c_1(M)$, and suppose there is a K\"ahler-Ricci shrinker $(g_{\sol},f_{\sol})$ on $(M,J)$. Then there exist $a_c,\epsilon_c>0$ depending only on this data such that the following holds. 

Suppose that for some $s_0\geq \epsilon_c^{-1}$, there exists $h\in H$ and a solution $\varphi$ of \eqref{eq:normalizedmodifiedflow} on $M \times [0,\infty)$ satisfying:
\begin{enumerate}
    \item \label{prop:compactdecayassumption1} $\omega_{s}:= e^{s}\phi_{-s}^{\ast}\widetilde{\omega}_{-e^{-s}}=\omega_{h}+\sqrt{-1}\partial\overline{\partial}\varphi_{s}$ on $M_{\sol}\times [0,\infty)$,
    
    \item \label{prop:compactdecayassumption2} $\|h\|_{H}<\epsilon_c$,

    \item \label{prop:compactdecayassumption3} $\sup_{s\in [s_0-2,s_0]} \|\omega_{s}-\omega_h\|_{C^{\lfloor \epsilon_c^{-1} \rfloor}(M,g_h)}<\epsilon_c$,
    
    \item \label{prop:compactdecayassumption4} $\widetilde{\mathcal{D}}_{c,h}(\varphi,s_0)\le \epsilon_c$,
\end{enumerate}
\noindent Then there exist $h'\in H$ and $b\in \mathbb{R}$ such that
\begin{equation} \label{eq:compactdecayboundsforcoeffsinstatement} \|h-h'\|^2_H +b^2 \leq 2\widetilde{\mathcal{D}}_{c,h}(\varphi,s_0),    
\end{equation}
and such that
\begin{equation*}
    \varphi_s':= \varphi_s +(u_{h}-u_{h'})+ e^{s-s_0}b
\end{equation*}
is a solution of \eqref{eq:modifiedKRFprime} such that the following hold:
\begin{enumerate}[label=(\alph*)]
    \item \label{prop:compactdecayconclusion1} $\omega_s = \omega_{h'}+\sqrt{-1}\partial \overline{\partial} \varphi_s'$ on $M\times [0,\infty)$

    \item \label{prop:compactdecayconclusion2} $\widetilde{\mathcal{D}}_{c,h'}(\varphi',s_0+1)\leq e^{-\frac{a_c^2}{2}}\widetilde{\mathcal{D}}_{c,h}(\varphi,s_0)$,

    \item \label{prop:compactdecayconclusion4} there exists $\delta_c=\delta_c(\epsilon_c)>0$ such that if in addition $\widetilde{\mathcal{D}}_{c,h}(\varphi,s_0)\leq \delta_c$, then 
    \begin{equation*}
        \sup_{s\in [s_0-1,s_0+1]}\|\omega_s-\omega_{h'}\|_{C^{\lfloor \epsilon_c^{-1} \rfloor}(M_{\sol},g_{h'})}<\epsilon_c.
    \end{equation*}
\end{enumerate}
\end{prop}
\begin{proof}
    The proof of \ref{prop:compactdecayconclusion1}, \ref{prop:compactdecayconclusion2} proceeds by contradiction as in Proposition \ref{prop:decay}. Set $a_c^2:=\frac{1}{2}(\lambda_{N_0+1}-1)$, where $\lambda_{N_0+1}$ is as in Lemma \ref{lem:eigenvalues}. Suppose by way of contradiction there exist $\epsilon_i \searrow 0$, $s_i \in [\epsilon_i^{-1},\infty)$, $h_i \in H$, and $\varphi_i \in C^{\infty}(M\times[0,\infty))$ such that $\|h_i\|_H <\epsilon_i$, $\varphi_i$ solves \eqref{eq:normalizedmodifiedflow} with $h$ replaced by $h_i$, and
    \begin{equation*}
        \sup_{s\in [s_i-2,s_i]} \|\omega_s-\omega_{h_i}\|_{C^{\lfloor \epsilon_i^{-1}\rfloor}(M,g_{h_i})}<\epsilon_i, \qquad \kappa_i^2:=\widetilde{\mathcal{D}}_{c,h_i}(\varphi_i,s_i)\leq \epsilon_i,
    \end{equation*}
    but there are no $h'\in H$ and $b\in \mathbb{R}$ with $\|h_i-h'\|_H^2 +b^2 \leq 2\kappa_i^2$ such that $\varphi_s':=\varphi_{i,s}+(u_{h_i}-u_{h'})+e^{s-s_i}b$ satisfies 
    \begin{equation*}
        \omega_s = \omega_{h'}+\sqrt{-1}\partial \overline{\partial} \varphi_s', \qquad \widetilde{\mathcal{D}}_{c,h'}(\varphi',s_i+1)\leq e^{-\frac{a_c^2}{2}}\kappa_i^2.
    \end{equation*}
Set $\widehat{\varphi}_{i,s}:=\kappa_i^{-1}\varphi_{i,s+s_i}$, and apply Lemma \ref{lem:nonconcentrationcompact} with $\varphi$ replaced by $\varphi_{i,s+s_i-1}$, obtaining
\begin{equation} \label{eq:compactLpbounds}
\sup_{s\in [-2,0]} \int_M \widehat{\varphi}_{i,s}^2 d\nu_{h_i} \leq 1, \qquad \sup_{s\in [-1,1]}\int_M |\widehat{\varphi}_{i,s}|^{2p_c}d\nu_{h_i}\leq A_c,
\end{equation}
where $p_c,A_c$ are as in Lemma \ref{lem:nonconcentrationcompact}. Recalling that
\begin{equation*}
    \partial_s \widehat{\varphi}_{i,s} = \mathcal{F}_{i,s}(\sqrt{-1}\partial \overline{\partial}\widehat{\varphi}_{i,s})-\frac{1}{2}X\widehat{\varphi}_{i,s}+\widehat{\varphi}_{i,s},
\end{equation*}
where $\mathcal{F}_{i,s}$ are the uniformly elliptic operators defined in \eqref{eq:ellipticfunctional}, we may combine Lemma \ref{lem:pseudoapplication} \ref{lem:pseudoapplication1} and interior parabolic estimates \cite[Theorem 7.22]{Lieberman} with \eqref{eq:compactLpbounds} to pass to a subsequence so that $\widehat{\varphi}_i \to \psi$ in $C_{\operatorname{loc}}^{\infty}(M\times (-2,1])$ for some solution $\psi \in C^{\infty}(M\times (-2,1])$ of \eqref{eq:driftheateqdecayprop} satisfying
\begin{equation}
    \sup_{s\in (-2,0]} \int_M \psi_s^2 d\nu_{\sol} \leq 1, \qquad \sup_{s\in [-1,1]} \int_M |\psi_s|^{2p_c} d\nu_{\sol} \leq C(g_{\sol}).
\end{equation}
Because the only pluriharmonic functions on $M$ are the constants, there exist $b\in \mathbb{R}$ and $k \in H$ such that 
\begin{equation*} \|k\|_H^2 + b^2 \leq 1,
\end{equation*}
and such that
$\psi_s^{\leq 0}:= k+e^sb$, $\psi^{>0}:=\psi-\psi^{\leq 0}$ satisfy \eqref{eq: decay of L^2-psi>0},
\begin{equation*}
    \varphi_{s}':=\varphi_{i,s}-\left( u_{h_i+\kappa_i k}-u_{h_i} + \kappa_i e^{s-s_i}b \right)
\end{equation*}
satisfy $\omega_s=\omega_{h_i+\kappa_i k}+\sqrt{-1}\partial \overline{\partial} \varphi_s'$, and \ref{prop:compactdecayconclusion1} holds. Moreover, from $\widehat{\varphi}_i \to \psi$ and Proposition \ref{prop:modelmetrics} \ref{modelmetrics4}, 
\begin{equation*}
    \widehat{\varphi}_{i,s}':=\kappa_i^{-1}\varphi_{i,s_i+s}' \to \psi_s^{>0} \quad \text{ in } \quad C_{\operatorname{loc}}^0(M\times [-1,1])
\end{equation*}
as $i\to \infty$, so that $\widetilde{\mathcal{D}}_{c,h_i+\kappa_i k}(\varphi_i',s_i+1) \leq e^{-\frac{a_c^2}{2}}\kappa_i^2$, yielding a contradiction.

It remains to show that \ref{prop:compactdecayconclusion4} holds. By Lemma \ref{lem:nonconcentrationcompact}, we have
\begin{equation} \label{prop:decaycompactpartc1}
    \sup_{s\in [s_0-1,s_0+1]} \int_M |\varphi_s|^{2p_c} d\nu_h \leq C\widetilde{\mathcal{D}}_{c,h}^{p_c}(\varphi,s_0),
\end{equation}
and by Lemma \ref{lem:pseudoapplication} \ref{lem:pseudoapplication1}, we have
\begin{equation} \label{prop:decaycompactpartc2}
    \sup_{s\in [s_0-1,s_0+1]} \|\omega_s-\omega_h\|_{C^j(M,g_h)} \leq \Psi(\epsilon_c|j).
\end{equation}
After possibly shrinking $\epsilon_c$, and recalling that $\sup_{h\in H}|f_h|\leq C(g_{\sol})$, we may therefore use \eqref{prop:decaycompactpartc1},\eqref{prop:decaycompactpartc2} and apply Lemma \ref{lem:L2toLinfty}, Lemma \ref{lem: schauder estimates}, and Lemma \ref{lem:interpolation} as in the proof of Proposition \ref{prop:decay} \ref{prop:decayconclusion4} to obtain
\begin{equation*}
    \sup_{s\in [s_0-1,s_0+1]} \|\omega_s-\omega_h\|_{C^2(M,g_h)} \leq \Psi(\delta_c).
\end{equation*}
Combining this with Proposition \ref{prop:modelmetrics} \ref{modelmetrics3} and $\|h-h'\|_H \leq C\delta_c^{\frac{1}{2}}$ yields the claim. 
\end{proof}

\section{Proofs of main theorems}

\label{section:mainproofs}

We now prove Theorem \ref{thm:main1} by iterating the decay estimate (Proposition \ref{prop:decay}) and combining with pseudolocality and the estimates for relative potentials of model metrics proved in Proposition \ref{prop:modelmetrics}.
In more detail, we first use the assumption that the given flow $(M,J,(\widetilde{g}_t)_{t\in [-T,0)},\widetilde{\nu})$ can be holomorphically approximated by AC K\"ahler-Ricci shrinkers to guarantee that the hypotheses of Proposition \ref{prop:decay} are satisfied for sufficiently large rescalings of the flow. Starting at a sufficiently large rescaling, we then iterate Proposition \ref{prop:decay} infinitely many times, obtaining a sequence $\varphi_j$ of K\"ahler potentials for the modified flow $\omega_s$ relative to a sequence $\omega_{h_j}$ of model metrics. Moreover, the quantities $\widetilde{\mathcal{D}}_{h_j}(\varphi_j,j)$ decay exponentially fast in $j$, ensuring that $h_j$ converge to some $h_{\infty}\in H$. Combined with parabolic regularity, we get that $\varphi_j$ decay exponentially on a region of size comparable to $\sqrt{j}$. 

Next, we use the explicit form of the modifications $\varphi_{j+1,s}-\varphi_{j,s}$ and the pointwise estimates of Proposition \ref{prop:modelmetrics} and Lemma \ref{lem:eigenfunctiongrowth} to find a limit K\"ahler potential $(\varphi_{\infty,s})_{s \in [0,\infty)}$ with respect to the model metric $\omega_{h_{\infty}}$ such that $\varphi_{j,j}-\varphi_{\infty,j}$ decays exponentially fast on a slightly larger region as $j\to \infty$. At this point, we have established the exponential decay of $\varphi_{\infty}$ on a region of the form $\{b_{h_{\infty}} \leq c\sqrt{s}\}$ as $s\to \infty$. Rewriting this in terms of the un-normalized K\"ahler potential $\widetilde{\varphi}_{\infty,t}$ and then applying pseudolocality once again allows us to finally propagate the bounds for $\widetilde{\varphi}_{\infty,t}$ to the entire region of interest. From here, the constructions of $F,\widehat{F}$ and the verifications of their properties are straightforward.

\begin{proof}[Proof of Theorem \ref{thm:main1}] After parabolic rescaling, we may assume $T\geq 1$. Set $Y:=1+|\mathcal{N}_{\widetilde{\nu}}(1)|$, and let $a=a(g_{\sol},Y)>0$ and  $\varepsilon_2=\epsilon_2(g_{\sol},Y)>0$ be as in the statement of Proposition \ref{prop:decay}. Let $\delta_2=\delta_2(\varepsilon_2,g_{\sol},Y)>0$ be as in Proposition \ref{prop:decay}, and choose $k\ge\max\{\epsilon_2^{-1}+2,3+\lfloor\frac{2}{(1+\theta_1)a^2}\log (2\delta_2^{-1})\rfloor\}$ such that $\widetilde{\nu}$ is $(\epsilon_2,e^{-\frac{s}{2}})$-self-similar for all $s\ge k-2$, and such that $\sum_{j=k}^{\infty} e^{-\frac{a^2}{4}(j-2)}<\frac{1}{2}\epsilon_2$. For $\delta_0 \in (0,\min\{\epsilon_2,\frac{1}{k+2}\})$ to be chosen depending on $\epsilon_2,a,\delta_2,k$, we apply Lemma \ref{lem:deldelbar} to obtain an open holomorphic embedding $F:\Omega(\delta_0^{-1})\to M$ and a solution $\varphi_0 \in C^{\infty}(\cup_{s\in [-\delta_0^{-1},\infty)}\Omega(\delta_0^{-1}e^{\frac{2}{5}s})\times \{s\})$ to
\begin{equation}
   \label{eq:PCMA3} \partial_s\varphi_{0,s}=\log\frac{(\omega_{\sol}+\sqrt{-1}\partial\bar\partial\varphi_{0,s})^n}{\omega_{\sol}^n}-\frac{1}{2}X\cdot\varphi_{0,s}+\varphi_{0,s}
\end{equation}
satisfying
$\omega_s = \omega_{\sol}+\sqrt{-1}\partial \overline{\partial}\varphi_{0,s}$, where $(\omega_s)_{s \in [-\delta_0^{-1},\infty)}$ is the modified flow defined as in \eqref{eq:modifyingtheflow}, and 
\begin{equation} \label{eq:inductionbase}
    \sup_{s\in [-\delta_0^{-1},\delta_0^{-1}]}\|\varphi_{0,s}\|_{C^{\lfloor \delta_0^{-1} \rfloor}(\Omega(e^{\frac{2}{5}s}\delta_0^{-1}) ,g_{\sol})}<\delta_0, \qquad \sup_{s\in [-\delta_0^{-1},\delta_0^{-1}]}\|f_{s} - (f_{\sol}+W)\|_{C^\lfloor \delta_0^{-1} \rfloor(\Omega(e^{\frac{2}{5}s}\delta_0^{-1}),g_{\sol})}<\delta_0,
\end{equation}
where $W$ is the entropy of the soliton $(M_{\sol},g_{\sol})$, and $f_s$ is defined as in \eqref{eq:modifyingtheflow}. In particular, we can ensure that $f_s(x_0)\leq f_{\sol}(x_0)+W+\frac{1}{4}$ for all $s\in [-\delta_0^{-1},\delta_0^{-1}]$. 

In the following, we set $h_0:=0$ and $\xi_{0,\alpha}:=0$ for $0\leq \alpha \leq N_-$. 
\begin{claim} \label{claim:inductivehype}
For each $j \in \mathbb{N}$, there exist $h_j \in H$ with $\|h_j-h_{j-1}\|_H \leq \sqrt{2}e^{-\frac{a^2}{4}(j-1+k)}$ for $j\ge 1$ and $\xi_{j,0},...,\xi_{j,N_-} \in \mathbb{R}$ with $\sum_{\alpha=0}^{N_-} \xi_{j,\alpha}^2 \leq 2e^{-\frac{a^2}{2}(j-1+k)}$, such that the recursively defined
\begin{equation*}
    \varphi_{j,s}:=\varphi_{j-1,s}+(u_{h_{j-1}}-u_{h_j})+\sum_{\alpha=0}^{N_-}e^{(1-\lambda_{\alpha})(s-(k+j-1))}\xi_{j,\alpha}\Psi_{\alpha}, \qquad j\geq 1
\end{equation*}
satisfy the following:
\begin{enumerate}
    \item \label{claim:inductive1} $\omega_s=\omega_{h_j}+\sqrt{-1}\partial\bar\partial\varphi_{j,s}$ on $\cup_{s\in [0,\infty)}(\Omega(e^{\frac{2}{5}s})\times \{s\})$,
    \item \label{claim:inductive2}  $\sup_{s\in [j+k-2,j+k]}\|\omega_s-\omega_{h_j}\|_{C^{\lfloor\varepsilon_2^{-1}\rfloor}(\Omega_{h_j}(a\sqrt{s}),g_{h_j})}<\varepsilon_2$,
    \item \label{claim:inductive3} $\widetilde{\mathcal{D}}_{h_j}(\varphi_j,j+k)\le \min\{\delta_2,e^{-\frac{a^2}{2}(j+k)}\} $,
    \item \label{claim:inductive4} $f_{j+k}(x_0)\leq f_{\sol}(x_0)+W+\frac{1}{2}$.
\end{enumerate}
\end{claim}
\begin{proof}[Proof of Claim \ref{claim:inductivehype}]
If $j=0$, then the claim follows from \eqref{eq:inductionbase} by taking $\delta_0 \leq \overline{\delta}_0(\epsilon_2)$. Now assume that for some $j_0 \in \mathbb{N}$, Claim \ref{claim:inductivehype} holds for all $j\in \{0,...,j_0\}$. Then $\|h_{j_0}\|_H \leq 2\sum_{j=0}^{\infty}e^{-\frac{a^2}{4}(j+k)}<\epsilon_2$ by our choice of $k$. Thus all hypotheses of Proposition \ref{prop:decay} are satisfied, so we may apply Proposition \ref{prop:decay} with $s_0=k+j_0$ to obtain $h_{j_0+1} \in H$ and $\xi_{j_0+1,0},...,\xi_{j_0+1,N_-} \in \mathbb{R}$ satisfying $\|h_{j_0+1}-h_{j_0}\|_H^2 \leq 2 e^{-\frac{a^2}{2}(j_0+k)}$ and $\sum_{\alpha=0}^{N_-}\xi_{j_0+1,\alpha}^2 \leq 2 e^{-\frac{a^2}{2}(j_0+k)}$, and such that $\varphi_{j_0+1,s}$ satisfy \ref{claim:inductive1}, \ref{claim:inductive2}, \ref{claim:inductive4} with $j=j_0+1$, as well as 
\begin{equation*}
    \widetilde{\mathcal{D}}_{h_{j_0+1}}(\varphi_{j_0+1},j_0+1+k)\leq e^{-\frac{a^2}{2}}\widetilde{\mathcal{D}}_{h_{j_0}}(\varphi_{j_0},j_0+k)\leq \min \{\delta_2,e^{-\frac{a^2}{2}(j_0+1+k)}\},
\end{equation*}
so that \ref{claim:inductive3} holds as well. The claim therefore follows by induction.
\end{proof}

Define $h_{\infty}:= \lim_{j\to \infty}h_j$ and
\begin{equation*}
    \mu_{m,\alpha}(s) := \sum_{j=1}^{m} e^{(1-\lambda_{\alpha})(s-j-k+1)}\xi_{j,\alpha}
\end{equation*}
for $m\in \mathbb{N} \cup \{\infty\}$, so that $\|h_j - h_{\infty}\|_{H} \leq \frac{1}{2}e^{-\frac{a^2}{4} j}$ and 
\begin{equation*}
    |\mu_{\infty,\alpha}(s)-\mu_{j,\alpha}(s)|\leq C\sum_{i=j+1}^{\infty} e^{-\frac{a^2}{4}i}e^{(1-\lambda_{\alpha})(s-i)}.
\end{equation*}
Next, we define
\begin{equation*}
    \varphi_{\infty,s} := \varphi_{0,s} - u_{h_{\infty}} +\sum_{\alpha=0}^{N_-}\mu_{\infty,\alpha}(s)\Psi_{\alpha} = \varphi_{j,s}+(u_{h_j}-u_{h_{\infty}})+\sum_{\alpha=0}^{N_-}(\mu_{\infty,\alpha}(s)-\mu_{j,\alpha}(s))\Psi_{\alpha}.
\end{equation*}
From Claim \ref{claim:inductivehype} \ref{claim:inductive3}, we obtain 
\begin{equation*}
      \sup_{s\in[j+k-2,j+k]} \int_{\Omega_{h_j}(\frac{a}{2}\sqrt{s})}\varphi_{j,s}^2 \omega_{h_j}^n\le  e^{-\frac{a^2}{2} j}e^{\frac{1}{4}a^2j}\le e^{-\frac{a^2j}{4}}.
\end{equation*}
By combining this with Claim \ref{claim:inductivehype} \ref{claim:inductive2}, and applying Lemma \ref{lem:L2toLinfty} with $r=(\epsilon_2^{-1}e^{-\frac{a^2 j}{8}})^{\frac{1}{n+2}}$, we obtain
\begin{equation} \label{eq:mainpf1}
    \sup_{s\in[j+k-2,j+k]} \sup_{\Omega_{h_j}(\frac{a}{3}\sqrt{s})}|\varphi_{j,s}|\le C(g_{\sol})e^{-\frac{a^2j}{8(n+2)}}.
\end{equation}
Then, by Lemma \ref{lem: schauder estimates} and Lemma \ref{lem:interpolation}, for all $m\le 102$ we obtain
\begin{equation*}
   \sup_{s\in[j+k-2,j+k]} \sup_{\Omega_{h_j}(\frac{7a}{24}\sqrt{s})}|(\nabla^{g_{h_j}})^{m}\varphi_{j,s}|_{g_{h_j}}\le C_m(g_{\sol})e^{-\alpha j},
\end{equation*}
for some uniform $\alpha=\alpha(a)>0$.
And for sufficiently large $j$, we have that for all $m\le 100$
\begin{equation*}
    \sup_{s\in[j+k-2,j+k]} \sup_{\Omega_{h_j}(\frac{a}{4}\sqrt{s})}|(\nabla^{g_{h_\infty}})^{m}(\omega_s-\omega_{h_\infty})|_{g_{h_\infty}}\le C_me^{-\alpha j}.
\end{equation*}
We now set $\widetilde{\varphi}_{j,t}:=|t|\phi_{\log(\frac{1}{|t|})}^{\ast} \varphi_{j,\log(\frac{1}{|t|})}$, $\tau_j:=-e^{-(j+k-1)}$ so that after pulling back we get 
\begin{equation} \label{eq:metricclosenessinshrinkerregion}
     \sup_{t\in[\tau_{j-1},\tau_{j+1}]}\sup_{\Omega_{h_{\infty},t}\left(\frac{a}{4}\sqrt{|t|\log \frac{1}{|t|}} \right)}|t|^{\frac{m}{2}-\alpha}|(\nabla^{g_{h_\infty,t}})^{m}(F^*\widetilde\omega_t-\omega_{h_\infty,t})|_{g_{h_\infty,t}}\le C_m,
\end{equation}
\begin{equation} \label{eq:pointwisephitilde}
    \sup_{t\in[\tau_{j-1},\tau_{j+1}]} \sup_{\Omega_{h_{\infty},t}\left(\frac{a}{4}\sqrt{|t|\log \frac{1}{|t|}} \right)}|t|^{\frac{m}{2}-1-\alpha}|(\nabla^{g_{h_\infty,t}})^{m}\widetilde \varphi_{j,t}|_{g_{h_\infty,t}}\le C_m,
\end{equation}
for all $m\le 100$. For $t\in [\tau_{j-1},\tau_{j+1}]$ and $x\in \Omega_{h_{\infty},t}\left(\frac{a}{4}\sqrt{|t|\log \frac{1}{|t|}} \right)$, we can use Lemma \ref{lem:eigenfunctiongrowth} \ref{lem:eigenfunctiongrowth0}, Proposition \ref{prop:modelmetrics} \ref{modelmetrics1}, and Claim \ref{claim:inductivehype} to estimate
\begin{equation}
\label{eq:closenessattimet}
\begin{aligned}
    |\widetilde{\varphi}_{j,t}-\widetilde{\varphi}_{\infty,t}|(x) &\leq |t|\cdot|u_{h_{\infty}}-u_{h_j}|(\phi_{\log(\frac{1}{|t|})}(x))+C|t|\sum_{\alpha=0}^{N_-} |\mu_{\infty,\alpha}(-\log|t|)-\mu_{j,\alpha}(-\log|t|)|\cdot |\Psi_{\alpha}|(\phi_{\log(\frac{1}{|t|})}(x))\\ &\leq C\|h_{\infty}-h_j\|_H (|t|+b_{\sol,t}^2(x) )+ C|t|\sum_{\alpha=0}^{N_-} |\mu_{\infty,\alpha}(-\log|t|)-\mu_{j,\alpha}(-\log|t|)|\cdot (1+ f_{\sol,t}(x))^{\lambda_{\alpha}}\\
    &\leq C(|t|+b_{\sol,t}^2(x))e^{-\frac{a^2}{4}j}+C\sum_{\alpha=0}^{N_-} e^{-(1-\lambda_{\alpha}+\frac{a^2}{4})j}(|t|+b_{\sol,t}^2(x))^{\lambda_{\alpha}}\\
    &\le C|t|^{\frac{a^2}{4}}(|t|+b_{\sol,t}^2(x)).
\end{aligned}
\end{equation}
Since $\operatorname{R}_{\omega_{\sol}}$ is bounded, there exists $C(g_{\sol}) \in (1,\infty)$ such that
\begin{equation*} |b_{\sol,t}^2(x)-b_{\sol,s}^2(x)|=
2| |t|f_{\sol,t}(x)-|s|f_{\sol,s}(x)|\leq C|t-s|
\end{equation*}
for all $x\in M_{\sol}$ and $0 > s\ge t.$ We can, therefore, let $t\to 0^{-}$ to get a limit $b_{\sol,0}.$ In particular, we have
\begin{equation}\label{eq; time drift of potentials}
    |b^2_{\sol,t}-b^2_{\sol,0}|\leq C(g_{\sol})|t|,
\end{equation}
Note also that the same holds for $b_{h_{\infty},t}$. Then there exists a $\widetilde{T}_0 \in [-1,0)$ such that for all $t\in[\widetilde T_0,0]$,
\begin{equation*}
    \Omega_{h_{\infty},0}\left( \frac{a}{5}\sqrt{|t|\log \frac{1}{|t|}} \right) \subseteq \Omega_{h_{\infty},t}\left( \frac{a}{4}\sqrt{|t|\log \frac{1}{|t|}} \right).
\end{equation*}
Therefore, for all $m\le 100$,
\begin{equation} \label{eq:onehundo}
    \sup_{t\in [\widetilde{T}_0,0)}\sup_{\Omega_{h_{\infty},0}\left( \frac{a}{5}\sqrt{|t|\log \frac{1}{|t|}} \right)}|t|^{\frac{m}{2}-\alpha}|(\nabla^{g_{h_\infty,t}})^{m}(F^*\widetilde\omega_t-\omega_{h_\infty,t})|_{g_{h_\infty,t}}\le C_m,
\end{equation}
   \begin{figure}
\centering
\begin{minipage}{0.6\textwidth}
    \centering
    \resizebox{\linewidth}{!}{\tikzset{every picture/.style={line width=0.75pt}} 

\begin{tikzpicture}[x=0.75pt,y=0.75pt,yscale=-1,xscale=1]

\draw  [color={rgb, 255:red, 255; green, 255; blue, 255 }  ,draw opacity=1 ][fill={rgb, 255:red, 184; green, 233; blue, 134 }  ,fill opacity=0.46 ] (541.7,382.2) -- (145.8,117) -- (541.7,117) -- cycle ;
\draw  (588.1,382.51) -- (125.1,382.51)(541.8,58.51) -- (541.8,418.51) (132.1,377.51) -- (125.1,382.51) -- (132.1,387.51) (546.8,65.51) -- (541.8,58.51) -- (536.8,65.51)  ;
\draw    (145.8,117) -- (541.7,382.2) ;
\draw  [color={rgb, 255:red, 0; green, 0; blue, 0 }  ,draw opacity=1 ][fill={rgb, 255:red, 208; green, 2; blue, 27 }  ,fill opacity=0.34 ][line width=0.75]  (145.8,117) .. controls (145.8,117) and (145.8,117) .. (145.8,117) .. controls (364.5,117) and (541.8,235.87) .. (541.8,382.51) -- (145.8,382.51) -- cycle ;
\draw  [dash pattern={on 4.5pt off 4.5pt}]  (447,209) -- (542,209) ;
\draw  [dash pattern={on 4.5pt off 4.5pt}]  (447,209) -- (447,383) ;
\draw  [fill={rgb, 255:red, 0; green, 0; blue, 0 }  ,fill opacity=0.75 ] (545.38,209) .. controls (545.38,207.14) and (543.86,205.63) .. (542,205.63) .. controls (540.14,205.63) and (538.63,207.14) .. (538.63,209) .. controls (538.63,210.86) and (540.14,212.38) .. (542,212.38) .. controls (543.86,212.38) and (545.38,210.86) .. (545.38,209) -- cycle ;
\draw  [fill={rgb, 255:red, 0; green, 0; blue, 0 }  ,fill opacity=0.75 ] (450.38,383) .. controls (450.38,381.14) and (448.86,379.63) .. (447,379.63) .. controls (445.14,379.63) and (443.63,381.14) .. (443.63,383) .. controls (443.63,384.86) and (445.14,386.38) .. (447,386.38) .. controls (448.86,386.38) and (450.38,384.86) .. (450.38,383) -- cycle ;
\draw  [fill={rgb, 255:red, 0; green, 0; blue, 0 }  ,fill opacity=0.75 ] (450.38,209) .. controls (450.38,207.14) and (448.86,205.63) .. (447,205.63) .. controls (445.14,205.63) and (443.63,207.14) .. (443.63,209) .. controls (443.63,210.86) and (445.14,212.38) .. (447,212.38) .. controls (448.86,212.38) and (450.38,210.86) .. (450.38,209) -- cycle ;
\draw [color={rgb, 255:red, 0; green, 0; blue, 0 }  ,draw opacity=1 ] [dash pattern={on 4.5pt off 4.5pt}]  (145.8,117) .. controls (199,122) and (547,73) .. (541.8,382.51) ;
\draw [line width=1.5]    (447,181) -- (447,236) ;
\draw  [line width=0.75]  (448,208) .. controls (451.71,207.86) and (453.49,205.94) .. (453.35,202.24) -- (453.35,202.24) .. controls (453.16,196.94) and (454.91,194.22) .. (458.62,194.09) .. controls (454.91,194.22) and (452.96,191.64) .. (452.77,186.35)(452.85,188.74) -- (452.77,186.35) .. controls (452.63,182.64) and (450.71,180.86) .. (447,181) ;

\draw (550,71.4) node [anchor=north west][inner sep=0.75pt]    {$b_{h_{\infty } ,0}^{2}$};
\draw (543.8,385.91) node [anchor=north west][inner sep=0.75pt]    {$( 0,0)$};
\draw (182,392.4) node [anchor=north west][inner sep=0.75pt]    {$t$};
\draw (247,250.4) node [anchor=north west][inner sep=0.75pt]  [font=\scriptsize]  {$b_{h_{\infty } ,0}^{2} =\frac{|t|}{(\lambda \epsilon_{PS})^2}$};
\draw (243,152.4) node [anchor=north west][inner sep=0.75pt]  [font=\scriptsize]  {$b_{h_{\infty } ,0}^{2} =\frac{a^{2}}{10^{2}} |t|\log\frac{1}{|t|}$};
\draw (550,198.4) node [anchor=north west][inner sep=0.75pt]    {$( 0,x)$};
\draw (428,390.4) node [anchor=north west][inner sep=0.75pt]    {$( t_{0}(x) ,0)$};
\draw (387,202.4) node [anchor=north west][inner sep=0.75pt]    {$( t_{0}(x) ,x)$};
\draw (242,84.4) node [anchor=north west][inner sep=0.75pt]  [font=\scriptsize]  {$b_{h_{\infty } ,0}^{2} =\frac{a^{2}}{5^{2}} |t|\log\frac{1}{|t|}$};
\draw (462,187.4) node [anchor=north west][inner sep=0.75pt]  [font=\scriptsize]  {$\lambda b_{h_{\infty } ,0}( x)$};

\end{tikzpicture}}
\end{minipage}    
      \caption{This figure illustrates the relevant space--time region. Below the dashed curve $b_{h_\infty,0}^2=\frac{a^2}{5^2}|t|\log\frac{1}{|t|}$, we have precise estimates comparing $F^*\widetilde{\omega}_t$ and $\omega_{h_\infty,t}$. The curve $b_{h_\infty,0}^2=\frac{|t|}{(\lambda \varepsilon_{PS})^2}$ is the threshold arising from pseudolocality, below which we cannot propagate all the way to time $t=0$. 
      In the proof, we are given $x$, and choose $t_0(x)$ so that $(x,t_0(x))$ lies on the second curve. We apply pseudolocality at scale $\approx b_{h_{\infty},0}(x)$ to propagate the smallness of $F^{\ast}\widetilde{\omega}_{t_0}-\omega_{h_{\infty},t_0}$ near $x$ forwards in time up to $t=0$. 
}
\label{normalized space time}
\end{figure}
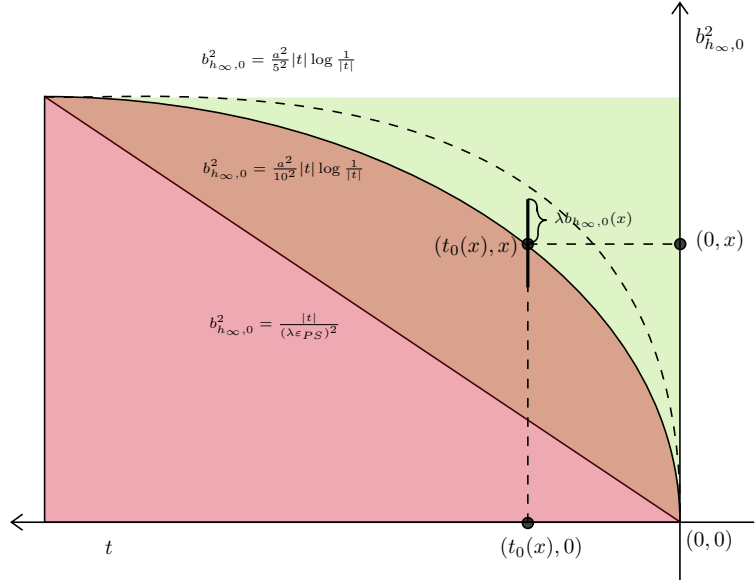

By shrinking $|\widetilde{T}_0|$ if necessary, we may assume that $\frac{a}{10}\sqrt{|t|\log \frac{1}{|t|}}$ is monotone on $[\widetilde{T}_0,0).$ Fix $\rho_0 \in (0,e^{-1})$ sufficiently small and let $T_0\in [\widetilde{T}_0,0)$ be the unique time such that 
\begin{equation*}
    \rho_0=\frac{a}{10}\sqrt{|T_0|\log \frac{1}{|T_0|}}.
\end{equation*}
For every $x \in M_{\sol}$ with $b_{h_{\infty},0}(x)\in (0,\rho_0],$ we then define $t_0(x)\in [T_0,0)$ by
\begin{equation*}
    b_{h_{\infty},0}(x)=\frac{a}{10}\sqrt{|t_0(x)|\log \frac{1}{|t_0(x)|}}.
\end{equation*}
We first estimate $|\widetilde{\varphi}_{\infty,t}|$ when $t\in [T_0,t_0(x)]$. 
Choose $j =\lfloor \log \frac{1}{|t|}\rfloor-k+2$. Then \eqref{eq:closenessattimet} and \eqref{eq; time drift of potentials} together imply
\begin{align*}
    |\widetilde{\varphi}_{\infty,t}- \widetilde{\varphi}_{j,t}|(x) 
    & \le C|t|^{1+\frac{a^2}{4}}+C b_{h_{\infty},0}^2(x)|t|^\frac{a^2}{4}+ C\sum_{\alpha=0}^{N_{-}}|t|^{(1-\lambda_{\alpha}+\frac{a^2}{4})}b_{h_{\infty},0}^{2\lambda_{\alpha}}(x)\\
    &\leq C(b_{h_{\infty},0}^2(x)+|t|)|t|^{\frac{a^2}{4}}.
\end{align*}
Since $x \in \Omega_{h_{\infty},0}(\frac{a}{10}\sqrt{|t|\log \frac{1}{|t|}})$, we may combine this with \eqref{eq:pointwisephitilde} to get

\begin{equation}
\label{eq; estimate on limit potential -positive time}
     |\widetilde{\varphi}_{\infty,t}(x)|\le C(b_{h_\infty,0}^{2}(x)+|t|)|t|^{\gamma}
\end{equation}
for all $t\in [T_0,t_0(x)]$, where $\gamma:=\frac{1}{2}\min\{\alpha,\frac{a^2}{4}\}$. Next, we estimate $|\widetilde{\varphi}_{\infty,t}|$ for $t\in [t_0(x),0)$. 
Combining \eqref{eq:onehundo} and Lemma \ref{lem:basicACfacts} \ref{lem:basicACfacts3}, \ref{lem:basicACfacts4}, if $b_{h_\infty,0}(x)$ is sufficiently small, we obtain $\lambda=\lambda(g_{\sol})>0$ such that
\begin{equation*}
    \sup_{B_{\widetilde{g}_{t_0}}(F(x),\lambda b_{h_{\infty},0}(x))} |\Rm(\widetilde{g}_{t_0})|_{\widetilde{g}_{t_0}} \leq \frac{1}{(\lambda b_{h_{\infty},0}(x))^2}, \qquad \frac{\operatorname{Vol}_{\widetilde{g}_{t_0}}\left(B_{\widetilde{g}_{t_0}}(F(x),\lambda b_{h_{\infty},0}(x))\right)}{(\lambda b_{h_{\infty},0}(x))^{2n}}\geq \lambda.
\end{equation*}
We may therefore apply pseudolocality \cite[Theorem 1.2]{LuPseudolocal}
to conclude that 
\begin{equation}\label{eq:curvatureboundspseudolocality}
    \sup_{\tau \in [t_0,\min\{t_0+(\lambda \epsilon_{PS}b_{h_{\infty},0}(x))^{2},0\})} |\Rm(F^{\ast}\widetilde{g}_{\tau})|_{F^{\ast}\widetilde{g}_\tau}(x)\leq \frac{C}{(\epsilon_{PS}\lambda b_{h_{\infty},0}(x))^2},
\end{equation}
where $\epsilon_{PS}=\epsilon_{PS}(g_{\sol})>0$. By our choice of $t_0$, we have $t_0+(\lambda \epsilon_{PS}b_{h_{\infty},0}(x))^{2}\geq 0$, hence 
\begin{equation} \label{eq:curvatureboundsuptotime0}
\sup_{\tau \in [t_0(x),0)} |\Rm_{F^{\ast}\widetilde{g}_\tau}|_{F^{\ast}\widetilde{g}_{\tau}}(x)\leq \frac{C(g_{\sol})}{b_{h_{\infty},0}^2(x)}.
\end{equation}
By combining Lemma \ref{lem:basicACfacts} \ref{lem:basicACfacts3}, \eqref{eq:curvatureboundsuptotime0}, integrating in time, and using \eqref{eq:metricclosenessinshrinkerregion}, we get
\begin{equation*} \exp\left(-C(g_{\sol})\frac{|t_0(x)|}{b_{h_{\infty},0}^2(x)}\right)\omega_{h_{\infty},t_0(x)}\leq 
\omega_{h_{\infty},t} \leq \exp\left(C(g_{\sol})\frac{|t_0(x)|}{b_{h_{\infty},0}^2(x)}\right)\omega_{h_{\infty},t_0(x)},
\end{equation*}
\begin{equation*} \exp\left(-C(g_{\sol})\frac{|t_0(x)|}{b_{h_{\infty},0}^2(x)}\right)F^{\ast}\widetilde{\omega}_{t_0(x)}\leq F^{\ast}
\widetilde{\omega}_t \leq \exp\left(C(g_{\sol})\frac{|t_0(x)|}{b_{h_{\infty},0}^2(x)}\right)F^{\ast}\widetilde{\omega}_{t_0(x)}
\end{equation*}
for all $t\in [t_0(x),0)$. It follows that
\begin{equation}
\label{eq:closenessofmetricsinconeregion}
|F^{\ast}\widetilde{g}_t-g_{h_{\infty},t}|_{g_{h_\infty},t}(x)\leq C(g_{\sol})\frac{|t_0(x)|}{b_{h_{\infty},0}^2(x)} \leq \frac{C(g_{\sol})}{\log\frac{1}{|t_0(x)|}}\leq \frac{C(g_{\sol})}{\log\frac{1}{b_{h_{\infty},0}(x)}},
\end{equation}
\begin{equation}
\label{eq:closenessofvolumeforms}
|\partial_t \widetilde{\varphi}_{\infty,t}|(x)=\left| \log \left( \frac{F^{\ast}\widetilde{\omega}_t^n}{\omega_{h_{\infty},t}^n}\right) \right|(x) \leq \frac{C(g_{\sol})}{\log\frac{1}{|t_0(x)|}}
\end{equation}
for all $t\in [t_0(x),0)$. In particular, $\widetilde{\varphi}_{\infty}$ extends continuously to time $0$. Integrating \eqref{eq:closenessofvolumeforms} in time and combining with \eqref{eq:pointwisephitilde} and \eqref{eq:closenessattimet} yields
\begin{align} \label{eq:pointwisebutlog}
    \sup_{t\in [t_0(x),0)} |\widetilde{\varphi}_{\infty,t}(x)| \leq C_0|t_0(x)|^{1+\gamma}+\frac{C(g_{\sol})|t_0(x)|}{\log\frac{1}{|t_0(x)|}} \leq \frac{C(g_{\sol})b_{h_{\infty},0}^2(x)}{\log\frac{1}{b_{h_{\infty},0}(x)}}.
\end{align} 
Combining \eqref{eq:pointwisebutlog} and \eqref{eq; estimate on limit potential -positive time} yields 
\begin{equation} \label{eq:C0estimatefinal} |\widetilde{\varphi}_{\infty,t}|(x) \leq \frac{C_0 (b_{h_{\infty},0}(x)+\sqrt{|t|})^2}{\log(\frac{1}{b_{h_{\infty},0}(x)+\sqrt{|t|}})} 
\end{equation}
for all $(x,t)\in \Omega_{h_{\infty},0}(\rho_0)\times[T_0,0]$, where $T_0 \in [\widetilde{T}_0,0)$ is defined by $\frac{a^2}{10^2}|T_0|\log \frac{1}{|T_0|}=\rho_0^2$. 
Setting $\rho(x,t):=b_{h_{\infty},0}(x)+\sqrt{|t|}$, we combine the $m=0$ case of \eqref{eq:onehundo} with \eqref{eq:closenessofmetricsinconeregion} to get $c>0$ such that
\begin{equation} \label{eq:C0intermsofrho}
   \sup_{y\in B_{g_{h_{\infty},t}}(x,c\rho(x,t))} \sup_{s\in [t-(c\rho(x,t))^2,t]} |F^{\ast}\widetilde{g}_s-g_{h_{\infty},s}|_{g_{h_{\infty},s}}(y) \leq \frac{C}{\log\frac{1}{\rho(x,t)}}
\end{equation}
for every $(x,t)\in M_{\sol}\times [\frac{1}{2}T_0,0)$ with $\rho(x,t) <\rho_0$. The metrics $F^{\ast}\widetilde{g}_t$ and $g_{h_{\infty},t}$ are K\"ahler with respect to the same fixed complex structure, which implies that the identity map between them is holomorphic and, in particular, harmonic. 
By \eqref{eq:C0intermsofrho} and Lemma \ref{lem:basicACfacts} \ref{lem:basicACfacts3}, we may therefore apply \cite[Lemma A.2]{BamKlein} with $r=c\rho(x,t)$ and $H=C(g_{\sol})(\log\frac{1}{\rho(x,t)})^{-1}$ to obtain
\begin{equation} \label{eq:pushingclosenessforwards}
    |(\nabla^{g_{h_\infty,t}})^k(F^*\widetilde
    g_{t}-g_{h_\infty,t})|_{g_{h_{\infty},t}}(x) \le \frac{C_k}{(b_{h_{\infty},0}(x)+\sqrt{|t|})^{k}\log(\frac{1}{b_{h_{\infty},0}(x)+\sqrt{|t|}})}
\end{equation}
for all $k\in \mathbb{N}$ and $(x,t)\in \Omega_{h_{\infty},0}(\rho_0)\times [\frac{1}{2}T_0,0)$ such that $\rho(x,t)\leq \overline{\rho}(k)$. On the complementary compact region $\{ \rho \geq \overline{\rho}(k)\}$, the estimate also holds (after possibly increasing $C_k$) since the flow is smooth in any such region. By \eqref{eq:C0estimatefinal},\eqref{eq:pushingclosenessforwards}, and Lemma \ref{lem: schauder estimates}, we finally obtain
\begin{equation} \label{eq:higherderivativesoftildevarphi}
    |(\nabla^{g_{h_\infty,t}})^k\widetilde{\varphi}_{\infty,t}|_{g_{h_\infty,t}}(x)\le \frac{C_k(b_{h_\infty,0}(x)+\sqrt{|t|})^{2-k}}{\log \left(\frac{1}{b_{h_\infty,0}(x)+\sqrt{|t|}}\right)},
\end{equation}
for all $k\in \mathbb{N}$ and  $(x,t)\in \Omega_{h_{\infty},0}(\rho_0) \times [T_0,0)$.
In particular, assertions \ref{mainthm:2},\ref{mainthm:4} hold if we replace $F$ by $F\circ (\zeta_1^{h_\infty})^{-1}$, where we set $\mathcal{U}:= \{r < \frac{\rho_0}{2}\}\subseteq \mathcal{C}$, so that $\pi_{\sol}^{\ast}r=b_{\sol,0}$ and $\pi_{\sol}^{-1}(\mathcal{U})=\{x\in M_{\sol} \: | \: b_{\sol,0}(x)<\frac{\rho_0}{2}\}$. 

It follows from \eqref{eq:higherderivativesoftildevarphi} that $\lim_{t\nearrow 0}\operatorname{Vol}_{\widetilde{g}_t}(M)>0$, so that $\alpha:= [\widetilde{\omega}_{-T}]-Tc_1(M)$ is big and nef. Set $\Sigma := E_{nK}(\alpha)$, so that $M\setminus \Sigma$ is exactly the (open) set where $\widetilde{\omega}_t$ converges smoothly to a K\"ahler metric $\widetilde{\omega}_0$ as $t\nearrow 0$ by \cite[Theorem 1.5]{collinstosatti}. Setting $Y^{\circ}:=Y\setminus \pi(\Sigma)$, the properness of $\pi$ implies that $\pi(\Sigma)$ is a complex subvariety of $Y$ of positive codimension (since $\dim(Y)=\dim(M)$ is big, hence $\pi$ is generically finite), and in particular $Y^{\circ}$ is open. For each $x\in M$, $\pi^{-1}(\pi(x))$ is a complex analytic subvariety of $M$. If $x\in M$ satisfies $\dim_{\mathbb{C}}\pi^{-1}(\pi(x))\geq 1$, then there is an irreducible variety $Z\subseteq \pi^{-1}(\pi(x))$ of dimension $k\geq 1$ containing $x$, so that $\int_Z \alpha^k=0$, hence $x\in \Sigma$. It follows that for all $x\in M\setminus \Sigma$, we have $\dim_{\mathbb{C}}\pi^{-1}(\pi(x))=0$, hence (because $M$ is compact and $\pi$ has connected fibers) $\pi^{-1}(\pi(x))=\{x\}$. In particular (recalling the fact that $\pi$ is surjective), $\pi$ restricts to a proper holomorphic bijection $\pi:M\setminus \Sigma \to Y^{\circ}$ between open subsets of $M,Y$, respectively, hence this restriction is a homeomorphism. Because any holomorphic homeomorphism between normal analytic spaces is an isomorphism \cite[p. 166]{GrauRemm}, it follows that $\pi:M\setminus \Sigma \to Y^{\circ}$ is in fact a biholomorphism. 

We now construct $\widehat{F}$. By the proof of Lemma \ref{lem:improvementofcomparability}, \begin{equation*}
\sup_{y\in \pi_{\sol}^{-1}(o)}\widetilde{f}_t(F(y))\leq \sup_{\pi_{\sol}^{-1}(o)}f_{\sol}+W+1
\end{equation*}
for $|t|$ sufficiently small. Thus $(F(y),t)$ is a $C(g_{\sol})$-center of $\widetilde{\nu}$ for each $y\in \pi_{\sol}^{-1}(o)$, so that for any $H_{2n}$-center $(z_t,t)$ of $\widetilde{\nu}$, we have $d_{\widetilde{g}_t}(z_t,F(y))\leq C(g_{\sol})\sqrt{|t|}$. Because $\widetilde{\omega}_t \geq C^{-1}\pi^{\ast}\omega_Y$, this implies 
\begin{equation*}
    \widehat{\pi}(\widetilde{\nu})=\lim_{t\nearrow 0} \pi(z_t)= \pi(F(y)),
\end{equation*}
hence $F(\pi_{\sol}^{-1}(o))\subseteq \pi^{-1}(\widehat{\pi}(\widetilde{\nu})) \subseteq \Sigma$.
Because $|\Rm(\widetilde{g}_t)|_{\widetilde{g}_t}(F(x)) \leq C b_{h_{\infty},0}^{-2}(x)$ for all $t\in  [-T,0)$, we have $F(\pi_{\sol}^{-1}(\mathcal{U}\setminus \{o\})) \subseteq M\setminus \Sigma$. 
Setting $y_0:= \widehat{\pi}(\widetilde{\nu})$, it follows that $F(\pi_{\sol}^{-1}(o))=F(\pi_{\sol}^{-1}(\mathcal{U}))\cap \pi^{-1}(y_0)$, hence $F(\pi_{\sol}^{-1}(o))$ is relatively open in $\pi^{-1}(y_0)$. Because $F(\pi_{\sol}^{-1}(o))$ is also compact, and $\pi_{\sol}^{-1}(o)$ is connected since $\pi_{\sol}$ is the Remmert reduction map, it is therefore a connected component of $\pi^{-1}(y_0)$. Because $\pi$ has connected fibers, it follows that $\pi^{-1}(y_0)=F(\pi_{\sol}^{-1}(o))$. Moreover, $\pi \circ F \circ (\pi_{\sol}|_{\pi_{\sol}^{-1}(\mathcal{U}\setminus \{o\})})^{-1}$ restricts to a biholomorphism  $\widehat{F}:\mathcal{U}\setminus \{o\} \to \mathcal{V}\setminus \{y_0\}$, where $\mathcal{V}:= (\pi \circ F)(\pi_{\sol}^{-1}(\mathcal{U}))$ is open because $F(\pi_{\sol}^{-1}(\mathcal{U}))$ is a $\pi$-saturated open subset of $M$. 

Because $\lim_{x\to \pi_{\sol}^{-1}(o)}\widetilde{\varphi}_0(x) = 0$, it follows that $\widetilde{\varphi}_0 = \pi_{\sol}^{\ast}\widehat{\varphi}$ for some $\widehat{\varphi} \in C^{\infty}(\mathcal{U}\setminus \{o\})\cap C^0(\mathcal{U})$. Set $\widehat{\omega}_0 := \pi_{\ast}\widetilde{\omega}_0$. After possibly shrinking $\rho_0>0$, \eqref{eq:higherderivativesoftildevarphi} gives $\widehat{F}^{\ast}\widehat{\omega}_0=((\pi_{\sol}|_{\pi_{\sol}^{-1}(\mathcal{U}\setminus \{o\})})^{-1})^{\ast}F^{\ast}\widetilde{\omega}_0 = \omega_{\mathcal{C}}+\sqrt{-1}\partial \overline{\partial}\widehat{\varphi}$, where 
\begin{equation*} -\frac{1}{2}\omega_{\mathcal{C}}\leq \sqrt{-1}\partial \overline{\partial}\widehat{\varphi} \leq \frac{1}{2}\omega_{\mathcal{C}},\end{equation*} 
and where Theorem \ref{thm:main1} \ref{mainthm:3}, \ref{mainthm:5} hold. In particular, after further shrinking $\rho_0$, $\widehat{F}:(\mathcal{U}\setminus \{o\},d_{g_{\mathcal{C}}})\to (\mathcal{V}\setminus \{y_0\},d_{\widehat{g}_0})$ extends to a bi-Lipschitz map $\widehat{F}:\mathcal{U} \to \mathcal{V}$, which is holomorphic by the Riemann extension theorem for normal spaces (applied to local coordinate functions). Because $\widehat{F}:\mathcal{U} \to \mathcal{V}$ is a holomorphic homeomorphism between normal complex analytic spaces, it is a biholomorphism \cite[p. 166]{GrauRemm}, which satisfies Theorem \ref{thm:main1} \ref{mainthm:1} by construction. 
\end{proof}

\begin{proof}[Proof of Corollary \ref{cor:2dim}] 
Suppose $\widetilde{\nu} =(\widetilde{\nu}_t)_{t\in [-T,0)} \in \widehat{\pi}^{-1}(\pi(E_i))$ for some $i\in \{1,...,N\}$, and let $(M_{\sol},J_{\sol},g_{\sol},f_{\sol})$ be the K\"ahler-Ricci shrinker corresponding to some tangent flow of $(M,(\widetilde{g}_t)_{t\in [-T,0)},\widetilde{\nu})$. By \cite[Theorem A]{CHM}, $M_{\sol}$ is smooth. If $M_{\sol}$ is the flat Gaussian, then by \cite[Theorem 2.37]{Bam3}, we have $\widetilde{\nu}=(\nu_{x,0;t})_{t\in [-T,0)}$ for some $x\in M\setminus \cup_{i=1}^N E_i$. Then $\widehat{\pi}(\widetilde{\nu})=\pi(x)\notin \pi(E_i)$, a contradiction. Thus $M_{\sol}$ is non-flat. By \cite[Proof of Theorem B]{CifarelliConlonDeruelleKsurface}, it then follows that $(M_{\sol},J_{\sol},g_{\sol},f_{\sol})$ is the AC K\"ahler-Ricci shrinker $(M_{\operatorname{FIK}},J_{\operatorname{FIK}},g_{\operatorname{FIK}},f_{\operatorname{FIK}})$ on the total space of $\mathcal{O}_{\mathbb{P}^1}(-1)$ constructed in \cite{FIK}. By Corollary \ref{cor:complexrigid} and Example \ref{example:FIK}, $(M,J,(\widetilde{g}_t)_{t\in [-T,0)},\widetilde{\nu})$ can be holomorphically approximated by $(M_{\operatorname{FIK}},J_{\operatorname{FIK}},g_{\operatorname{FIK}},f_{\operatorname{FIK}})$. Then Corollary \ref{cor:2dim} \ref{cor:2dim2}, \ref{cor:2dim3} follow from Theorem \ref{thm:main1}. Because $\operatorname{id}_{\mathbb{C}^2}:\mathbb{C}^2\to \mathbb{C}^2$ is the canonical model of $\mathbb{C}^2$, which is smooth, the hypotheses of \cite[Theorem A]{CHL} are satisfied. Then Corollary \ref{cor:2dim} \ref{cor:2dim1} follows from \cite[Theorem A(iii)]{CHL} and  Corollary \ref{cor:2dim} \ref{cor:2dim4} follows from \cite[Proposition 5.4]{CHL}.
\end{proof}

\begin{proof}[Proof of Corollary \ref{cor:calabi}] 
    By \cite[Proof of Theorem 10.11]{jiansongtian}, for any $\widetilde{\nu} \in \widehat{\pi}^{-1}(y_0)$, $(M,(\widetilde{g}_t)_{t\in [-T,0)},\widetilde{\nu})$ can be holomorphically approximated by $(M_{\sol},J_{\sol},g_{\sol},f_{\sol})$. We may thus apply Theorem \ref{thm:main1} to obtain Corollary \ref{cor:calabi} \ref{cor:calabi1}, \ref{cor:calabi2}. Because the canonical model of $\mathcal{O}_{\mathbb{P}^n}(-1)^{\oplus(m+1)} \to \mathcal{C}$ is $\mathcal{O}_{\mathbb{P}^m}(-1)^{\oplus(n+1)}$, which is smooth, \cite[Theorem A]{CHL} and \cite[Proposition 5.4]{CHL} yield \ref{cor:calabi} \ref{cor:calabi3} and the stated type I curvature behavior emerging from $t=0.$
\end{proof}

\begin{proof}[Proof of Corollary \ref{cor:orb}]
    The proof is the same as Corollary \ref{cor:calabi}, except that for the flow out of singularities, we need a minor generalization of \cite{CHL} to the case where the relative canonical model of $\pi_{\sol}:M_{\sol}\to \mathcal{C}$ is $\mathbb{C}^n/\mathbb{Z}_k$. Let $(M_{\exp},J_{\exp},g_{\exp},f_{\exp})$ denote the orbifold expander on $\mathbb{C}^n/\mathbb{Z}_k$ in the statement of Corollary \ref{cor:orb}. This is a global $\mathbb{Z}_k$-quotient of a smooth AC K\"ahler-Ricci expander $(\widetilde{M},\widetilde{J}_{\exp},\widetilde{g}_{\exp},\widetilde{f}_{\exp})$ on $\widetilde{M}=\mathbb{C}^n$, whose $\mathbb{Z}_k$-quotient is $M_{\exp}$. It follows that $M_{\exp}$ satisfies all conclusions of \cite[Section 2]{CHL}. The gluing construction of \cite[Section 3]{CHL} goes through verbatim, since the gluing only occurs in the region where $M_{\exp}$ is smooth. The estimates of \cite[Section 4]{CHL} and \cite[Subsections 5.1-5.3]{CHL} use only the maximum principle (which applies verbatim for smooth orbifolds) and pseudolocality, which holds for compact orbifold K\"ahler-Ricci flows by \cite{wangorbifold}. It remains only to note that because $M'$ is a smooth orbifold with isolated singularities, the corresponding complex space has log terminal singularities \cite[Corollary 5.21]{KollarMori}, hence one may appeal to \cite[Theorem 4.6.8]{BoucksomEyssydieuxGuedj} to conclude that the flow constructed as a limit of the glued metrics in \cite[Theorem A]{CHL} coincides with the flow out of singularities constructed in \cite{SW2}, since the two flows have the same initial current and have continuous local potentials.
\end{proof}

\begin{proof}[Proof of Theorem \ref{thm:main3}] By \cite[Corollary 4.3]{DerSz}, there are $t_i \nearrow 0$ and biholomorphisms $\psi_i:M\to M$ such that $\widetilde{g}_{i,t}:=|t_i|^{-1}\psi_i^{\ast}\widetilde{g}_{|t_i|t}\to g_{\sol,t}$ in $C_{\operatorname{loc}}^{\infty}(M\times (-\infty,0))$. Writing $\widetilde{\nu}_t^i := \psi_i^{\ast}\widetilde{\nu}_{|t_i|t}=(2\pi |t|)^{-n}e^{-\widetilde{f}_{i,t}}d\widetilde{g}_{i,t}$, parabolic regularity and the heat kernel estimates of \cite{bamscalar} give $\widetilde{f}_{i,t}\to f_{\infty,t}$ for some $f_{\infty}\in C^{\infty}(M\times (-\infty,0))$. Perelman's monotonicity formula implies
\begin{equation*}
    \limsup_{i\to \infty} \int_{-\alpha^{-1}}^{-\alpha} \int_M |\Ric_{\widetilde{g}_{i,t}}+\nabla^{\widetilde{g}_{i,t}}\nabla^{\widetilde{g}_{i,t}}\widetilde{f}_{i,t}-\frac{1}{|t|}\widetilde{g}_{i,t}|^2 d\widetilde{\nu}_{t}^idt =0
\end{equation*}
for each $\alpha \in (0,1)$, which yields (upon passing to the limit) 
\begin{equation*}
    \Ric(g_{\sol,t})+\nabla^{g_{\sol,t}} \nabla^{g_{\sol,t}}f_{\infty,t}=\frac{1}{|t|}g_{\sol,t}
\end{equation*}
for all $t<0$. In particular, $\nabla^{g_{\sol}}\nabla^{g_{\sol}}(f_{\infty,t}-f_{\sol,t})=0$, so that because $M_{\sol}$ does not split $\mathbb{R}$ isometrically, $f_{\infty,t}-f_{\sol,t}$ must be constant. By the normalizations 
\begin{equation*} \widetilde{\nu}_t^i(M)=1=\frac{1}{(2\pi)^n}\int_M e^{-f_{\sol}-W} dg_{\sol}, \end{equation*} 
it follows that $f_{\infty,t}=f_{\sol,t}+W$ for all $t\in (-\infty,0)$. By the $\partial \overline{\partial}$-lemma, there are $\widetilde{\varphi}_{i}\in C^{\infty}(M\times [-|t_i|^{-1},0))$ such that $\widetilde{\omega}_{i,t} = \omega_{\sol,t}+\sqrt{-1}\partial \overline{\partial}\widetilde{\varphi}_{i,t}$ and $\int_{M_{\sol}}\widetilde{\varphi}_{i,-1} dg_{\sol}=0$. It then follows from elliptic estimates that $\widetilde{\varphi}_i \to 0$ in $C_{\operatorname{loc}}^{\infty}(M \times (-\infty,0))$, hence
\begin{equation*} \omega_{s}:=e^s \phi_{-s}^{\ast} \widetilde{\omega}_{i,-e^{-s}},\qquad \varphi_{0,s}:= e^s \phi_{-s}^{\ast} \widetilde{\varphi}_{i,-e^{-s}},
\end{equation*}
satisfy \eqref{eq:modifiedflow} and \eqref{eq:PCMA3}, respectively, $\omega_s=\omega_{\sol}+\sqrt{-1}\partial \overline{\partial} \varphi_{0,s}$, and \begin{equation*}
    \sup_{s\in [-2,k]} \| \omega_s-\omega_{\sol}\|_{C^{\lfloor \epsilon_c^{-1}\rfloor}(M,\omega_{\sol})} <\epsilon_c, \qquad \widetilde{\mathcal{D}}_{c,0}(\varphi_0,k) \leq \delta_c
\end{equation*}
for sufficiently large $i\in \mathbb{N}$,
where $\epsilon_c,\delta_c>0$ are as in Proposition \ref{prop:compactdecay} and $k\in \mathbb{N}$ satisfies $k \geq \epsilon_c^{-1}+2$.  
Repeatedly applying Proposition \ref{prop:compactdecay} as in the proof of Theorem \ref{thm:main1} yields $h_{j} \in H$ and $b_j \in \mathbb{R}$ such that 
\begin{equation*}
    \varphi_{j,s}:= \varphi_{j-1,s}-(u_{h_j}-u_{h_{j-1}})+b_je^{s-(j+k-1)} 
\end{equation*}
satisfies \eqref{eq:normalizedmodifiedflow} with $h$ replaced by $h_j$, and such that \begin{equation} \label{eq:decayincompactcase} \sup_{s\in [j+k-2,j+k]}\int_{M_{\sol}}\varphi_{j,s}^2 d\nu_{\sol} + \|h_j-h_{j-1}\|_H^2+|b_j|^2<Ce^{-\frac{a_c^2}{2}j},
\end{equation}
as well as
\begin{equation} \label{eq:coarseestimatecompactcase} \sup_{s\in [j+k-2,j+k]}\|\omega_s-\omega_{h_j}\|_{C^{\lfloor\epsilon_c^{-1}\rfloor}(M_{\sol},g_{\sol})} \leq \epsilon_c.
\end{equation}
By applying parabolic regularity and interpolation as in the proof of Theorem \ref{thm:main1}, we may moreover conclude 
\begin{equation*}
    \sup_{s\in [j+k-2,j+k]} \sup_{M_{\sol}} |\varphi_{j,s}| \leq e^{-\beta j}
\end{equation*}
for some $\beta=\beta(a_c) \in (0,\frac{a_c^2}{8})$. As in the proof of Theorem \ref{thm:main1}, we set $\mu_{m}(s):=\sum_{j=1}^m b_j e^{s-(j+k)}$ for $m\in \mathbb{N} \cup \{\infty\}$, and $h_{\infty}:= \lim_{j\to \infty}h_j$, so that 
\begin{equation} \label{eq:decayofmodscompact}
    \|h_{\infty}-h_m\|_H \leq Ce^{-\frac{a_c^2}{8}m}, \qquad |\mu_{\infty}(s)-\mu_m(s)|\leq C\sum_{j=m}^{\infty} e^{-\frac{a_c^2}{8}j} e^{s-j}.
\end{equation}
Setting $\tau_j:=-e^{-(j+k-1)}$, it follows that for all $t\in [\tau_{j-1},\tau_j]$, we have
\begin{equation} \label{eq:fastC0decaycompact}
    |\widetilde{\varphi}_{\infty,t}|\leq |t|\cdot \sup_{M_{\sol}}|u_{h_{\infty}}-u_{h_j}|+|t|\cdot |\mu_{\infty}(-\log|t|)-\mu_j(-\log|t|)| +\sup_{M_{\sol}}|\widetilde{\varphi}_{j,t}| \leq C|t|^{1+\beta}.
\end{equation}
On the other hand, \eqref{eq:coarseestimatecompactcase}, \eqref{eq:decayofmodscompact}, and Proposition \ref{prop:modelmetrics} \ref{modelmetrics3} imply that for some $T_0>0$, we have
\begin{equation} \label{eq:coarseestimatescompact2} 
\sup_{t\in [-T_0,0)}\|\widetilde{g}_{i,t}-g_{h_{\infty},t}\|_{C^0(M,g_{h_{\infty},t})} < 2\epsilon_c
\end{equation}
Because $(g_{h_{\infty},t})_{t\in (-\infty,0)}$ satisfies the Type I curvature bound \eqref{eq:TypeIcurvatureconjecturein}, we may apply Shi's derivative estimates, \eqref{eq:coarseestimatescompact2}, and \cite[Lemma A.2]{BamKlein} to obtain $C_k \in (1,\infty)$ such that
\begin{equation} \label{eq:derivativesofmetriccompact}
    \sup_{t\in [-T_0,0)} \sup_M |(\nabla^{g_{h_{\infty},t}})^k (\widetilde{g}_{i,t} -g_{h_{\infty},t})|_{g_{h_{\infty},t}}\leq C_k |t|^{-\frac{k}{2}}
\end{equation}
The claim then follows by combining \eqref{eq:fastC0decaycompact}, \eqref{eq:derivativesofmetriccompact}, Lemma \ref{lem: schauder estimates}, and Lemma \ref{lem:interpolation} as in the proof of Proposition \ref{prop:decay} \ref{prop:decayconclusion4}, where we take
\begin{equation*}
    \zeta:= \psi_i \circ (\zeta_1^{h_{\infty}})^{-1} \circ \phi_{\log \frac{1}{|t_i|}}.
\end{equation*}
\end{proof}
\begin{appendix}

 \section{Proof of Lichnerowicz's theorem} \label{section:appendixsymmetry}

We rewrite here the proof of \cite[Th\'eor\`em 1]{Lich}.

\begin{theorem}[\cite{Lich}] \label{thm:lich} If $(M,J,g)$ is a complete simply connected K\"ahler manifold and $X\in \mathfrak{X}(M)$ is a Killing field which is not real-holomorphic, then there are K\"ahler manifolds $(M',J',g'),(M'',J'',g'')$ such that $\dim_{\mathbb{C}}(M'')\geq 1$, $\Ric(g'')\equiv 0$, and there exists a holomorphic isometry 
\begin{equation*}
    (M,J,g) \cong (M'\times M'',J'+J'',g'+g'').
\end{equation*}
\end{theorem}

\begin{proof}[Proof of Theorem \ref{thm:lich}]

Because $\omega$ is $g$-parallel and the flow of $X$ preserves $\nabla^g$, it follows that $\Psi:=\mathcal{L}_{X}\omega \in \mathcal{A}^2(M)$ is parallel. For any $Y,Z\in \mathfrak{X}(M)$, we have 
\begin{equation*}
\Psi(Y,Z)=g((\mathcal{L}_{X}J)(Y),Z)=-g((\mathcal{L}_{X}J)(Z),Y).
\end{equation*}
Using
\begin{align*}
(\mathcal{L}_{X}J)(JY)=  -\mathcal{L}_{X}Y-J(\mathcal{L}_{X}J)(Y)-J^{2}\mathcal{L}_{X}Y
=  -J(\mathcal{L}_{X}J)(Y),
\end{align*}
we can thus compute
\begin{align*}
\Psi(JY,JZ)= g((\mathcal{L}_{X}J)(JY),JZ)
=  -g(J(\mathcal{L}_{X}J)(Y),JZ)
=  -g((\mathcal{L}_{X}J)(Y),Z)
=  -\Psi(Y,Z).
\end{align*}
Thus $\ker(\Psi)$ is $J$-invariant. Because $\ker(\Psi)$ is also
invariant under parallel transport, and because $M$ is simply connected,
there exist K\"ahler manifolds $(M',J',g')$ and $(M'',J'',g'')$ and a holomorphic isometry
\begin{equation*}
(M,J,g)\cong (M'\times M'',J'+J'',g'+g''),
\end{equation*}
such that $TM'=\ker(\Psi)$, and
where $\Psi$ is a nondegenerate 2-form when restricted to $TM''$. Because $X$ is not real-holomorphic, it follows that $\ker(\Psi)\neq TM$, hence $\dim_{\mathbb{C}}(M'')\geq 1$. 
Restricting $\Psi$ to $M''$ and then identifying
$\Psi$ with its $\mathbb{C}$-bilinear extension to $T_{\mathbb{C}}M''=TM''\otimes_{\mathbb{R}}\mathbb{C}$,
we write 
\begin{equation*}
\Psi=\Psi^{(2,0)}+\Psi^{(1,1)}+\Psi^{(0,2)},
\end{equation*}
where $\Psi^{(2,0)}\in\mathcal{A}^{2,0}(M'')$, $\Psi^{(0,2)}\in\mathcal{A}^{0,2}(M'')$,
and $\Psi^{(1,1)}\in\mathcal{A}^{1,1}(M'')$. Because
\begin{align*}
\Psi^{(2,0)}+\Psi^{(1,1)}+\Psi^{(0,2)}= \Psi
=  -\Psi(J''\cdot,J''\cdot) =\Psi^{(2,0)}-\Psi^{(1,1)}+\Psi^{(0,2)},
\end{align*}
we obtain $\Psi^{(1,1)}=0$. Because $\Psi$ is real, we conclude
that $\Psi^{(0,2)}=\overline{\Psi^{(2,0)}}$, hence $\Psi=2\text{Re}(\Psi^{(0,2)})$
and
\begin{align*}
\frac{1}{2}\left(\Psi-\sqrt{-1}\Psi(J''\cdot,\cdot)\right)= \frac{1}{2}(\Psi^{(2,0)}+\Psi^{(0,2)})-\frac{\sqrt{-1}}{2}\left(\sqrt{-1}\Psi^{(2,0)}-\sqrt{-1}\Psi^{(0,2)}\right)
=  \Psi^{(2,0)},
\end{align*}
hence $\Psi^{(2,0)}$ is also a nondegenerate
parallel $(2,0)$-form on $M''$. This implies that $\dim_{\mathbb{C}}(M''):=k$ is even, and that $\Theta:=(\Psi^{(2,0)})^{\frac{k}{2}}\in\mathcal{A}^{k,0}(M'')$
is a parallel holomorphic volume form, hence $(\sqrt{-1})^{k^{2}}\Theta\wedge\overline{\Theta}=\lambda(\omega'')^{k}$
for some $\lambda >0$, and $\Ric(g'') \equiv 0$.
\end{proof}

\end{appendix}
\bibliographystyle{alpha}
\bibliography{references-shrinkers}
\end{document}